\documentclass[11pt]{amsart}

\usepackage{epigamath-voisin}

\usepackage[english]{babel}

\numberwithin{equation}{section}

\usepackage{enumitem}

\setlist[enumerate,1]{label={\rm(\arabic*)}, ref={\rm\arabic*}}
\setlist[enumerate,2]{label={\rm(\alph*)}, ref={\rm\alph*}}

\usepackage{latexsym,amscd,amssymb,url}
\usepackage{extarrows}
\usepackage{verbatim}

\usepackage{dsfont}

\usepackage[all,cmtip]{xy}  
\usepackage{graphicx}
\usepackage{xcolor}

\definecolor{dblue}{rgb}{0,0,.6}

\theoremstyle{plain}
\newtheorem{theorem}{Theorem}[section]

\newtheorem{proposition}[theorem]{Proposition}
\newtheorem{lemma}[theorem]{Lemma}
\newtheorem{corollary}[theorem]{Corollary}

\theoremstyle{definition}
\newtheorem{definition}[theorem]{Definition}

\theoremstyle{remark}

\newtheorem{remark}[theorem]{Remark}

\newcommand{\del}{\partial}
\newcommand{\Z}{\mathbb Z}
\newcommand{\Q}{\mathbb Q}

\newcommand{\C}{\mathbb C}
\newcommand{\N}{\mathbb N}

\newcommand{\RR}{\operatorname{R}}

\newcommand{\CP}{\mathbb P}

\newcommand{\im}{\operatorname{im}}

\newcommand{\Hom}{\operatorname{Hom}}

\newcommand{\Div}{\operatorname{div}}

\newcommand{\id}{\operatorname{id}}

\newcommand{\Spec}{\operatorname{Spec}}

\newcommand{\pr}{\operatorname{pr}}

\newcommand{\codim}{\operatorname{codim}}

\newcommand{\rat}{\operatorname{rat}}

\newcommand{\CH}{\operatorname{CH}}
\newcommand{\supp}{\operatorname{supp}}
\newcommand{\sing}{\operatorname{sing}} 
\newcommand{\red}{\operatorname{red}}

\newcommand{\cl}{\operatorname{cl}}

 \newcommand{\sm}{\operatorname{sm}}

  \newcommand{\Grifft}{A_0}

\newcommand{\Sh}{\operatorname{Shv}}
\newcommand{\et}{\mathrm{\acute{e}t}}
\newcommand{\proet}{\mathrm{pro\acute{e}t}}
\newcommand{\Ab}{\operatorname{Ab}}

\newcommand{\Tr}{\operatorname{Tr}}

\newcommand{\colim}{\operatorname{colim}}
\newcommand{\exc}{\operatorname{exc}}
\newcommand{\Mod}{\operatorname{Mod}}

\newcommand{\an}{\operatorname{an}}
\newcommand{\restr}{\operatorname{restr}}

\newcommand{\dashedlongrightarrow}{\xymatrix@1@=15pt{\ar@{-->}[r]&}}
\renewcommand{\twoheadrightarrow}{\xymatrix@1@=15pt{\ar@{->>}[r]&}}
\newcommand{\hooklongrightarrow}{\xymatrix@1@=15pt{\ar@{^(->}[r]&}}
\newcommand{\congpf}{\xymatrix@1@=15pt{\ar[r]^-\sim&}}
\renewcommand{\cong}{\simeq}

\newcommand{\lra}{\longrightarrow}
\def\lowcong{\vbox to 0pt{\vss\hbox{$\scriptstyle\cong$}\vskip-1.5pt}}
\def\lowsim{\vbox to 0pt{\vss\hbox{$\scriptstyle\sim$}\vskip-2pt}}

\newcommand{\BM}{\mathrm{BM}}
\newcommand{\nr}{\mathrm{nr}}
\newcommand{\cons}{\mathrm{cons}}
\newcommand{\cont}{\mathrm{cont}}
\newcommand{\comp}{\mathrm{comp}}
\newcommand{\tr}{\mathrm{tr}}

\newcommand{\op}{\operatorname{op}}

\newcommand{\supth}[1]{\ensuremath{#1^{\mathrm{th}}}}

\YearArticle{2024} \EpigaArticleNr{20} \ReceivedOn{September 13, 2022}
\InFinalFormOn{21 August, 2024}
\AcceptedOn{September 22, 2024}

\title{A moving lemma for cohomology with support}
\titlemark{A moving lemma for cohomology with support}

\author{Stefan Schreieder} 
\address{Institute of Algebraic Geometry, Leibniz University Hannover, Welfengarten 1, 30167 Hannover, Germany.}
\email{schreieder@math.uni-hannover.de}

\authormark{S.~Schreieder} 

\AbstractInEnglish{For a natural class of cohomology theories with support (including \'etale or pro-\'etale cohomology with suitable coefficients), we prove a moving lemma for cohomology classes with support on smooth quasi-projective $k$-varieties that admit a smooth projective compactification (\textit{e.g.} $\operatorname{char}k=0$).  
This has the following consequences for such $k$-varieties and cohomology theories: a local and global generalization of the effacement theorem of Quillen, Bloch--Ogus, and Gabber, a finite-level version of the Gersten conjecture in characteristic zero, and a generalization of the injectivity property and the codimension $1$ purity theorem for \'etale cohomology.  
Our results imply that the new invariants from the author's 2023 Compositio paper are motivic.}

 
\acknowledgement{This project has received funding from the European Research Council (ERC) under the European Union's Horizon 2020 research and innovation programme under grant agreement No 948066 (ERC-StG RationAlgic).}

\begin{document}

%%%%%%%%%%%%%%%%%%%%%%%%%%%%%%%
% Title page
%%%%%%%%%%%%%%%%%%%%%%%%%%%%%%%

\maketitle

\begin{prelims}

\DisplayAbstractInEnglish

\bigskip

\DisplayKeyWords

\medskip

\DisplayMSCclass

\end{prelims}

%%%%%%%%%%%%%%%%%%%%%
% Table of Contents
%%%%%%%%%%%%%%%%%%%%%

\newpage

\setcounter{tocdepth}{1}

\tableofcontents

%%%%%%%%%%%%%%%%%%%%%
% Content begins here
%%%%%%%%%%%%%%%%%%%%%

\section{Introduction}

On a smooth quasi-projective variety $X$, Chow's moving lemma, see \cite{chow}, allows one to move an algebraic cycle modulo rational equivalence into good position with respect to a closed subset $S\subset X$.

The essential content of the Gersten conjecture for smooth varieties, proven for K-theory by Quillen, see \cite{quillen}, and for \'etale cohomology by Bloch--Ogus, see \cite{BO}, and Gabber, see \cite{gabber,CTHK}, is an effacement theorem, which is equivalent to a similar moving lemma for classes with support in the particular case where $X$ is affine and $S$ is a finite set of points.  For instance, in the case of \'etale cohomology, the effacement theorem is equivalent to saying that for any smooth affine $k$-variety $X$, any finite set $S\subset X$, and any class $\alpha\in H^i_Z(X)$ whose support $Z\subset X$ is nowhere dense, there is a class $\alpha'\in H^i_{Z'}(X)$ with $S\cap Z'=\emptyset$ such that $\alpha$ and $\alpha'$ have the same image in $H^i_W(X)$ for some closed $W\subset X$ with $Z,Z'\subset W$ and $\dim W=\dim Z+1$.  This is a fundamental result in algebraic geometry.  We refer to the surveys \cite{CTHK,mochizuki} for more details, applications, references, and historical remarks.

The above analogy leads naturally to the question whether the effacement theorems of Quillen, Bloch--Ogus, and Gabber are special instances of a more general moving lemma which allows one to move a class with support $Z\subset X$ to one with support $Z'$ such that $Z'$ is in good position with respect to an arbitrary given closed subset $S\subset X$.  This paper answers that question positively for a natural class of cohomology theories on smooth quasi-projective varieties in characteristic zero (or more generally, with a smooth projective compactification).  In the classical case where $X$ is affine and $\dim S=0$, this yields a new proof of the Gersten conjecture for \'etale cohomology in characteristic zero.  Our proof yields a stronger conclusion than what was known before, as it is well behaved with respect to localization and hence gives rise to new Gersten-type resolutions on finite levels.  The general case where $\dim S>0$ is new and has several applications that go beyond the original Gersten conjecture.

\subsection{Main result}  \label{subsec:higher-effacement:intro}
We fix a field $k$ and a twisted cohomology theory with support $(X,Z)\mapsto H^\ast_Z(X,n)$ for smooth equi-dimensional algebraic $k$-schemes $X$ with $Z\subset X$ closed.  We assume that some natural axioms, as outlined in Section \ref{sec:axioms} below, are satisfied.  Concrete examples include \'etale or pro-\'etale cohomology with suitable coefficients; see Proposition \ref{prop:pro-etale-coho}.

To simplify notation, we write $H^\ast(X,n):=H^\ast_X(X,n)$ and $H^\ast_Z(U,n):=H^\ast_{Z\cap U}(U,n)$ for $U\subset X$ open.  Our main result is the following.

\begin{theorem}  [Moving lemma]
\label{thm:moving-lemma-cohomology}
Let $X$ be a smooth equi-dimensional $k$-scheme that admits a smooth projective compactification.  Let $S,Z\subset X$ be closed subsets with $\dim Z<\dim X$.  Then there are closed subsets $Z'\subset W\subset X$ with $Z\subset W$, $\dim Z'=\dim Z$, and $\dim W=\dim Z+1$ such that $Z'$ and $W\setminus Z$ meet $S$ properly and for any $\alpha\in H^\ast_Z(X,n)$, there is a class $\alpha'\in H_{Z'}^\ast(X,n)$ such that $\alpha$ and $\alpha'$ have the same image in $H^\ast_W(X,n)$.
\end{theorem}
 
In Theorem \ref{thm:moving-body} below, we prove a stronger, but more technical, version of Theorem~\ref{thm:moving-lemma-cohomology}, where $S\subset X$ is replaced by a morphism $f\colon S\to X$ and the behaviour under localization on $X$ is discussed.

The main idea of our proof is to reduce the problem to a version of Chow's moving lemma due to Levine.  More precisely, we construct an action of correspondences on the cohomology class in question and note that the diagonal acts as the identity.  Moving the diagonal via Chow's moving lemma will then move our class.  This is particularly clear in the case where $X$ is smooth projective, but technical difficulties appear in the (important) case where $X$ is only an open subset of a smooth projective scheme.

Even if $X$ in Theorem \ref{thm:moving-lemma-cohomology} is affine, our proof is global and makes essential use of a smooth projective compactification.  We explain several applications in the following two subsections.

\subsection{Effacement theorems and a finite-level version of the Gersten conjecture} 
By the long exact sequence of triples, the conclusion of Theorem \ref{thm:moving-lemma-cohomology} is equivalent to saying that the natural map $H^\ast_{Z}(X,n)\to H^\ast_W(X\setminus Z',n)$ is zero.  If $\dim S+\dim Z<\dim X$, then the condition that $Z'$ meets $S$ properly simply means that $X\setminus Z'$ is a neighbourhood of $S$.  We thus get the following.

\begin{corollary}[Global effacement]
\label{cor:global-effacement}
Let $X$ be a smooth equi-dimensional $k$-scheme that admits a smooth projective compactification.  Let $S,Z\subset X$ be closed subsets with $\dim S+\dim Z<\dim X$.  Then there exist a neighbourhood $U\subset X$ of $S$ and a closed subset $W\subset X$ with $Z\subset W$ and $\dim W=\dim Z+1$ such that the following composition is zero:
$$
H^\ast_{Z }(X,n)\longrightarrow H^\ast_{W}(X,n)\longrightarrow H^\ast_{W}(U,n).
$$ 
 \end{corollary} 

The case $\dim S=0$ implies formally by specialization that the result holds for any finite set of (possibly non-closed) points $S\subset X$ and any nowhere dense closed subset $Z\subset X$.  This is the aforementioned effacement theorem of Bloch--Ogus and Gabber in our context; \textit{cf.} \cite[Theorems 2.2.1]{CTHK}.  A similar effacement theorem for K-theory had previously been proven by Quillen; see \cite[p.\ 125, Theorem 5.11]{quillen}.

Now let $X_S$ be the (Zariski) localization of $X$ along $S\subset X$, \textit{i.e.}~the pro-scheme given by the system of all open neighbourhoods $U\subset X$ of $S$ in $X$.  One defines
\begin{align*} %\label{eq:coho-of-localization}
H^\ast_Z(X_S,n):=\lim_{\substack{\longrightarrow\\S\subset U\subset X}} 
 H^\ast_{Z }(U,n) \quad \text{and}\quad H^\ast (X_S\setminus Z,n):=\lim_{\substack{\longrightarrow\\S\subset U\subset X}} 
 H^\ast (U\setminus Z,n).
\end{align*} 

\begin{corollary} 
\label{cor:tubular-neighbourhood}
Let $X$ be a smooth equi-dimensional $k$-scheme that admits a smooth projective compactification.  Let $S,Z\subset X$ be a closed subsets with $ \dim S+\dim Z<\dim X$.  Then $H^\ast_Z(X_S,n)\to H^\ast(X_S,n)$ is zero, and the long exact sequence of triples induces for all $i$ a short exact sequence
\begin{align*} %\label{eq:ses-neighbourhoods}
0\longrightarrow H^i(X_S,n)\longrightarrow H^i(X_S\setminus Z,n)\stackrel{\del}\longrightarrow H^{i+1}_Z(X_S,n)\longrightarrow 0 .
\end{align*}
\end{corollary} 

An interesting special case is when $S=Z$ with $\dim Z<\frac{1}{2}\dim X$, where we get an algebro-geometric analogue of the following fact from differential topology: if $N$ is a tubular neighbourhood of a submanifold $A$ of a real manifold $M$ of real codimension $c$, then the natural map $H^i_A(N,\Z)\to H^i(N,\Z)$ identifies with the map $H^{i-c}(A,\Z)\to H^i(A,\Z)$ given by cup product with the Euler class of the normal bundle of $A$ in $M$.  Hence the map in question is zero if the Euler class is zero.
Corollary~\ref{cor:tubular-neighbourhood} is an analogue of that result which also applies  to singular subvarieties $Z\subset X$ and to the (comparatively coarse) Zariski localization $X_S$; the triviality of the Euler class is replaced by the condition $\dim Z<\frac{1}{2}\dim X$.

Contrary to Chow's moving lemma, the subset $Z'$ in Theorem \ref{thm:moving-lemma-cohomology} cannot be chosen to be well behaved with respect to localization: if we shrink $X$, then $Z'$ typically has to be enlarged; see Remark \ref{rem:localization-Z'}.  Surprisingly, the closed subset $W$ in Theorem \ref{thm:moving-lemma-cohomology} is much better behaved.  This yields the following.

\begin{corollary} [Local effacement]
\label{cor:local-effacement}
Let $X$ be a smooth equi-dimensional $k$-scheme which admits a smooth projective compactification.  Let $Z\subset X$ be a nowhere dense closed subset.  Let $S\subset X$ be either closed with $\dim S <\codim Z-1$ or a finite set of points.  Then there is a closed subset $W\subset X$ with $Z\subset W$ and $\dim W=\dim Z+1$ such that the natural map $ H^\ast_Z(X_S,n)\to H^\ast_W(X_S,n) $ is zero.
\end{corollary}
 
Corollary~\ref{cor:local-effacement} is new even in the classical case where $S$ is a finite set of points.  In this case, $X_S=\Spec(\mathcal O_{X,S})$
is the spectrum of the semi-local ring $\mathcal O_{X,S}$, and it was known before that the above vanishing result holds if one passes to the direct limit over all such closed subsets $W$; see \cite[Theorem 4.2.3]{BO}. The fact that a single subscheme suffices is new  (and may be somewhat surprising). In fact,  Colliot-Th\'el\`ene--Hoobler--Kahn write in \cite[Remark 2.2.8]{CTHK}  the following concerning this issue:
\smallskip

``\textit{We would like to point out that $($contrary to the definition of effaceability$)$ the statement in Theorem 2.2.7 is not local: the proof by no means implies that the map of Theorem 2.2.7 (2) remains $0$ when U is replaced by a smaller open set. \emph{[\ldots]}  This shows the subtlety of the situation and probably why Gersten’s conjecture is so difficult for general regular local rings of dimension $\geq 2$.}''
\smallskip

As a consequence, we obtain the following finite-level version of the Gersten conjecture for the \'etale cohomology of varieties over fields of characteristic zero.

\begin{corollary}[The Gersten conjecture on finite levels] \label{cor:gersten}
Let $X$ be a smooth affine variety over a field $k$. Assume $X$ admits a smooth projective compactification $($e.g.\ $\operatorname{char} k=0)$.  Let $x\in X^{(c)}$ with localization $X_x=\Spec(\mathcal O_{X,x})$.  Let $ Z_c=\{x\}\subset Z_{c-1}\subset \dots \subset Z_1\subset Z_0=X_x$ be a chain of closed subsets of $X_x$ of increasing dimensions.  Up to replacing the given chain $\{Z_j\}_j $ by one that is finer $($i.e.\ by a chain $\{Z_j'\}$ as above with $Z_j\subset Z_j'$ for all $j)$, the following complex is exact for all $i$:
\begin{align*}
0\lra H^i(X_x,n)\lra  H_{\BM}^{i}(X_x\setminus Z_{1}) \stackrel{\del}\lra H_{\BM}^{i-1} (Z_1\setminus Z_{2}) \stackrel{\del}\lra \cdots \stackrel{\del}\lra H_{\BM}^{1}(Z_{i-1}\setminus Z_{i}) \stackrel{\del}\lra H_{\BM}^{0}(Z_{i}\setminus Z_{i+1}) \lra 0 ,
\end{align*} 
where
$Z_j=\emptyset$ for $j>c$,  
$$
H^{i-j}_{\BM}(Z_j\setminus Z_{j+1}):=\lim_{\substack{\longrightarrow \\ x\in  U\subset X}}H^{i+j}_{\bar Z_{j}\setminus \bar Z_{j+1}}(U\setminus  \bar Z_{j+1},n+j ), 
$$
and $\bar Z_j\subset X$ denotes the closure of\, $Z_j$ in $X$.  If moreover $T\subset X_x$ is closed and each closed subset $Z_j\subset X_x$ of the given chain $\{Z_j\}$ meets $T\setminus \{x\}$ dimensionally transversely, then the above chosen refinement $\{Z_j'\}$ still satisfies this transversality condition with respect to $T\setminus \{x\}$. 
\end{corollary}

If $H^\ast_Z(X,n)=H^\ast_Z(X_{\et},\mu_{\ell^r}^{\otimes n})$ is \'etale cohomology with support and coefficients in $\mu_{\ell^{r}}^{\otimes n}$ for some prime $\ell$ invertible in $k$, then $H^{i-j}_{\BM}(Z_j\setminus Z_{j+1})$ as defined above coincides up to a shift with Borel--Moore homology of $Z_j\setminus Z_{j+1}$ (\textit{cf.} \cite[Section~1]{BO}), whence the notation.  If moreover $Z_j\setminus Z_{j+1}$ is regular and equi-dimensional, then $H^{i-j}_{\BM}(Z_j\setminus Z_{j+1})=H^{i-j} (Z_j\setminus Z_{j+1},\mu_{\ell^r}^{\otimes n+j}) $ agrees by Gabber's purity theorem, see \cite{fujiwara-2}, with ordinary \'etale cohomology of $Z_j\setminus Z_{j+1}$.  The latter holds in particular in the limit where we run over all chains $\{Z_j\}_j $ as above.
 
Corollary~\ref{cor:gersten} says that we can compute the cohomology of the localization $X_x=\Spec(\mathcal O_{X,x})$ in terms of arbitrarily fine stratifications of $X_x$.  The original Gersten conjecture  for \'etale cohomology, proven in \cite{BO}, asserts this only in the limit: the above complex is exact if we pass to the direct limit over all chains $\{Z_j\}_j$.  The fact that exactness happens already on finite levels as well as the possibility of requiring transversality conditions with respect to a closed subset $T\subset X_x$ are new.

\subsection{Codimension j+1 purity and new motivic invariants}
For a subset $S\subset X$, we denote by $F_j^SX$ the pro-scheme given by the inverse system of all open neighbourhoods $U\subset X$ of $S$ with $\codim_X(X\setminus U)>j$.  We further set $F_jX:=F_j^{\emptyset}X$.  As before, the cohomology $H^\ast(F_j^SX,n)$ is defined as direct limit over $ H^\ast(U,n)$, where $U$ runs through the given inverse system.  Intuitively, one obtains $F_j^SX$ by successively removing from $X$ all closed subsets $Z\subset X$ with $\codim_XZ\geq j+1$ that are disjoint from $S$.  For instance, if $X$ is irreducible, $F_0X\cong \Spec(k(X))$, and if $S$ is a finite set, $F_0^S X\cong \Spec(\mathcal O_{X,S})$.
  
Important consequences of \cite{BO}, highlighted for instance in \cite[Section~3.8]{CT}, are the injectivity and codimension $1$ purity property for \'etale cohomology; see \cite[Theorems 3.8.1 and 3.8.2]{CT}.  Theorem \ref{thm:moving-lemma-cohomology} implies the following generalization.

\begin{corollary}[Injectivity and codimension j+1 purity] 
\label{cor:injectivity+purity-thm}
Let $X$ be a smooth equi-dimensional $k$-scheme that admits a smooth projective compactification, and let $S\subset X$ be closed.
\begin{enumerate}
\item  The   restriction map $H^\ast(F_j^SX,n)\to H^\ast(F_jX,n)$ is injective for $j\geq  \dim S$.
\item  
A class in $H^\ast(F_jX,n)$ that lifts to $F_{j+1}X$ also lifts to $F_{j+1}^SX$ for   $j   \geq \dim S-1$.
\end{enumerate} 
\end{corollary}

The $\dim S=0$ case of Corollary~\ref{cor:injectivity+purity-thm} implies by specialization that the results hold for $j\geq 0$ and any finite set of points $S\subset X$.  Hence, the case $j=\dim S=0$ of Corollary~\ref{cor:injectivity+purity-thm} corresponds to the injectivity property and codimension $1$ purity theorem in \'etale cohomology; see \cite[Theorems 3.8.1 and 3.8.2]{CT}.  We note that our purity result is new even for $j=0$: we get that an unramified class on the generic point of $X$ lifts to an open neighbourhood of any given closed curve in $X$, while this was previously only known for points in $X$.

As pointed out by one of the referees: it is natural to wonder if Corollary \ref{cor:injectivity+purity-thm} has an analogue for $G$-torsors, \textit{i.e.}~for the functor $H^1_{\et}(-,G)$ for suitable group schemes $G$.
 
Following \cite{Sch-refined}, the refined unramified cohomology associated to the given twisted cohomology theory is given by
$$
H^\ast_{j,\nr}(X,n):=\im(H^\ast (F_{j+1}X,n)\to H^\ast (F_jX,n)) .
$$
In other words, an element $[\alpha]\in H^\ast_{j,\nr}(X,n)$ is represented by a class $\alpha \in H^\ast(U,n)$ on some open $U\subset X$ whose complement has codimension $j+2$, and any two such representatives yield the same element in $H^\ast_{j,\nr}(X,n)$ if they coincide on some open subset $V\subset X$ whose complement has codimension $j+1$.

This is a common generalization of traditional unramified cohomology and Kato homology; see \cite[Section 1.3]{Sch-refined}.  The latter are known to be motivic as a consequence of the Gersten conjecture; see \cite{BO}.  The generalization of the Gersten conjecture proven in this paper implies that refined unramified cohomology $H^i_{j,\nr}(X,n)$ is motivic for all $i$ and $j$.

\begin{corollary}\label{cor:motivic}
Let $X$ and $Y$ be smooth projective equi-dimensional schemes over a field $k$ with $d_X=\dim X$.  For all $c,i,j\geq 0$, there is a natural bi-additive pairing
$$
\CH^c(X\times Y)\times H^i_{j,\nr}(X,n)\longrightarrow H^{i+2c-2d_X}_{j+c-d_X,\nr}(Y,n+c-d_X),\quad ([\Gamma],[\alpha])\longmapsto [\Gamma]_\ast([\alpha])
$$
which is functorial with respect to the composition of correspondences.
\end{corollary}

Corollary~\ref{cor:motivic} establishes in particular the existence of functorial pullbacks $f^\ast\colon H^i_{j,\nr}(Y,n)\to H^i_{j,\nr}(X,n)$ along morphisms $f\colon X\to Y$ between smooth projective varieties; \textit{cf.}~Corollary \ref{cor:pullbacks-for-H_nr}.  The existence of pullbacks is non-trivial already in the case where $Y=\CP^n$ and $X$ is the blow-up in a smooth subvariety.  The main problem is that a class on some open subset $U\subset Y$ can of course be pulled back to a class on $f^{-1}(U)$, but if $f$ is not flat, then the complement of $f^{-1}(U)$ in $X$ may have the wrong codimension.  The key ingredient that allows one to overcome this issue is the codimension $j$ purity property proven in Corollary \ref{cor:injectivity+purity-thm}, which allows one to represent unramified classes by classes on particular open subsets $U$ where this issue does not occur.

Despite the simple definition, the refined unramified cohomology groups of a variety $X$ are rather subtle invariants that interpolate between cohomology and cycle theory of $X$; see \cite{Sch-refined,Sch-griffiths}.  For instance, $H_{0,\nr}^i(X,n)\cong H_{\nr}^i(X,n)$ is classical unramified cohomology, while
$$
H^i_{j,\nr}(X,n)\cong H^i (X,n)\quad\text{for }  \lceil i/2\rceil\leq j;
$$
see  Lemma \ref{lem:H_nr=H} below or \cite[Corollary 5.10]{Sch-refined}.

If $ \lceil i/2\rceil >j$, then the refined unramified cohomology groups differ in general from ordinary cohomology.  In this (interesting) range we show that refined unramified cohomology satisfies the following basic properties that generalize the fact that traditional unramified cohomology is a stable birational invariant by \cite{CTO}.  To state our result, recall the decreasing filtration $F^\ast$ on $H^i_{j,\nr}(X,n)$, given by $F^mH^i_{j,\nr}(X,n):=\im(H^i(F_mX,n)\to H^i(F_jX,n))$ for $m\geq j+1$.

\begin{corollary} \label{cor:H_j-nr-basic}
Let $X$ and $Y$ be smooth projective equi-dimensional $k$-schemes, and let $i,j,n,m\geq 0$.  Then the following properties hold true:
\begin{enumerate} 
\item There is a natural isomorphism 
$$
\sum_lf_l:\bigoplus_{l=0}^{\min(j,n)}H^{i-2l}_{j-l,\nr}(Y,m-l)\stackrel{\lowcong}\lra H^i_{j,\nr}(Y \times \CP^n_k,m),
$$
where $f_l$ is the composition of the pullback along the projection $Y\times \CP^{n-l}_k\to Y$ followed by the pushforward along the inclusion $Y\times \CP^{n-l}_k\to Y\times \CP^{n}_k$ induced by some linear embedding $ \CP^{n-l}_k\subset \CP^{n}_k $.
\label{item:cor:H_j-nr-basic:2}
\item If $f\colon X\stackrel{\lowsim}\dashrightarrow Y$ is a birational map that is an isomorphism in codimension $c$, then $f$ induces for any $j\leq c$ an isomorphism $ f^\ast\colon H^i_{j,\nr}(Y,n) \stackrel{\lowsim}\to H^i_{j,\nr}(X,n)$.\label{item:cor:H_j-nr-basic:3}
\item \label{item:cor:H_j-nr-basic:4} The isomorphisms in \eqref{item:cor:H_j-nr-basic:2} and \eqref{item:cor:H_j-nr-basic:3} respect the decreasing filtration $F^\ast$ on both sides.
\end{enumerate}  
\end{corollary}

The formula stated in item~\eqref{item:cor:H_j-nr-basic:2} implies for instance 
$$
 H^i_{j,\nr}(\CP^n_k,m)\cong \bigoplus_{l=0}^{\min(j,n)} H^{i-2l}(\Spec(k),m-l) .
 $$  
This is, thanks to Corollary \ref{cor:motivic}, a consequence of the motivic decomposition of $\CP^n_k$, while a direct computation of $ H^i_{j,\nr}(\CP^n_k,m)$ appears to be a rather difficult task because $H^i(F_j\CP^n_k,m)$ is typically a huge group.

Item \eqref{item:cor:H_j-nr-basic:4} is new even for $j=0$, where it shows that the unramified cohomology groups from \cite{CTO}, that are known to be stable birational invariants, carry a filtration that is still a stable birational invariant.
 
\section{Preliminaries}

\subsection{Conventions} \label{subsec:not-preliminaries} 
 
An algebraic scheme is a separated scheme of finite type over a field; it is smooth if it is smooth over the ground field.  A variety is an integral algebraic scheme.  A closed subset of a scheme is implicitly identified with the corresponding reduced closed subscheme.  A morphism $f\colon X\to Y$ of Noetherian schemes is of pure relative dimension $d$ if for each $x\in X$, we have $d=\dim_x(X)-\dim_{f(x)}(Y)$.

If $X$ is an algebraic scheme and $Z\subset X$ is an irreducible subset, then the codimension $\codim_X(Z)$ of $Z$ in $X$ is the dimension of the local ring $\mathcal O_{X,\eta_Z}$, where $\eta_Z\in Z$ denotes the generic point of $Z$.  If $Z$ is not necessarily irreducible, then the local codimension $\codim_{X,z}(Z)$ of $Z$ in $X$ at a point $z\in Z$ is given by
$$
\codim_{X,z}(Z):=\inf_{z\in Z'\subset Z}\dim \mathcal O_{X,\eta_{Z'}}
$$
where the infimum (which is in fact a minimum because $X$ is of finite type over a field) ranges over all irreducible components $Z'$ of $Z$ that contain the point $z$.  The (global) codimension of $Z$ in $X$ is the minimum of the local codimensions at points $z\in Z$:
$$
\codim_X(Z)=\inf_{z\in Z} \codim_{X,z}(Z).
$$
This agrees with the definition that can for instance be found in \cite[Definition 5.28]{goertz-wedhorn}.  (We warn the reader that there are places in the literature where different definitions are used; \textit{e.g.}~one could replace the infimum by the supremum in the above definitions; see \textit{e.g.}~comments to \cite[\href{https://stacks.math.columbia.edu/tag/02I0}{Tag 02I0}]{stacks-project}.  However, the above definition is the one that works for our purposes.)

For an algebraic scheme $X$, we denote by $X_{(j)}$ the set of points of dimension $j$.  If $X$ is equi-dimensional, then we will also write $X^{(j)}:=X_{(d_X-j)}$, where $d_X=\dim X$.  The free abelian group generated by the closures of points $X^{(j)}$ is denoted by $Z^j(X)$.
The Chow group $\CH^j(X):=\CH_{d_X-j}(X)$ is the quotient of $Z^j(X)$ modulo rational equivalence; see \cite{fulton}.  In this paper we use the above convention only in the case where $X $ is pure-dimensional, so that points in $X^{(j)}$ have local codimension $j$ in $X$ in the above sense.

The support of a cycle $\Gamma= \sum a_i Z_i\in Z^j(X)$ is the reduced subscheme $\supp(\Gamma):=\bigcup Z_i$, where the union runs through all $i$ with $a_i\neq 0$.  If $X$ is equi-dimensional, then we say that two cycles $\Gamma_1\in Z^{j_1}(X) $ and $\Gamma_2\in Z^{j_2}(X) $ meet properly (or dimensionally transversely) if each point of $\supp(\Gamma_1)\cap \supp(\Gamma_2)$ has local codimension at least $j_1+j_2$ in $X$.  In this case, either $\supp(\Gamma_1)\cap \supp(\Gamma_2)$ is  empty, or it has pure codimension $j_1+j_2$.
 
If $\Gamma= \sum a_i Z_i\in Z^j(X)$ is a cycle and $W\subset X$ is a closed subset with $\supp(\Gamma)\subset W$, then we say that $\Gamma$ is rationally equivalent to zero on $W$ if $\Gamma$, viewed as a class in $\CH^\ast(W)$, is zero.

\subsection{Chow's moving lemma}
We will need the following version of Chow's moving lemma due to Levine; see \cite[Section I.II.3.5]{levine-2} and \cite[Theorem 2.13]{levine}.  The given references  prove a moving lemma for Bloch's cycle complex and hence for higher Chow groups (see also \cite{bloch-moving}); the version below concerns the special case of ordinary Chow groups and is deduced from \cite[Section I.II.3.5]{levine-2} and \cite[Theorem 2.13]{levine} in a straightforward way.

\begin{theorem}  \label{thm:moving}
Let $X$ be a smooth projective equi-dimensional scheme over a field $k$.  Let $S$ be a locally equi-dimensional algebraic $k$-scheme with a morphism $f\colon S\to X$. Then the following hold:
\begin{enumerate}
\item \label{item:Levine:moving:1} Any class $[\Gamma]\in \CH^c(X)$ can be represented by a cycle $\Gamma$ such that each point of $f^{-1}(\supp(\Gamma))$ has local codimension $c$ on $S$ $($\textit{i.e.} the expected codimension$)$.
\item \label{item:Levine:moving:2} Let $\Gamma\in Z^c(X)$ be such that $f^{-1}(\supp(\Gamma))$ has local codimension $c$ at each point.  If\, $\Gamma\sim_{\rat}0$, then there is a closed subscheme $W\subset X$ of codimension $c-1$ such that
\begin{itemize}
\item $f^{-1}(W)$ has, locally at each point, codimension $c-1$ on $S$;
\item $\supp \Gamma\subset W$, and $\Gamma$, viewed as a cycle on $W$, is rationally equivalent to zero on $W$.
\end{itemize}
\end{enumerate}  
\end{theorem} 

Occasionally, we will use Theorem~\ref{thm:moving} in conjunction with the following simple result.   

\begin{lemma} \label{lem:fibre-dimension}
Let $f\colon S\to X$ be a morphism between locally equi-dimensional algebraic $k$-schemes.  Let $Z\subset X$ be closed such that $f^{-1}(Z)$ is locally on $S$ of codimension at least $ \dim (f(S))+1$.  Then $f^{-1}(Z)=\emptyset$.
\end{lemma}

\begin{proof} 
It suffices to prove the result for each irreducible component of $S$; hence we may assume that $S$ is irreducible.  Towards a contradiction, assume that $f^{-1}(Z)\neq \emptyset$, or equivalently, $f(S)\cap Z\neq \emptyset$.  Let $x\in f(S)\cap Z$ be a closed point.  Since $S$ is irreducible by the above reduction step, the theorem on the fibre dimensions of morphisms between algebraic schemes shows that $\dim S\leq \dim (f(S)) +\dim f^{-1}(x)$.  Since $f^{-1}(x)\subset f^{-1}(Z)$, it follows that $f^{-1}(Z)$ is, locally on $S$, of codimension at most $\dim (f(S)) $.  This contradicts our assumptions, which concludes the proof.
\end{proof}

\section{Twisted cohomology theory with an action by cycles} \label{sec:axioms}

In this section we list several natural properties of a twisted cohomology theory with support which admits an action by algebraic cycles.  We will show that any theory that satisfies (some of) these properties satisfies the moving lemma for classes with support as in Theorem \ref{thm:moving-lemma-cohomology}.

We fix a field $k$.  A pair $(X,Z)$ of algebraic $k$-schemes is an algebraic $k$-scheme $X$ and a closed subset $Z\subset X$.  A morphism of pairs $f\colon (X,Z_X)\to (Y,Z_Y)$ is a morphism of schemes $f\colon X\to Y$ with $f^{-1}(Z_Y)\subset Z_X$.\footnote{This condition translates to the standard definition of a morphism of pairs in topology if one replaces $(X,Z)$ by $(X,X\setminus Z)$.}  The total space of a pair $(X,Z)$ is the scheme $X$.

Let $\mathcal V_k$ be the category of pairs $(X,Z)$ of algebraic $k$-schemes. We have 
\begin{align} \label{eq:cohomology-functor} 
\mathcal V_k^{\op} \longrightarrow \{\text{graded abelian groups}\},\quad (X,Z )\longmapsto H^\ast_{Z }(X,n)
\end{align}
for $n\in \Z$.  The degree $i$ part of $H^\ast _{Z }(X,n)$ is denoted by $H^i_{Z }(X,n)$.  For any morphism $f\colon (X,Z_X)\to (Y,Z_Y)$, contravariance yields functorial pullback maps $f^\ast\colon H^i_{Z_Y}(Y,n)\to H^i_{Z_X}(X,n)$ for all $i$.

For $U\subset X$ open, we write $H^\ast_Z(U,n):=H^\ast_{Z\cap U}(U,n)$.  (This is in line with the fact that $H^\ast_Z(-,n)$ is a Zariski presheaf on $X$.)  Moreover, for $Z=X$, we write $H^\ast(X,n):=H^\ast_X(X,n)$.

We will need the following natural properties: 

\begin{enumerate}[label={C\arabic*},ref={\rm C\arabic*}]
\item \label{item:excision}
(Excision) Let $f\colon U\to X$ be an open immersion of smooth algebraic $k$-schemes, and let $Z \subset X$ be closed with $Z\subset U$.  Then the natural map $f^\ast\colon H^i_{Z }(X,n)\to H^i_{Z }(U,n)$ is an isomorphism.
\item   \label{item:f_*}
(Pushforwards)
Let $f\colon X\to Y$ be a proper morphism between smooth equi-dimensional algebraic $k$-schemes.  Let $Z_Y\subset Y$ and $Z_X\subset X$ be closed subsets with $f(Z_X)\subset Z_Y$.  Then there are pushforward maps
$$
f_\ast\colon H^{i-2c}_{Z_X}(X,n-c)\longrightarrow H^{i}_{Z_Y}(Y,n), 
$$
where $c:=\dim Y-\dim X$.  These are functorial, \textit{i.e.}~satisfy $f_\ast\circ g_\ast=(f\circ g)_\ast$.

Consider the  diagrams
$$
\xymatrix{
X'\ar[r]^{g'}\ar[d]_-{f'} & X \ar[d]^-{f} \\
Y'\ar[r]_{g}& Y\rlap{,}
}
\qquad
\xymatrix{
H_{Z_{X'}}^{i-2c} (X',n-c)\ar[d]^-{f'_\ast} &  H^{i-2c}_{Z_X} (X,n-c) \ar[d]^-{f_\ast} \ar[l]_{(g')^\ast} \\
 H^{i}_{Z_{Y'}}(Y',n)&  H^{i}_{Z_{Y}} (Y,n) \ar[l]_{g^\ast}\rlap{,}
}
$$ 
where the diagram on the left is Cartesian, $X'$ and $Y'$ are smooth and equi-dimensional, $f$ is proper, $Z_X\subset X$ and $Z_Y\subset Y$ are closed with $f(Z_X)\subset Z_Y$, $Z_{Y'}=g^{-1}(Z_Y)$, and $Z_{X'}=Z_{Y'}\times_{Y}Z_{X}\subset X'$.
\begin{enumerate}[ref=\theenumi{}(\alph*)]
\item If $g$ is an open immersion, then the diagram on the right commutes.\label{item:f_*:open-immersion}
\item If $f$ and $g$ are smooth of pure relative dimensions and if $Z_Y=Y$ and $Z_X=X$, then the diagram on the right commutes. \label{item:f_*:f,g-smooth}
\end{enumerate} 
\item  \label{item:les-triple}
(Long exact sequence of triples)
Let $Z\subset W\subset X$ be closed subsets.  Then there is a long exact sequence
$$
\dots \longrightarrow H^i_Z(X,n)  \stackrel{\iota_\ast} \longrightarrow 
H^i_W(X,n) \xrightarrow{\restr} H^i_{W\setminus Z}(X\setminus Z,n)\stackrel{\del}\longrightarrow H^{i+1}_Z(X,n) \longrightarrow \cdots,
$$
where $\iota_\ast$ denotes the proper pushforward with respect to the identity on $X$ and $\restr$ denotes the canonical restriction map.  This sequence is functorial for pullbacks: for any morphism $f\colon X'\to X$ and closed subsets $Z'\subset W'\subset X'$ such that $f^{-1}(Z)\subset Z'$ and $f^{-1}(W)\subset W'$, pullback along $f$ induces a commutative ladder between the long exact sequence of the triple $(X',W',Z')$ and that of $(X,W,Z)$.
\item \label{item:action-of-cycles}
(Action of cycles) 
For a cycle $\Gamma\in Z^c(X)$ on a smooth equi-dimensional algebraic $k$-scheme $X$ and any closed subset $W\subset X$ with $\supp(\Gamma)\subset W$, there is an additive action
$$
\cl_W^X(\Gamma)\cup\colon H^i_Z(X,n)\longrightarrow H^{i+2c}_{W\cap Z}(X,n+c),\quad \alpha\longmapsto \cl_W^X(\Gamma)\cup \alpha 
$$
which is linear in $\Gamma$ and such that the following hold:
\begin{enumerate}[ref=\theenumi{}(\alph*)]
\item \label{item:action-of-cycles-rationally-trivial}
If $\Gamma$, viewed as a cycle on $W$, is rationally equivalent to zero on $W$, then the above action is zero, \textit{i.e.} $ \cl_W^X(\Gamma)\cup \alpha =0$ for all $\alpha$.
\item \label{item:action-of-cycles:projection-formula}
If $\Gamma=[W]$ is a prime cycle with smooth support $W$ and inclusion $f\colon W\hookrightarrow X$,  then
\begin{align*} % \label{eq:projection-formula}
\cl_W^X(\Gamma)\cup \alpha =f_\ast(f^\ast \alpha)\in H^{i+2c}_{W\cap Z}(X,n+c) \quad \text{for all $\alpha\in H^i_Z(X,n)$,}
\end{align*}
where $f^\ast\colon H^i_Z(X,n)\to H^i_{W\cap Z}(W,n)$ and $f_\ast\colon H^i_{W\cap Z}(W,n)\to H^{i+2c}_{W\cap Z}(X,n+c)$ denote the natural pullback and pushforward maps, respectively.
\item \label{item:action-of-cycles:enlargening-the-support}
If $Z\subset Z'\subset X$ and $W'\subset X$ are closed with $W\subset W'$, then the following diagram commutes: 
$$
\xymatrix{
H^i_Z(X,n)\ar[r]^-{\iota_\ast} \ar[d]_-{\cl_{W}^X(\Gamma)\cup }  &  \ar[d]^-{\cl_{W'}^X(\Gamma)\cup }H^i_{Z'} (X,n)
\\
H^{i+2c}_{W\cap Z}(X,n+c) \ar[r]^-{\iota_\ast} & H^{i+2c}_{W'\cap Z'}(X,n+c),
}
$$
where the $\iota_\ast$ denote the respective pushforwards with respect to the identity on $X$.
\item \label{item:action-of-cycles:compatible-with-restriction}
For any open subset $U\subset X$,  the  following diagram commutes: 
$$
\xymatrix{
H^i_Z(X,n)\ar[rrr]^-{\cl_{W}^X(\Gamma)\cup } \ar[d]_-{\restr }  &&& H^{i+2c}_{W\cap Z}(X,n+c) \ar[d]^-{\restr } \\
H^{i}_{Z }(U,n) \ar[rrr]^-{\cl_{W\cap U}^U( \Gamma|_U)\cup }  &&& H^{i+2c}_{W\cap Z }(U,n+c) ,
}
$$
where the vertical arrows are the canonical restriction maps and $\Gamma|_U\in Z^c(U)$ is the flat pullback of $\Gamma$.
\item \label{item:action-of-cycles:compatible-with-intersections}
If $\Gamma\in Z^c(X)$ and $\Gamma'\in Z^{c'}(X)$ are such that $W=\supp \Gamma$ meets $W'=\supp \Gamma'$ properly, then
$$
\cl_W^X(\Gamma)\cup  (\cl_{W'}^X(\Gamma')\cup  \alpha)=\cl^X_{W\cap W'}(\Gamma\cdot \Gamma')\cup \alpha \quad \text{for all $\alpha\in H^i_Z(X,n)$.}
$$ 
\item \label{item:action-of-cycles:projection-formula-everything}
Let $f\colon X'\to X$ be a morphism between smooth equi-dimensional algebraic $k$-schemes, and let $\alpha\in H^i_Z(X,n)$.
\begin{itemize}
\item If $f$ is flat, then
$ 
f^\ast(\cl_W^X(\Gamma)\cup  \alpha)=\cl_{f^{-1}(W)}^{X'}(f^\ast\Gamma)\cup f^\ast \alpha .
$ 
\item If $f$ is smooth and proper,   then  the following projection formulas hold true:
 $$
f_\ast(\cl_{W'}^{X'}(\Gamma)\cup f^\ast \alpha)=\cl_{f(W')}^X(f_\ast\Gamma)\cup \alpha
\quad \text{and}\quad 
f_\ast(\cl_{f^{-1}(W)}^{X'}(f^\ast \Gamma)\cup \alpha)=\cl_{W}^X(\Gamma)\cup f_\ast \alpha  .
$$
\end{itemize} 
\end{enumerate}  
\item \label{item:semi-purity}
(Semi-purity)
If $X$ is smooth and equi-dimensional, then $H^i_Z(X,n)=0$ for $i<2\codim_Z(X)$. 
\end{enumerate}

\begin{remark}
Items \ref{item:f_*:f,g-smooth}, \ref{item:action-of-cycles:compatible-with-restriction}, \ref{item:action-of-cycles:compatible-with-intersections}, \ref{item:action-of-cycles:projection-formula-everything}, and \ref{item:semi-purity} are  not needed to prove the moving lemma (see Theorems \ref{thm:moving-lemma-cohomology} and \ref{thm:moving-body}) and its immediate applications (see Corollaries \ref{cor:global-effacement}, \ref{cor:local-effacement}, \ref{cor:gersten}, and \ref{cor:injectivity+purity-thm}). 
We list these properties because we will need  
 items  \ref{item:f_*:f,g-smooth} and \ref{item:action-of-cycles:compatible-with-restriction}  to get a natural action on refined unramified cohomology, \textit{cf.} Corollaries \ref{cor:motivic} and \ref{cor:motivic-body},  while  \ref{item:action-of-cycles:compatible-with-intersections} and \ref{item:action-of-cycles:projection-formula-everything} are needed to guarantee that this action is functorial with respect to the composition of correspondences. 
 Finally, semi-purity (\textit{i.e.}~\ref{item:semi-purity}) is only needed to guarantee some normalizations in Corollaries \ref{cor:gersten} and \ref{cor:H_j-nr-basic}.
\end{remark}

\subsection{Examples} \label{subsec:example}

\begin{proposition} \label{prop:pro-etale-coho}
Let $k$ be a field, let $\ell$ be a prime invertible in $k$, and denote by $\pi_X\colon X\to \Spec(k)$ the structure map of a $k$-scheme $X$.
\begin{enumerate}
\item \label{item:prop:etale-coho} Let $F$ be an $\ell^\infty$-torsion \'etale sheaf on $\Spec(k)$.  The \'etale cohomology functor
$$
(X,Z)\longmapsto H^\ast_Z(X,n):= H^\ast_{Z}(X_\et,  \mathcal F (n))
$$
satisfies all conditions \ref{item:excision}--\ref{item:semi-purity} above,  where $\mathcal F(n)=\colim_r  (\pi_X^\ast F\otimes_\Z \mu_{\ell^r}^{\otimes n})$.  
\item \label{item:prop:continuous-etale-coho} Let $F=(F_r)_r$ be an inverse system of  \'etale sheaves on $\Spec(k)$  such that $F_r$ is $\ell^r$-torsion and the transition maps $F_{r+1}\to F_r$ are surjective. 
The continuous \'etale cohomology functor
$$
(X,Z)\longmapsto H^\ast_Z(X,n):= H^\ast_{Z,\cont}(X_\et,  \mathcal F (n))
$$
satisfies all conditions \ref{item:excision}--\ref{item:semi-purity} above, where $\mathcal F (n)=(\pi_X^\ast F_r   \otimes_{\Z/\ell^r} \mu_{\ell^r}^{\otimes n})_r$.  
\item  \label{item:prop:pro-etale-coho} Let $K\in D_{\cons}((\Spec(k))_{\proet},\widehat {\Z}_\ell)$ be a constructible complex of $\widehat {\Z}_\ell$-modules on the pro-\'etale site $(\Spec(k))_\proet$. 
Then the pro-\'etale $($hyper-$)$cohomology functor  
$$
(X,Z)\longmapsto H^\ast_Z(X,n):= \RR^\ast \Gamma_{Z}(X_\proet, \mathcal K(n))
$$
satisfies all conditions \ref{item:excision}--\ref{item:action-of-cycles} above, where $\mathcal K(n)=(\pi_{X})_{\comp}^\ast K \otimes _{\widehat \Z_\ell} \widehat \Z_{\ell}(n)$ is the $\supth{n}$ Tate twist of the completed pullback of $K$; \textit{cf.} Appendix \ref{app:A}. 
If the complex $K$ is concentrated in non-negative degrees, then condition \ref{item:semi-purity} holds true as well. 
\end{enumerate} 
\end{proposition}

Proposition~\ref{prop:pro-etale-coho} is certainly well known to experts.  For the convenience of the reader, we include a detailed proof in the appendix of this paper.
We limit ourselves here to sketching the argument briefly.

Using the compatibility of \'etale cohomology with direct limits in the coefficients from \cite[\href{https://stacks.math.columbia.edu/tag/09YQ}{Tag 09YQ}]{stacks-project}, one reduces item \eqref{item:prop:etale-coho} to the case where $\mathcal F$ is $\ell^r$-torsion for some fixed $r$.  Hence, \cite[Equation~(3.1)]{jannsen} implies  that \eqref{item:prop:continuous-etale-coho} $\Rightarrow$ \eqref{item:prop:etale-coho}.  Moreover, the arguments in \cite[Section~5.6]{BS} imply that \eqref{item:prop:pro-etale-coho} $\Rightarrow$ \eqref{item:prop:continuous-etale-coho}, and so Proposition~\ref{prop:pro-etale-coho}\eqref{item:prop:pro-etale-coho} implies the rest.  Conversely, the way the six-functor formalism of Bhatt--Scholze from \cite[Section~6.7]{BS} works allows one to essentially reduce item \eqref{item:prop:pro-etale-coho} to the case of \'etale cohomology.  The proof of Proposition~\ref{prop:pro-etale-coho} then goes roughly as follows: conditions \ref{item:excision} and \ref{item:les-triple} are straightforward, the pushforwards in condition \ref{item:f_*} is a consequence of Poincar\'e duality (\textit{cf.} \cite[Theorem~XVIII.3.2.5]{SGA4.3}) and the action in condition \ref{item:action-of-cycles} is given by cup products with a suitable cycle class (\textit{cf.} \cite[p.\ 129]{SGA4.5}).  Finally, semi-purity reduces by topological invariance to the case where $k$ is perfect, and so the result may be deduced from Poincar\'e duality.
 
\begin{remark} %\label{prop:analytic}
Proposition \ref{prop:pro-etale-coho} has an analogue for singular cohomology that applies to algebraic schemes over $k=\C$.  In fact, if $k=\C$ and $A$ is an abelian group, then the sheaf cohomology functor
 $$
(X,Z)\longmapsto H^\ast_Z(X,n):= H^\ast_{Z}(X_{\an},  \underline A_X)
$$
satisfies all conditions \ref{item:excision}--\ref{item:semi-purity} above, where $X_{\an}$ denotes the analytic space that underlies $X$ and $\underline A_X$ denotes the constant sheaf on $X_{\an}$ associated to $A$.
The proof of this statement  follows the same lines as that of Proposition \ref{prop:pro-etale-coho}, and we leave it to the reader.
 \end{remark}

\section{Action of cycles on open varieties} \label{sec:action}
 In this section we fix a field $k$ and a twisted cohomology theory as in \eqref{eq:cohomology-functor} which satisfies  conditions \ref{item:excision}--\ref{item:action-of-cycles} from Section \ref{sec:axioms}.  Moreover, $X$ and $Y$ denote smooth projective equi-dimensional algebraic schemes over $k$, and we set $d_X:=\dim (X)$.  We will denote the natural projections by $p\colon X\times Y\to X$ and $q\colon X\times Y\to Y$, respectively.

Since $X$ and $Y$ are assumed to be smooth projective and equi-dimensional, any algebraic cycle $\Gamma\in Z^c(X\times Y)$ gives rise to an action 
$$
\Gamma_\ast\colon H^i(X,n)\lra H^{i+2c-2d_X}(Y,n+c-d_X), \quad \alpha\longmapsto q_\ast(\cl_{X\times Y}^{X\times Y}(\Gamma)\cup p^\ast \alpha)
$$
which by condition~\ref{item:action-of-cycles-rationally-trivial} depends only on the rational equivalence class of $\Gamma$.
Roughly speaking, the purpose of this section is  to generalize this result to smooth quasi-projective varieties that are not necessarily projective but which admit a smooth projective compactification.
In other words, we aim to construct a similar action between the cohomology of suitable open subsets $U\subset X$ and $U'\subset Y$. 
Moreover, instead of ordinary cohomology groups, we will consider the more general situation of cohomology with support.  
A closely related discussion of the action of correspondences on Chow groups with support can be found in \cite[Section 1]{ruelling}.

\begin{lemma}\label{lem:Gamma-ast}
Let $\Gamma\in Z^{c}(X\times Y)$ be a cycle, and let $W\subset X\times Y$ be a closed subset that contains the support of\, $\Gamma$: $\supp(\Gamma)\subset W$.
Let $R,Z\subset X$ be closed, and put  $W_R:= W\cap (R\times Y)$ and $W_Z:= W\cap (Z\times Y)$.  
Let $R',Z'\subset Y$ be closed with $q(W_R)\subset R'$ and $q(W_Z)\subset Z'$.
Finally, let $U:=X\setminus R$ and $U':=Y\setminus R'$ be the complements of $R$ and $R'$, respectively.\footnote{The closed subsets $R\subset X$ and $R'\subset Y$ need not be of codimension $1$, but play the role of a ``divisor at $\infty$'', while the subsets $Z,Z'$ will be the supports of the respective cohomology classes.}

Then the composition in  \eqref{eq:big-composition-Gamma} below defines an additive action 
$$
\Gamma(W)_\ast\colon H^i_{Z }(U,n)\longrightarrow H^{i+2c-2d_X}_{Z' }(U', n+c-d_X ) .
$$
If\, $\Gamma=\Gamma_1+\Gamma_2$ with $\supp(\Gamma_i)\subset W$ for $i=1,2$, then $\Gamma(W)_\ast=\Gamma_1(W)_\ast+\Gamma_2(W)_\ast$.
\end{lemma}

\begin{proof}
  We denote the flat pullback of $\Gamma$ to $U\times Y$ by $\Gamma|_{U\times Y}$ and define $\Gamma(W)_\ast$ by the following composition: 
\begin{align} \label{eq:big-composition-Gamma}
\xymatrix{
 H^i_{Z }(U,n) \ar[d]^{p^\ast}\\
 H^i_{Z\times Y}(U\times Y,n)\ar[d]^{\cl^{U\times Y}_{W\setminus W_R}(\Gamma|_{U\times Y}) \cup  } \\
H^{i+2c}_{W_Z }(U\times Y,n+c)
\ar[d]^{\exc}_{\cong}
\\
H^{i+2c}_{W_Z }((X\times Y)\setminus W_{R}, n+c ) \ar[d]^{\restr} \\ 
H^{i+2c}_{ W_Z } (X\times U' , n+c) \ar[d]^{q_\ast} \\
 H^{i+2c-2d_X}_{Z'}(U',  n+c-d_X ),
}
\end{align}
where $p^\ast$ denotes the pullback map with respect to the projection $U\times Y\to U$, $ \cl^{U\times Y}_{W\setminus W_R}(\Gamma|_{U\times Y})\cup$ denotes the action from condition~\ref{item:action-of-cycles}, $\restr$ denotes the restriction map given by pullback, and $q_\ast$ denotes the proper pushforward from condition~\ref{item:f_*}.  Moreover, $\exc$ is constructed as follows.  The closed subset $W_Z\subset X\times Y$ intersects the complement of $U\times Y\subset X\times Y$ in $W_R$ (because $R=X\setminus U$), and so
$$
 W_Z\cap (( X\times Y)\setminus W_R)=W_Z\cap (U\times Y)  \quad \text{and}\quad U\times Y\subset ( X\times Y)\setminus W_R .
$$ 
It follows that the natural pullback map 
$$
H^{i+2c}_{W_Z }((X\times Y)\setminus W_{R}, n+c )\longrightarrow H^{i+2c}_{W_Z }(U\times Y,n+c)
$$
 is an isomorphism by  excision (see condition~\ref{item:excision}).
The map 
$$
\exc\colon H^{i+2c}_{W_Z }(U\times Y,n+c)\stackrel{\lowcong}\longrightarrow H^{i+2c}_{W_Z }((X\times Y)\setminus W_{R}, n+c )
$$
is then the inverse of the above isomorphism. 

The composition in \eqref{eq:big-composition-Gamma} yields an additive action because each map is a group homomorphism.
Additivity in $\Gamma$ follows from the fact that the action of $\Gamma$ from condition~\ref{item:action-of-cycles} is additive in $\Gamma$. 
This concludes the lemma.
\end{proof}

\begin{lemma}\label{lem:Gamma-ast:Gamma-sim-0}
Assume in the notation of Lemma \ref{lem:Gamma-ast} that  $\Gamma$, viewed as a cycle on $W$, is rationally equivalent to zero on $W$.
Then the action $\Gamma(W)_\ast$ 
from Lemma \ref{lem:Gamma-ast} is zero.
\end{lemma}

\begin{proof}
The assumption implies that $\Gamma|_{U\times Y}$ is rationally equivalent to zero on $W\cap (U\times Y)=W\setminus W_R$.
The result thus follows from condition \ref{item:action-of-cycles-rationally-trivial}, which implies that the second morphism in the composition \eqref{eq:big-composition-Gamma} is zero.
\end{proof}

\begin{lemma} \label{lem:Gamma-ast-compatible-W&U}
Let\, $\Gamma\in Z^{c}(X\times Y)$ be a cycle, and let $W_i\subset X\times Y$ for $i=1,2$ be closed subsets with $\supp(\Gamma)\subset W_i$ for $i=1,2$. 
Let $Z_1\subset Z_2\subset X$ and $Z'_1\subset Z'_2\subset X$ be closed with $q(W_i\cap (Z_i\times Y))\subset Z'_i$ for $i=1,2$.
Let $R_1\subset R_2\subset X$ and  $R'_1\subset R'_2\subset Y$ be closed with $ q(W_i\cap (R_i\times Y))\subset R_i'$ and with complements  $U_i:=X\setminus R_i$ and  $U'_i:=Y\setminus R'_i$ for $i=1,2$. 

Then the following diagram commutes: 
{\small{
\begin{align} \label{eq:diag:lem:Gamma-ast-compatible-W&U}
\xymatrix{
 H^i_{Z_1} (U_1,n)\ar[d]^-{\Gamma(W_1)_\ast}\ar[r]^-{\restr} & 
 H^i_{Z_1} (U_2,n)\ar[r]^{\iota_\ast} \ar[d]^-{\Gamma(W_1)_\ast} &
  H^i_{Z_2} (U_2,n)\ar[d]^-{\Gamma(W_2)_\ast} 
 \\
 H^{i+2c-2d_X}_{Z_1'} (U'_1,n+c-d_X)\ar[r]^-{\restr} 
 & H^{i+2c-2d_X}_{Z_1'} (U'_2,n+c-d_X) \ar[r]^{\iota_\ast}
   &H^{i+2c-2d_X}_{Z_2'} (U'_2,n+c-d_X),
}
\end{align}
}}
where the vertical arrows are the respective actions from Lemma \ref{lem:Gamma-ast}.  
\end{lemma}

\begin{proof} 
We first assume that $W_1\subset W_2$.
The action from Lemma  \ref{lem:Gamma-ast} is defined via the diagram in \eqref{eq:big-composition-Gamma}.
Since  $W_1\subset W_2$,  the diagram in question may therefore be expanded to the  big diagram
{\small{
$$
\xymatrix{
 H^i_{Z_1 }(U_1,n) \ar[d]^{p^\ast} \ar[r]^{\restr}& H^i_{Z_1 }(U_2,n) \ar[d]^{p^\ast} \ar[r]^{\iota_\ast}& H^i_{Z_2}(U_2,n) \ar[d]^{p^\ast} 
 \\
 H^i_{Z_1\times Y}(U_1\times Y,n)\ar[d]^{\cl^{U_1\times Y}_{W_1\setminus {W_1}_{R_1}}(\Gamma|_{U_1\times Y}) \cup  } \ar[r]^{\restr}&
 H^i_{Z_1\times Y}(U_2\times Y,n)\ar[d]^{\cl^{U_2\times Y}_{W_1\setminus {W_1}_{R_2}}(\Gamma|_{U_2\times Y}) \cup  } \ar[r]^{\iota_\ast}&
  H^i_{Z_2\times Y}(U_2\times Y,n)\ar[d]^{\cl^{U_2\times Y}_{W_2\setminus {W_2}_{R_2}}(\Gamma|_{U_2\times Y}) \cup  }  
 \\
H^{i+2c}_{{W_1}_{Z_1} }(U_1\times Y,n+c)
\ar[d]^{\exc}_{\cong} \ar[r]^{\restr} &
H^{i+2c}_{{W_1}_{Z_1} }(U_2\times Y,n+c)
\ar[d]^{\exc}_{\cong} \ar[r]^{\iota_\ast} &
H^{i+2c}_{{W_2}_{Z_2} }(U_2\times Y,n+c)
\ar[d]^{\exc}_{\cong} 
\\
H^{i+2c}_{{W_1}_{Z_1}}((X\times Y)\setminus {W_1}_{R_1}, n+c ) \ar[d]^{\restr} \ar[r]^{\restr}&
H^{i+2c}_{{W_1}_{Z_1}}((X\times Y)\setminus {W_1}_{R_2}, n+c ) \ar[d]^{\restr} \ar[r]^{\xi}&
H^{i+2c}_{{W_2}_{Z_2}}((X\times Y)\setminus {W_2}_{R_2}, n+c ) \ar[d]^{\restr} 
\\ 
H^{i+2c}_{ {W_1}_{Z_1}} (X\times U_1' , n+c) \ar[d]^{q_\ast} \ar[r]^{\restr} &
H^{i+2c}_{ {W_1}_{Z_1}} (X\times U_2' , n+c) \ar[d]^{q_\ast} \ar[r]^{\iota_\ast} &
H^{i+2c}_{ {W_2}_{Z_2}} (X\times U_2' , n+c) \ar[d]^{q_\ast}
\\
 H^{i+2c-2d_X}_{Z_1'}(U_1',  n+c-d_X ) \ar[r]^{\restr}&
 H^{i+2c-2d_X}_{Z_1'}(U_2',  n+c-d_X ) \ar[r]^{\iota_\ast} &
  H^{i+2c-2d_X}_{Z_2'}(U_2',  n+c-d_X ) ,
}
$$
}}
\noindent
where $\xi$ is given by the  composition
$$
H^{i+2c}_{{W_1}_{Z_1}}((X\times Y)\setminus {W_1}_{R_2}, n+c )\xrightarrow{\restr} H^{i+2c}_{{W_1}_{Z_1}}((X\times Y)\setminus {W_2}_{R_2}, n+c ) 
\stackrel{\iota_\ast}\longrightarrow H^{i+2c}_{{W_2}_{Z_2}}((X\times Y)\setminus {W_2}_{R_2}, n+c ). 
$$
We first show  that each square on the left is commutative.
This follows for the  first, third, and fourth squares from the functoriality of pullbacks. 
 The second square on the left  is commutative because of condition~\ref{item:action-of-cycles:compatible-with-restriction}, while the last square  on the left column is commutative because of condition~\ref{item:f_*:open-immersion}.

Next, we show that the right squares of the above diagram are commutative.
This follows for the first square from the compatibility of proper pushforwards and pullbacks via open immersions (see condition~\ref{item:f_*:open-immersion}), and for the third and fourth squares by the same compatibility together with the functoriality of pullbacks.
The commutativity of the second square on the right  follows from condition~\ref{item:action-of-cycles:enlargening-the-support}, while the commutativity of the last square on the right follows from the functoriality of pushforwards; see condition~\ref{item:f_*}.
This concludes the proof in the case where $W_1\subset W_2$. 
The general case follows from this by replacing $W_2$ by $W_1\cup W_2$ and noting that this does not require to change $R'_2$ because $R'_1\subset R'_2$ holds by assumption.
\end{proof}

\begin{lemma} \label{lem:Gamma-ast:Gamma=Delta}
Assume in the notation of Lemma \ref{lem:Gamma-ast} that $X=Y$, $c=\dim X$, and $\Gamma=\Delta_X$ is the diagonal with $\Delta_X\subset W$.
Then $U'\subset U$ and $Z\subset Z'$, and the action
$$
\Gamma(W)_\ast\colon H^i_{Z }(U,n)\longrightarrow H^{i }_{Z' }(U', n  )
$$ 
from Lemma \ref{lem:Gamma-ast} identifies with the natural composition
$$
H^i_{Z }(U,n)\stackrel{\iota_\ast}\longrightarrow H^{i }_{Z' }(U, n  )\xrightarrow{\restr} H^{i }_{Z' }(U', n) .
$$ 
\end{lemma}
\begin{proof}
First assume  that $W=\supp \Delta_X$.
Then the result is a straightforward consequence of condition~\ref{item:action-of-cycles:projection-formula}  together with the compatibility of pushforwards and restrictions to open subsets (see condition~\ref{item:f_*:open-immersion}) and the functoriality of pushforwards (see condition~\ref{item:f_*}).
The general case follows from this and Lemma \ref{lem:Gamma-ast-compatible-W&U}.
\end{proof}

\begin{remark} \label{rem:H_emptyset=0}
  Condition~\ref{item:les-triple} together with the functoriality of pullbacks and pushforwards implies $H^i_\emptyset(X,n)=0$ for any smooth equi-dimensional $k$-scheme $X$.
Indeed, condition~\ref{item:les-triple} applied to $Z=W=\emptyset$ shows that the sequence
$$
H^i_\emptyset(X,n)\stackrel{\iota_\ast}\longrightarrow H^i_\emptyset(X,n) \xrightarrow{\restr} H^i_\emptyset(X,n)
$$
is exact, while both arrows are isomorphisms by the functoriality of pullbacks and pushforwards, respectively. 
\end{remark}

\begin{lemma} \label{lem:Gamma-ast-zero-if-wrong-support}
Assume in the notation of Lemma \ref{lem:Gamma-ast} that 
$  
(\supp \Gamma )\cap (Z \times Y)\subset X\times R'.
$ 
Then $\Gamma(W)_\ast$ from Lemma \ref{lem:Gamma-ast}  is zero.
\end{lemma}
\begin{proof}
The assumption $(\supp \Gamma) \cap (Z \times Y)\subset X\times R'$ implies by Lemma \ref{lem:Gamma-ast-compatible-W&U} that $\Gamma(W)_\ast$ from Lemma \ref{lem:Gamma-ast} factors through the group
$
H^{i+2c-2d_X}_{\emptyset}(U', n+c-d_X )
$, which vanishes by Remark \ref{rem:H_emptyset=0}.
This concludes the proof. 
\end{proof}

The proof of the following functoriality property of the action from Lemma \ref{lem:Gamma-ast} is slightly tedious.
The result is only needed to see that the action of cycles on refined unramified cohomology that we construct in this paper is functorial with respect to the composition of correspondences.
In particular, the result is not needed in the proof of the moving lemma for classes with support (see Theorem \ref{thm:moving-lemma-cohomology} and \ref{thm:moving-body}) and its immediate consequences.
Readers who are mainly interested in the latter may therefore skip the next result and move directly to Section \ref{sec:moving} below. 

In the following, we let $p_i\colon X_1\times X_2\times X_3 \to X_i $,  $p_{ij}\colon X_1\times X_2\times X_3 \to X_i\times X_j$,  $p_{i}^{ij}\colon  X_i\times X_j\to X_i$, and $p_{j}^{ij}\colon X_i\times X_j\to X_j$ be the natural projections.

\begin{proposition} \label{prop:action:composition-of-correspondences}
Let $X_i$ for $i=1,2,3$ be smooth projective equi-dimensional algebraic $k$-schemes of dimensions $d_{X_i}=\dim X_i$.
Let $\Gamma_1\in Z^{c_1}(X_1\times X_2)$ and $\Gamma_2\in Z^{c_2}(X_2\times X_3)$ be cycles, and let   $W_i:=\supp \Gamma_i $ for $i=1,2$.
Let $W_{12}:=  (W_1\times X_3) \cap (X_1\times W_2) $, and assume that
 $W_{12}$ has codimension at least $c_1+c_2$.
 Let $W_3:=p_{13}(W_{12})\subset X_1\times X_3$, which is of codimension at least $c_3:=c_1+c_2-d_{X_2}$.
Consider the cycle 
$$
\Gamma_3:= (p_{13})_\ast(p_{12}^\ast \Gamma_1 \cdot p_{23}^\ast \Gamma_2)\in Z^{c_3}(X_1\times X_3) 
$$  
with $\supp \Gamma_3\subset W_3$.
$($This is well defined on the level of cycles because the intersection is dimensional transverse as $W_{12}$ has codimension at least $c_1+c_2$.$)$
Let $U_1\subset X_1$ be an open subset with complement $R_1$. 
Let
\begin{align} \label{eq:R3}
R_2:=p_2^{12}(W_1\cap (R_1\times X_2))\subset X_2, \quad 
R_3:=
p_3(W_{12}\cap (R_1\times X_2\times X_3))  \subset X_3,
\end{align} 
and put $U_i:=X_i\setminus R_i$ for $i=2,3$.
Then the following diagram commutes:
$$
\xymatrix{
H^i(U_1,n)\ar[d]_{\Gamma_1(W_{1})_\ast }  \ar[rrd]^-{\Gamma_3(W_3)_\ast}     & &  \\
  H^{i+2c_1-2d_{X_1}}  (U_2,n+c_1-d_{X_1})          \ar[rr]^{\Gamma_2(W_{2})_\ast}      & &   H^{i+2c_3-2d_{X_1}}  (U_3,n+c_3-d_{X_1})    .           
}
$$
\end{proposition}

Before we turn to the proof of the proposition, we need the following result on $R_3$ from \eqref{eq:R3}.

\begin{lemma} \label{lem:R3}
In the notation of Proposition \ref{prop:action:composition-of-correspondences}, we have
$$
R_3=
p_3^{13}(W_3\cap (R_1\times X_3))=
 p_3^{23}(W_2\cap(R_2\times X_3)) .
$$
\end{lemma}
\begin{proof}
We have
\begin{align*}
p_3^{13}(W_3\cap (R_1\times X_3)) &=p_3^{13}(p_{13}(W_{12})\cap (R_1\times X_3)) \\
&=p_3^{13}(p_{13}(W_{12} \cap (R_1\times X_2\times X_3))) \\
&=p_3( W_{12} \cap (R_1\times X_2\times X_3)) ,
\end{align*}
where the second equality follows from the projection formula.
This proves the first equality claimed in the lemma.

Since $R_2=p_2(W_1\cap (R_1\times X_2))$, we have
$$
R_2\times X_3=p_2(W_1\cap (R_1\times X_2))\times X_3=p_{23}(p_{12}^{-1}(W_1)\cap (R_1\times X_2\times X_3))
$$
and hence
$$
W_2\cap(R_2\times X_3)=W_2\cap p_{23}(p_{12}^{-1}(W_1)\cap (R_1\times X_2\times X_3)).
$$
The projection formula with respect to $p_{23}$ then gives
\begin{align} \label{eq:R3-1}
W_2\cap(R_2\times X_3)=p_{23}(p_{23}^{-1}(W_2)\cap p_{12}^{-1}(W_1)\cap (R_1\times X_2\times X_3)).
\end{align}
Since $p_{23}^{-1}(W_2)=X_1\times W_2$ and $p_{12}^{-1}(W_1)=W_1\times X_3$, we conclude
$$
p_3^{23}(W_2\cap(R_2\times X_3))=p_3(W_{12}\cap (R_1\times X_2\times X_3))=R_3 ,
$$
which proves the second equality in the lemma.
This concludes the proof. 
\end{proof}

\begin{proof}[Proof of Proposition \ref{prop:action:composition-of-correspondences}]
For $\alpha\in H^i(U_1,n)$, by the construction in Lemma \ref{lem:Gamma-ast}, we have 
\begin{align} \label{eq:lem:action:composition-of-correspondences:1}
\Gamma_3(W_3)_\ast(\alpha) &= (p^{13}_{3})_\ast \left( \exc_1 \left( \cl^{U_1\times X_3}_{W_3\setminus (W_3)_{R_1}}( (p_{13})_\ast (p_{12}^\ast \Gamma_1 \cdot p_{23}^\ast \Gamma_2 ) )\cup p_{1}^\ast \alpha \right)|_{X_1\times U_3} \right) ,
\end{align}
where 
$$
\exc_1\colon H^\ast_{W_3\setminus (W_3)_{R_1}}(U_1\times X_3,-)\stackrel{\lowcong} \longrightarrow H^\ast_{W_3\setminus (W_3)_{R_1}}( (X_1\times X_3)\setminus (W_3)_{R_1},-)
$$ 
denotes the   isomorphism  given by excision and $p^{13}_3\colon X_1\times U_3\to U_3$ denotes the projection.
Here
 $(W_3)_{R_1}=W_3\cap (R_1\times X_3)$, and 
so
$p_3^{13}((W_3)_{R_1})= R_3$ by Lemma \ref{lem:R3}. 

Since $W_3=p_{13}(W_{12})$, the projection formula in condition~\ref{item:action-of-cycles:projection-formula-everything} yields
\begin{align} \label{eq:lem:action:composition-of-correspondences:2}
\cl^{U_1\times X_3}_{W_3\setminus (W_3)_{R_1}}( (p_{13})_\ast (p_{12}^\ast \Gamma_1 \cdot p_{23}^\ast \Gamma_2 ) )\cup p_{1}^\ast \alpha  &=   (p_{13})_\ast \left( \cl^{U_1\times X_2\times X_3}_{W_{12}\setminus (W_{12})_{R_1}} (p_{12}^\ast \Gamma_1 \cdot p_{23}^\ast \Gamma_2 ) \cup p_{13}^\ast ( p_{1}^\ast \alpha ) \right)  ,
\end{align}
where $p_{1}^\ast \alpha$ denotes the pullback of $\alpha$ via the projection $U_1\times X_3\to U_1$.
Note that $p_{12}^\ast \Gamma_1 \cdot p_{23}^\ast \Gamma_2= p_{23}^\ast \Gamma_2 \cdot p_{12}^\ast \Gamma_1$ by the commutativity of the intersection product of cycles.
Since $W_{12}=  (W_1\times X_3) \cap (X_1\times W_2)$ has codimension at least $c_1+c_2$, the compatibility of the action of cycles with the  intersection product (see condition~\ref{item:action-of-cycles:compatible-with-intersections}) thus shows
\begin{align*}
 \cl^{U_1\times X_2\times X_3}_{W_{12}\setminus (W_{12})_{R_1}} (p_{12}^\ast \Gamma_1 \cdot p_{23}^\ast \Gamma_2 ) \cup p_{13}^\ast ( p_{1}^\ast \alpha ) =  
 \cl^{U_1\times X_2\times X_3}_{U_1\times W_{2} } (p_{23}^\ast \Gamma_2) \cup \left( \cl^{U_1\times X_2\times X_3}_{(W_{1}\setminus (W_{1})_{R_1})\times X_3 }(  p_{12}^\ast \Gamma_1 )\cup p_{13}^\ast ( p_{1}^\ast \alpha )  \right) .
\end{align*}
Note that $p_{13}^\ast ( p_{1}^\ast \alpha ) =p_{12}^\ast {p'}_{1}^\ast \alpha $, where ${p'}_{1}^\ast \alpha $ denotes the pullback of $\alpha$ via the projection $ U_1\times X_2\to U_1$.
 Using the compatibility with pullbacks from condition~\ref{item:action-of-cycles:projection-formula-everything}, we thus get
\begin{align*}
 \cl^{U_1\times X_2\times X_3}_{W_{12}\setminus (W_{12})_{R_1}} (p_{12}^\ast \Gamma_1 \cdot p_{23}^\ast \Gamma_2 ) \cup p_{13}^\ast ( p_{1}^\ast \alpha ) =  \cl^{U_1\times X_2\times X_3}_{U_1\times W_{2} } (p_{23}^\ast \Gamma_2) \cup  p_{12}^\ast \left( \cl^{U_1\times X_2}_{W_{1}\setminus (W_{1})_{R_1}}(   \Gamma_1 )\cup ( {p'}_{1}^\ast \alpha )  \right) .
\end{align*} 
Substituting this into \eqref{eq:lem:action:composition-of-correspondences:2} and plugging the result into  \eqref{eq:lem:action:composition-of-correspondences:1}, we get
\begin{align*}
\Gamma_3(W_3)_\ast(\alpha) &= (p^{13}_{3})_\ast \left( \exc_1\left( (p_{13})_\ast \left(\cl^{U_1\times X_2\times X_3}_{U_1\times W_{2} } (p_{23}^\ast \Gamma_2) \cup  p_{12}^\ast \left( \cl^{U_1\times X_2}_{W_{1}\setminus (W_{1})_{R_1}}(   \Gamma_1 )\cup ( {p'}_{1}^\ast \alpha ) \right) \right) \right)|_{X_1\times U_3} \right) . 
\end{align*}
The compatibility of proper pushforward with pullbacks via open immersions (see condition~\ref{item:f_*:open-immersion}) and the fact that $\exc_1$ is the inverse of a pullback via an open immersion show that 
\begin{align*}
\Gamma_3(W_3)_\ast(\alpha) &= (p^{13}_{3})_\ast \left( (p_{13})_\ast \left( \exc_2 \left(\cl^{U_1\times X_2\times X_3}_{U_1\times W_{2} } (p_{23}^\ast \Gamma_2) \cup  p_{12}^\ast \left( \cl^{U_1\times X_2}_{W_{1}\setminus (W_{1})_{R_1}}(   \Gamma_1 )\cup ( {p'}_{1}^\ast \alpha ) \right) \right)|_{X_1\times X_2\times U_3}  \right)\right) ,
\end{align*}
where  
$$
\exc_2\colon H^\ast_{W_{12}\setminus (W_{12})_{R_1}}(U_1\times X_2\times X_3,-) \stackrel{\lowcong}\longrightarrow H^\ast_{W_{12}\setminus (W_{12})_{R_1}}( (X_1\times X_2\times X_3)\setminus  (W_{12})_{R_1},-)
$$
is the isomorphism
given by excision.
Here  we use that
$(W_{12})_{R_1}=W_{12}\cap (R_1\times X_2\times X_3)$ and $W_3=p_{13}(W_{12})$, and so
$p_3((W_{12})_{R_1})=p_3^{13}(W_3\cap R_1\times X_3)= R_3=X_3\setminus U_3$ by Lemma \ref{lem:R3}.

The functoriality of pushforwards then shows 
\begin{align*}
\Gamma_3(W_3)_\ast(\alpha) &= (p^{23}_{3})_\ast \left( (p_{23})_\ast \left( \exc_2 \left(\cl^{U_1\times X_2\times X_3}_{U_1\times W_{2} } (p_{23}^\ast \Gamma_2) \cup  p_{12}^\ast \left( \cl^{U_1\times X_2}_{W_{1}\setminus (W_{1})_{R_1}}(   \Gamma_1 )\cup ( {p'}_{1}^\ast \alpha ) \right) \right)|_{X_1\times X_2\times U_3}  \right)  \right) .
\end{align*}

Since the action of cycles is compatible with pullbacks along open immersions (hence with excision) by condition~\ref{item:action-of-cycles:compatible-with-restriction},   $\Gamma_3(W_3)_\ast(\alpha) $ is given by
\begin{align*}
(p^{23}_{3})_\ast \left( (p_{23})_\ast \left( \exc_2 \left(\cl^{X_1\times X_2\times X_3}_{X_1\times W_{2} } (p_{23}^\ast \Gamma_2) \cup  p_{12}^\ast \circ \exc_3 \left( \cl^{U_1\times X_2}_{W_{1}\setminus (W_{1})_{R_1}}(   \Gamma_1 )\cup ( {p'}_{1}^\ast \alpha ) \right)|_{X_1\times U_2} \right)|_{X_1\times X_2\times U_3}  \right) \right) ,
\end{align*}
where
$$
\exc_3\colon  H^\ast_{W_{1}\setminus (W_{1})_{R_1}}(U_1\times X_2 ,-) \stackrel{\lowcong}\longrightarrow H^\ast_{W_{1}\setminus (W_{1})_{R_1}}((X_1\times X_2)\setminus  (W_{1})_{R_1},-)
$$
is the isomorphism
given by excision, and where we use that $R_2=p_2^{12}((W_1)_{R_1})$; see \eqref{eq:R3}.

The compatibility of proper pushforwards with pullbacks along open immersions (hence with excision) from condition~\ref{item:f_*:open-immersion} then shows
that $\Gamma_3(W_3)_\ast(\alpha) $ is given by
\begin{align*}
(p^{23}_{3})_\ast \left( \exc_4\left( (p_{23})_\ast  \left(\cl^{X_1\times X_2\times X_3}_{X_1\times W_{2} } (p_{23}^\ast \Gamma_2) \cup  p_{12}^\ast \circ \exc_3 \left( \cl^{U_1\times X_2}_{W_{1}\setminus (W_{1})_{R_1}}(   \Gamma_1 )\cup ( {p'}_{1}^\ast \alpha ) \right)|_{X_1\times U_2 }  \right) \right)|_{X_2\times U_3} \right) ,
\end{align*}
where
$$
\exc_4\colon H^\ast_{W_{2}}(U_2\times X_3 ,-) \stackrel{\lowcong}\longrightarrow H^\ast_{W_{1}\setminus (W_{1})_{R_1}}((X_2\times X_3)\setminus  p_{23}((W_{12})_{R_1}),-)
$$
is the isomorphism
given by excision, where $p_{23}((W_{12})_{R_1})=p_{23}(W_{12} \cap (R_1\times X_2\times X_3))$  and so $ p_3(p_{23}( (W_{12}) _{R_1}))=R_3$; see \eqref{eq:R3}.

Applying the (last) projection formula from condition~\ref{item:action-of-cycles:projection-formula-everything} to $p_{23}$ then shows that $\Gamma_3(W_3)_\ast(\alpha) $ is given by
\begin{align*}
(p^{23}_{3})_\ast \left( \exc_4\left(  \cl^{  X_2\times X_3}_{W_{2} } ( \Gamma_2) \cup   (p_{23})_\ast \circ p_{12}^\ast \circ \exc_3 \left( \cl^{U_1\times X_2}_{W_{1}\setminus (W_{1})_{R_1}}(   \Gamma_1 )\cup ( {p'}_{1}^\ast \alpha ) \right) | _{X_1\times U_2 }  \right)|_{X_2\times U_3} \right) .
\end{align*}
We then consider the natural pushforward map
$$
\epsilon\colon H^\ast_{W_{1}\setminus (W_{1})_{R_1}}(X_1\times U_2 ,-)\longrightarrow H^\ast (X_1\times U_2 ,-) 
$$
and note that by the functoriality of pullbacks and proper pushforwards, the above class identifies with
\begin{align*}
(p^{23}_{3})_\ast \left( \exc_4\left( \cl^{  X_2\times X_3}_{W_{2} } ( \Gamma_2) \cup   (p_{23})_\ast \circ p_{12}^\ast \circ \epsilon\circ \exc_3 \left( \cl^{U_1\times X_2}_{W_{1}\setminus (W_{1})_{R_1}}(   \Gamma_1 )\cup ( {p'}_{1}^\ast \alpha ) \right) | _{X_1\times U_2 } \right) |_{X_2\times U_3} \right) .
\end{align*}
By condition~\ref{item:f_*:f,g-smooth}, we have $  (p_{23})_\ast \circ p_{12}^\ast \circ \epsilon = (p^{23}_2)^\ast \circ (p^{12}_2)_\ast \circ \epsilon $, and so
$\Gamma_3(W_3)_\ast(\alpha) $ is given by
\begin{align*}
(p^{23}_{3})_\ast \left( \exc_4\left(  \cl^{  X_2\times X_3}_{ W_{2} } ( \Gamma_2) \cup   (p^{23}_2)^\ast \circ (p^{12}_2)_\ast  \circ \epsilon \circ \exc_3 \left( \cl^{U_1\times X_2}_{W_{1}\setminus (W_{1})_{R_1}}(   \Gamma_1 )\cup ( {p'}_{1}^\ast \alpha ) \right) | _{X_1\times U_2 }  \right)|_{X_2\times U_3} \right) .
\end{align*} 
This shows that
\begin{align*}
\Gamma_3(W_3)_\ast(\alpha) 
&=
(p^{23}_{3})_\ast \left( \exc_4\left( \cl^{  X_2\times X_3}_{X_1\times W_{2} } ( \Gamma_2) \cup   (p^{23}_2)^\ast \left( \Gamma_1(W_1)_\ast(\alpha) \right) \right) |_{X_2\times U_3} \right) \\
&=
\Gamma_2(W_2)_\ast(\Gamma_1(W_1)_\ast(\alpha)) ,
\end{align*}
as we want.
\end{proof}

\section{The moving lemma} \label{sec:moving}

Theorem \ref{thm:moving-lemma-cohomology} stated in the introduction will be deduced from the following stronger (but more technical) result.

\begin{theorem} \label{thm:moving-body}
Let $k$ be a field, and fix a twisted cohomology theory  as in \eqref{eq:cohomology-functor} which satisfies conditions  \ref{item:excision}--\ref{item:action-of-cycles} from Section \ref{sec:axioms}. 
 Let $U$ be a smooth equi-dimensional algebraic $k$-scheme that admits a smooth projective compactification, and let
 $f\colon S\to U$ be a morphism from an algebraic $k$-scheme $S$ that is locally of pure dimension. 
Then for any nowhere dense closed subset $Z \subset U$, there are closed subsets $Z' \subset W \subset U$ with  $Z \subset W$ and $\dim W\leq \dim Z +1$  such that the following conditions are satisfied:
\begin{enumerate}
\item \label{item:moving-main:0} 
The subsets $Z '$ and $W\setminus Z $   are in good position with respect to $f$ in the following sense:
\begin{enumerate}[ref=\theenumi.\alph*]
\item The codimension of $f^{-1}(Z ')$ in $S$ is, locally at each point,  at  least $\codim_UZ $. \label{item:moving-main:1}
\item The codimension of $f^{-1}(W\setminus Z )$ in $ S$ is, locally at each point,  at  least $\codim_UZ -1$.\label{item:moving-main:2}
 \end{enumerate} 
\item  \label{item:moving-main:3}
Classes with support on $Z $ can be moved along $W$ to classes with support on $Z '$ in the following sense:
$$
\im (H^\ast_{Z}(U,n)\to H^\ast_{W}(U,n) ) \subset \im (H^\ast_{Z '}(U,n)\to H^\ast_{W}(U,n) ).
$$ 
\item \label{item:moving-main:4}
There is an open subset $S^\circ \subset S$ with the following properties:
\begin{itemize}
\item The complement $S\setminus S^\circ$  locally has codimension at least $\codim_UZ-1$ in $S$. 
In particular, $S^\circ$ is dense in $S$.
\item If $\dim f(S)<\codim_UZ -1$, then  $S^\circ=S$.
\item If we replace $S$ by $S^{\circ}$, then $W$ is well behaved under localization  in the following sense: if we replace $U$ by an open subset $V\subset U$  and $S$ by $f^{-1}(V)\cap S^{\circ}$, then   $W$ can be replaced by $W\cap V$. 
\end{itemize}
\item \label{item:moving-main:5}
Any component  $Z''\subset Z'$ with $\dim Z''>\dim Z $ satisfies $f^{-1}(Z'')=\emptyset$. 
If we only require items \eqref{item:moving-main:0} and \eqref{item:moving-main:3},  but not \eqref{item:moving-main:4}, then we may assume that no such component exists.  
\end{enumerate}  
\end{theorem}

\begin{remark} \label{rem:localization-Z'}
Even if $\dim S=0$, the subset $Z'\subset U$ in Theorem \ref{thm:moving-body} can in general not be chosen to be well behaved with respect to localization.
This can already be observed in the case where $U=\CP^2$, $Z\subset \CP^2$ is a line,  $S\subset Z$ is a closed point, $k =\bar k$ is algebraically closed, and $H^\ast(-,n)$ denotes \'etale cohomology with coefficients in $\mu_{\ell^r}^{\otimes n}$ for some prime $\ell$ invertible in $k$.
Indeed, if in this case $Z'\subset \CP^2$ is any closed subset which meets $S$ properly, then $Z\not \subset  Z'$, and so we can pick two points $\{p,q\}\subset Z\setminus Z'$.
But then there is a class
$$
\alpha\in H^3_Z(\CP^2\setminus \{p,q\},n)\cong H^1(Z\setminus \{p,q\},n-1)\cong \Z/\ell^r
$$
with non-trivial residue $\del_p\alpha$ at $p$, and so the pushforward of $\alpha$ to $H^3(\CP^2\setminus \{p,q\},n)$ does not admit a lift to $ H^3_{Z'}(\CP^2\setminus \{p,q\},n)$ because the latter agrees with $H^3_{Z'}(\CP^2 ,n)$ by excision and so the residue at $p$ would need to be trivial. 
\end{remark}

\begin{proof} [Proof of Theorem \ref{thm:moving-body}]
Let $X$ be a smooth projective compactification of $U$, and let $\bar Z  \subset X$ be the closure of $Z $ in $X$.
Since $U$ is equi-dimensional, so is $X$, and we let $d_X:=\dim X$.
Let $R:=\bar Z \setminus Z $.
By excision (see condition~\ref{item:excision}), the canonical restriction map $H^\ast_{Z }(X\setminus R,n)\to H^\ast_{Z }(U,n)$ is an isomorphism.
Since the pushforward maps from  condition~\ref{item:f_*} are compatible with respect to pullbacks  along open immersions (see condition~\ref{item:f_*:open-immersion}), one easily concludes that it suffices to prove the theorem in the case where
\begin{align}\label{eq:reduction-R}
U=X\setminus R\quad \text{with } R:=\bar Z \setminus Z .
\end{align}

For convenience, we will use the following terminology: if $\varphi:A\to B$ is a morphism of locally equi-dimensional algebraic schemes with $B$ equi-dimensional and $Z$ is a closed codimension $c$ subscheme of $B$ or a cycle on $B$ whose support has codimension $c$ in $B$, then we say that $Z$ is in good position with respect to $\varphi$ if $\varphi^{-1}(Z)\subset A$ locally has codimension at least $c$ (\textit{i.e.} locally at each point, the codimension is at least $c$).

The idea of the proof is to apply the moving lemma for algebraic cycles to the diagonal  $\Delta_X\subset X\times X$ and to exploit the action of cycles on open varieties from Section \ref{sec:action}.
The former yields the following precise statements.

\begin{lemma}\label{lem:transversal-auxiliary-lemma}
Let $\varphi\colon A\to X\times X$ be a morphism from a locally equi-dimensional algebraic scheme $A$ to $X\times X$.
\begin{enumerate}
\item \label{item:1-transverse-auxiliary}
There is a cycle $\Delta_X'\in Z^{d_X}(X\times X)$, rationally equivalent to the diagonal $\Delta_X$, that is in good position with respect to $\varphi$.
\item \label{item:2-transverse-auxiliary}
Assume  that  $\Delta_X$ is in good position with respect to $\varphi$, and let $\Delta_X'$ be a cycle on $X\times X$ that is rationally equivalent to $\Delta_X$ and that is in good position with respect to $\varphi$.
Then there is a closed subscheme $W_{X\times X}\subset X\times X$ of codimension $d_X-1$ which is in good position with respect to $\varphi$ and such that the following hold:
\begin{itemize} 
\item The supports of the cycles $\Delta_X$ and $\Delta_X'$ are contained in $W_{X\times X}$.
\item The cycle $\Gamma:=\Delta_X-\Delta_X'$, viewed as a cycle on $W_{X\times X}$, is rationally equivalent to zero on $W_{X\times X}$.
\end{itemize} 
\end{enumerate} 
\end{lemma}
\begin{proof}
  We apply Theorem \ref{thm:moving}\eqref{item:Levine:moving:1}  
  to the class of the diagonal $\Delta_X$ to get a cycle $\Delta'_X$ that is rationally equivalent to $\Delta_X$ and that is in good position with respect to $\varphi$. 
This proves the first assertion.
To prove the second assertion,  assume that $\Delta_X$ is in good position with respect to $\varphi$.
Then we consider the cycle $\Gamma=\Delta_X-\Delta_X'$ that is rationally equivalent to zero on $X\times X$ and in good position with respect to $\varphi$ by assumption.
We then apply  Theorem \ref{thm:moving}\eqref{item:Levine:moving:2}  to the cycle $\Gamma=\Delta_X-\Delta_X'$ and get a closed subscheme $W_{X\times X}\subset X\times X$ of codimension $d_X-1$ which is in good position with respect to $\varphi$ and which contains the support of $\Gamma$ such that $\Gamma$, viewed as a cycle on $W_{X\times X}$, is rationally equivalent to zero on $W_{X\times X}$.
The resulting cycle $\Delta_X'$ and the subscheme $W_{X\times X}$ then have the properties we want in the lemma.
\end{proof}

\begin{lemma} \label{lem:transversal}
There exist a closed subscheme  $W_{X\times X}\subset X\times X$ of codimension $d_X-1$ and a cycle $\Delta_X'\in Z^{d_X}(X\times X)$ with $\supp(\Delta'_X)\subset W_{X\times X}$ and $\Delta_X\subset W_{X\times X}$ such that $\Delta_X$ and $\Delta_{X}'$, viewed as cycles on $W_{X\times X}$, are rationally equivalent to each other on $W_{X\times X}$.
Moreover, the following transversality properties hold:
\begin{enumerate} %[label=(\emph{\alph*})] 
\item \label{item:transverse-intersection-1}
$W_{X\times X}$ and $\Delta'_X$ are in good position with respect to 
 the morphism
 $$
 e:= \tau_{\bar Z}\times\id\colon \bar Z^{\nu} \times X \longrightarrow X\times X ,
 $$
 where $\bar Z^{\nu}$ denotes the disjoint union of the components of\, $\bar Z$ and $\tau_{\bar Z}\colon \bar Z^\nu\to \bar Z$ denotes the natural map. 
\item \label{item:transverse-intersection-2}
$W_{X\times X}$ and $\Delta'_X$ are in good position with respect to 
 the morphism
$$
g:=\tau_R\times f\colon R^{\nu}\times S\longrightarrow X\times X ,
$$
where $R^{\nu}$ denotes the disjoint union of the components of $R$  and $\tau_R:R^{\nu}\to X$ denotes the natural map.  
That is,  the  preimage $g^{-1}(W_{X\times X})$ $($resp.\ $g^{-1}(\supp(\Delta_X')))$ locally has codimension at least $d_X-1$ $($resp.\ at least $d_X)$ in $ R^{\nu}\times S$. 
\item \label{item:transverse-intersection-3}
Consider the morphism 
$$
h:=\tau_{\bar Z }\times f\colon \bar Z ^\nu\times S\longrightarrow X\times X .
$$
Then the following hold:
\begin{enumerate}[ref=\theenumi.\alph*]
\item The cycle $\Delta'_X$ is in good position with respect to $h$; i.e.\  the preimage $h^{-1}(\supp(\Delta_X'))$  locally has codimension at least $d_X$ in $\bar Z ^\nu\times S$.\label{item:transverse-intersection-3.1}
\item The  preimage $h^{-1}(W_{X\times X}\setminus \Delta_X)$   locally has codimension at least $d_X-1$.\label{item:transverse-intersection-3.2}
\end{enumerate}   
\end{enumerate}
\end{lemma}

\begin{proof}
We will apply Lemma \ref{lem:transversal-auxiliary-lemma} to a suitable morphism from a locally equi-dimensional scheme to $X\times X$.
For convenience, we first explain how this implies items \eqref{item:transverse-intersection-1}--\eqref{item:transverse-intersection-3}  individually.
Afterwards, we will explain how to arrange that all items hold simultaneously, as we want. 

\eqref{item:transverse-intersection-1}~ 
The preimage of the diagonal $\Delta_X$ via $e\colon \bar Z^{\nu} \times X \to X\times X $ identifies with the graph of $\tau_{\bar Z}$ and hence has locally codimension $d_X$.
Item \eqref{item:transverse-intersection-1} therefore follows directly  by applying Lemma \ref{lem:transversal-auxiliary-lemma} to the morphism $e$.
 
\eqref{item:transverse-intersection-2}~
Since $\im f\subset U$, while $R=X\setminus U$, we find that $g^{-1}(\Delta_X)=\emptyset$.
Item  \eqref{item:transverse-intersection-2}  therefore follows by applying Lemma \ref{lem:transversal-auxiliary-lemma} to the morphism $g$ above.

\eqref{item:transverse-intersection-3.1}~
Item \eqref{item:transverse-intersection-3.1} follows from  Lemma \ref{lem:transversal-auxiliary-lemma}\eqref{item:1-transverse-auxiliary} applied to the morphism $h$.

\eqref{item:transverse-intersection-3.2}~
We note that $\Delta_X$ is not in good position with respect to $h$, and so  Lemma \ref{lem:transversal-auxiliary-lemma}\eqref{item:2-transverse-auxiliary} does not apply directly.
Instead, we consider the morphism
$$
h'\colon (\bar Z ^\nu\times S) \setminus h^{-1}(\Delta_X) \longrightarrow X\times X 
$$
that is given by the restriction of $h$.
Clearly, $\Delta_X$ is in good position with respect to this morphism, and so Lemma \ref{lem:transversal-auxiliary-lemma} applies.
Since $(h')^{-1}(W_{X\times X})$ identifies naturally with $h^{-1}(W_{X\times X}\setminus \Delta_X)$,  we find that $h^{-1}(W_{X\times X}\setminus \Delta_X)$  locally has codimension $d_X-1$,  as we want.

It remains to show that we can arrange that items  \eqref{item:transverse-intersection-1}--\eqref{item:transverse-intersection-3} hold simultaneously.
To this end, we introduce the following notation:
by the disjoint union of a collection of morphisms $\varphi_i\colon A_i\to X\times X$, we mean the natural morphism $\sqcup \varphi_i\colon \bigsqcup A_i\to X\times X$.

We consider the  disjoint union of the morphisms $e$,  $g$,  and $h$ above and apply  Lemma \ref{lem:transversal-auxiliary-lemma}\eqref{item:1-transverse-auxiliary} to this morphism to get a cycle $\Delta_X'$.
It follows from the above arguments that the resulting cycle  $\Delta_X'$ satisfies items \eqref{item:transverse-intersection-1}--\eqref{item:transverse-intersection-3} in the lemma.
We then apply  Lemma \ref{lem:transversal-auxiliary-lemma}\eqref{item:2-transverse-auxiliary} to the disjoint union of the morphisms $e$,  $g$,  and $h'$ to get a closed subset $W_{X\times X}\subset X\times X$ as in the lemma.
As before it follows from the above arguments that $W_{X\times X}$ satisfies items \eqref{item:transverse-intersection-1}--\eqref{item:transverse-intersection-3} in the lemma, as we want.
This concludes the proof of Lemma~\ref{lem:transversal}.   
\end{proof}

From now on we fix $\Delta_X'$ and $W_{X\times X}$ as in Lemma \ref{lem:transversal}.
We then define
\begin{align} \label{eq:Z1'}
\bar Z ':=q(W_{X\times X}\cap (R\times X))\cup q(\supp(\Delta_X')\cap (\bar Z \times X)), \quad   Z ':=\bar Z '\cap U ,
\end{align} 
\begin{align} \label{eq:Z2}
\bar W:=q(W_{X\times X}\cap (\bar Z \times X)), \quad\text{and} \quad W:=\bar W \cap U.
\end{align}
(Here and in the following,  $q\colon X\times X\to X$ denotes, as before, the projection onto the second factor.)
Since $\Delta_X\subset W_{X\times X}$, we have $Z \subset W$.
Since $R\subset \bar Z $ (by the definition of $R$, \textit{cf.} \eqref{eq:reduction-R}) and $\supp(\Delta'_X)\subset W_{X\times X}$, we also find that $Z' \subset W$.
Moreover, since $W_{X\times X}\subset X\times X$ has codimension $d_X-1$, Lemma \ref{lem:transversal}\eqref{item:transverse-intersection-1} implies $\dim W\leq\dim Z +1$, as we want.

It remains to prove that conditions~\eqref{item:moving-main:0}--\eqref{item:moving-main:5} of Theorem~\ref{thm:moving-body} are satisfied. 

\begin{enumerate}[wide,itemsep=10pt,label={\textbf{Step~\arabic*.}},ref=\arabic* ]
\item \label{step1}   %\textbf{Step 1.}
  Proof that condition~\eqref{item:moving-main:3} is satisfied.

Let $U':=U\setminus Z' $.
By the long exact sequence of triples (see condition~\ref{item:les-triple}), it suffices to show that the following composition is zero:
\begin{align} \label{eq:composition:step3-proof-of-moving}
H^\ast_{Z }(U,n)\stackrel{\iota_\ast}\longrightarrow H^\ast_{W}(U,n) \xrightarrow{\restr} H^\ast_{W}(U',n) .
\end{align}

We view $\Delta_X$ as a cycle on $X\times X$.
Its support is contained in $W_{X\times X}$.
By \eqref{eq:Z2}, $W=q(W_{X\times X} \cap (\bar Z \times X))\cap U$.
Moreover, \eqref{eq:Z1'} implies that $q(W_{X\times X}\cap (R\times X))\subset X\setminus U'$. 
Lemma \ref{lem:Gamma-ast} thus yields an action
$$
\Delta_X(W_{X\times X})_\ast\colon H^\ast_{Z }(U,n)\longrightarrow H^\ast_{W}(U',n) .
$$
By Lemma \ref{lem:Gamma-ast:Gamma=Delta}, this action identifies with the composition in \eqref{eq:composition:step3-proof-of-moving}.

By Lemma \ref{lem:transversal}, the cycle $\Delta_X$ is rationally equivalent to $\Delta_X'$ on $W_{X\times X}$.
Lemma \ref{lem:Gamma-ast:Gamma-sim-0} and the additivity of the action thus imply that 
$$
\Delta_X(W_{X\times X})_\ast=\Delta'_X(W_{X\times X})_\ast\colon H^\ast_{Z }(U,n)\longrightarrow H^\ast_{W}(U',n) .
$$

By \eqref{eq:Z1'}, we have  
$q(\supp(\Delta_X')\cap (\bar Z \times X))\subset \bar Z '$.   
Moreover,  $U'\cap \bar Z' =\emptyset$ because $U'=U\setminus Z' =X\setminus \bar Z' $, where the last equality uses that  $\bar Z \setminus Z =X\setminus U$ by \eqref{eq:reduction-R}.
Hence,  Lemma \ref{lem:Gamma-ast-zero-if-wrong-support} implies that
$$
\Delta'_X(W_{X\times X})_\ast\colon H^\ast_{Z }(U,n)\longrightarrow H^\ast_{W}(U',n)
$$
is zero. This concludes the proof of Step~\ref{step1}.

\item\label{step2}%\textbf{Step 2.}
  Proof of that condition~\eqref{item:moving-main:0} is satisfied. 

Since $\im f\subset U\subset X$, the definition of $Z'$ in \eqref{eq:Z1'} implies that
$$
f^{-1}(Z')=f^{-1}(q(W_{X\times X}\cap (R\times X)))\cup f^{-1}(q(\supp(\Delta_X')\cap (\bar Z \times X))) .
$$
This splits into a union of two closed subsets, and it suffices to bound the codimension of each of these two subsets separately.

Consider the  commutative diagram
$$
\xymatrix{
R^\nu \times S\ar[r]^{\id\times f} \ar[d]^{\pr_2}& R^\nu\times X\ar[d]^{\pr_2}\ar[r]^{\tau_R\times \id}& X\times X\ar[d]^{q}\\
S\ar[r]^f  & X\ar[r]^{\id} & X\rlap{,} 
}
$$
where each vertical arrow is the respective projection onto the second factor,  the square on the left is Cartesian, and the composition of the two horizontal arrows on the top coincides with the morphism $g$ from Lemma \ref{lem:transversal}\eqref{item:transverse-intersection-2}.
It is elementary to check that we have an equality of sets
$$
\pr_2((\id\times f)^{-1}(A))=f^{-1}(\pr_2(A))
$$
for any subset $A\subset R^\nu\times X$.
Applying this to $A=(\tau_R\times \id)^{-1}(W_{X\times X})$, we get
$$
\pr_2(g^{-1}(W_{X\times X}))=f^{-1}(\pr_2((\tau_R\times \id)^{-1}(W_{X\times X}))) .
$$
The right-hand side agrees with $f^{-1}(q(W_{X\times X} \cap (R\times X)))$,  and so
$$
\pr_2(g^{-1}(W_{X\times X}))=f^{-1}(q(W_{X\times X} \cap (R\times X))). 
$$
The codimension of $\pr_2(g^{-1}(W_{X\times X}))$ in $S$
 is bounded  below by 
$$
\codim_{R^{\nu}\times S} (g^{-1}(W_{X\times X})) -\dim R^\nu.
$$
By Lemma \ref{lem:transversal}\eqref{item:transverse-intersection-2}, the above number is bounded  below by $d_X-1-\dim R$, where we use $\dim R=\dim R^\nu$.
Since $R=\bar Z \setminus Z $ is nowhere dense in $\bar Z $, we finally conclude that the above number is bounded  below by $d_X-\dim Z $.
Hence, $f^{-1}(q(W_{X\times X}\cap (R\times X)))$ has codimension at least $\codim_X\bar Z =\codim_U Z $ in $S$, as we want.

Similarly, consider the commutative diagram
\begin{align} \label{eq:diag-f&h}
\xymatrix{
\bar Z^\nu \times S\ar[r]^{\id\times f} \ar[d]^{\pr_2}& \bar Z^\nu\times X\ar[d]^{\pr_2}\ar[r]^{\tau_{\bar Z}\times \id}& X\times X\ar[d]^{q}\\
S\ar[r]^f  & X\ar[r]^{\id} & X ,
}
\end{align}
where the composition of the morphisms in the top row is $h$.
As before, one checks that there is an equality of sets
\begin{align} \label{eq:diag-f&h-consequence}
\pr_2((\id\times f)^{-1}(A))=f^{-1}(\pr_2(A))
\end{align}
for any subset $A\subset \bar Z^\nu\times X$.
Applying this to $A=(\tau_{\bar Z}\times \id)^{-1}(\supp \Delta'_X)$, we conclude in a similar way as above that the codimension of $ f^{-1}(q(\supp(\Delta_X')\cap (\bar Z \times X)))$ in $S$ is bounded  below by 
$$
\codim_{\bar Z ^{\nu}\times S} \left(h^{-1}\left(\supp\left(\Delta'_X\right)\right)\right) -\dim \bar Z .
$$
By  Lemma \ref{lem:transversal}\eqref{item:transverse-intersection-3.1}, this number is bounded  below by $d_X-\dim \bar Z =\codim_X\bar Z $, as we want.
This concludes the proof that condition~\eqref{item:moving-main:1} is satisfied.

Next, we aim to prove that condition~\eqref{item:moving-main:2} is satisfied. 
Using \eqref{eq:diag-f&h} and applying \eqref{eq:diag-f&h-consequence} to $(\tau_{\bar Z}\times \id)^{-1} (W_{X\times X})$, we find that
$$
f^{-1}(\pr_2 ((\tau_{\bar Z}\times \id)^{-1} (W_{X\times X})) )=\pr_2(h^{-1}(W_{X\times X})).
$$
Note that  $\pr_2 ((\tau_{\bar Z}\times \id)^{-1} (W_{X\times X})) =q( W_{X\times X}\cap (\bar Z\times X))=\bar W $, where the first equality uses that $\bar Z^\nu$ is the disjoint union of the irreducible components of $\bar Z$ and the last equality uses \eqref{eq:Z2}.
Hence,
$$
f^{-1}(W)= f^{-1}(\bar W)=\pr_2(h^{-1}(W_{X\times X})),
$$
where the first equality uses $\im f\subset U$.
Similarly, $f^{-1}(Z)=\pr_2(h^{-1}(\Delta_X))$,  and so
$$
f^{-1}(W\setminus Z )\subset \pr_2(h^{-1}(W_{X\times X}\setminus \Delta_X)) .
$$
Hence, the codimension of $f^{-1}(W\setminus Z )$ in $S$ is bounded  below by 
$$
\codim_{\bar Z ^{\nu}\times S}(h^{-1}(W_{X\times X}\setminus \Delta_X) )-\dim \bar Z .
$$
By  Lemma \ref{lem:transversal}\eqref{item:transverse-intersection-3.2}, the above number is bounded  below by 
$$
d_X-1-\dim Z =\codim_UZ  -1 .
$$
This proves that condition~\eqref{item:moving-main:2} of Theorem \ref{thm:moving-body} is satisfied and hence concludes Step~\ref{step2}.

\item\label{step3}%\textbf{Step 3.}
  Proof of that condition~\eqref{item:moving-main:4} is satisfied. 

If $\codim_U Z=1$, then we may take $W=U$, and so condition~\eqref{item:moving-main:4} clearly holds for $S^\circ:=S$.  
 Since $Z\subset X$ is nowhere dense, it remains to consider the case $\codim_UZ\geq 2$. 
To this end, we consider $ h^{-1}(W_{X\times X} \setminus \Delta_X)\subset \bar Z^\nu\times S$ and define 
$$ 
S^{\circ}:=S\setminus S^{\cl}\quad \text{with }  
S^{\cl}:= \overline{ \pr_2(  h^{-1}(W_{X\times X} \setminus \Delta_X))}\subset S.
$$
By \eqref{eq:diag-f&h-consequence},
$$
S^{\cl}= \overline{f^{-1}(A)},\quad \text{where } A:=\pr_2((\tau\times \id)^{-1}(W_{X\times X}\setminus \Delta_X)) .
$$ 

By Lemma \ref{lem:transversal}\eqref{item:transverse-intersection-3}, $ h^{-1}(W_{X\times X} \setminus \Delta_X)$ locally has codimension at least $d_X-1$. 
Hence,  
$S^{\cl}$ locally has codimension at least  $\dim U-\dim Z-1=\codim_UZ-1$ in $S$.
In particular, the open subset $S^\circ \subset S$ is dense because $\codim_UZ\geq 2$.
Note further that $\im h=\bar Z\times f(S)$.
Therefore, if $\dim f(S)<\codim_UZ-1$, then $\dim (\im h)<d_X-1$.
The aforementioned observation that  $ h^{-1}(W_{X\times X} \setminus \Delta_X)$ locally has codimension at least $d_X-1$ thus implies by Lemma \ref{lem:fibre-dimension} that $S=S^\circ$, as we want. 

It remains to prove that up to replacing $S$ by $S^\circ$, the closed subset $W$ is well behaved under localization.
By the definition of $S^\circ$, this replacement has the effect that we may from now on assume that 
\begin{align} \label{eq:h^-1=emptyset}
 h^{-1}(W_{X\times X}\setminus \Delta_X)=\emptyset.
\end{align}

Now let $V\subset U$ be an open subset, and let $R_V:=X\setminus V$.
By excision, we reduce to the case where $R_V\subset \bar Z $.
Then $R\subset R_V\subset \bar Z $, and replacing $U$ by $V$ amounts to replacing $R$ by $R_V$ and $S$ by $S_V:=f^{-1}(V)$. 
We aim to show that under these replacements,  
$W$ can be replaced by $W\cap V$.
In view of the definition in \eqref{eq:Z2},
  it will be enough to show that $W_{X\times X}$  from Lemma \ref{lem:transversal} does not need to be changed if we replace $R$ by $R_V$ and $S$ by $S_V$.

It is clear that the conclusion of Lemma \ref{lem:transversal} is not affected if we replace $S$ by $S_V$, and so we may from now on assume that $f(S)\subset V$.
It remains to analyse the effect of replacing $R$ by the possibly larger subset $R_V\subset \bar Z $.
Items \eqref{item:transverse-intersection-1} and \eqref{item:transverse-intersection-3} of Lemma~\ref{lem:transversal} are not affected by this replacement. 
It thus suffices to show that Lemma \ref{lem:transversal}\eqref{item:transverse-intersection-2} remains true after replacing $R$ by $R_V$.
To this end, we consider  the map
$$
g_V:=\tau_{R_V}\times f\colon R_V^{\nu}\times S\longrightarrow X\times X
$$
that we obtain by replacing $R$ by $R_V$ in Lemma \ref{lem:transversal}\eqref{item:transverse-intersection-2}.
Since $R_V\subset \bar Z$ (by the above reduction step), 
$h^{-1}(W_{X\times X}\setminus \Delta_X)=\emptyset$ (see \eqref{eq:h^-1=emptyset}) implies $g_V^{-1}(W_{X\times X}\setminus \Delta_X)=\emptyset $. 
Since $f(S)\subset V$ by the above reduction step, we have $R_V\cap f(S)=\emptyset$ and so $g_V^{-1}(\Delta_X)=\emptyset$.
We thus conclude
$g_V^{-1}(W_{X\times X})=\emptyset$.
Since $\supp(\Delta_X')\subset W_{X\times X}$, this also implies that $g_V^{-1}(\supp(\Delta_X'))=\emptyset$.
Hence, the conclusion in  Lemma \ref{lem:transversal}\eqref{item:transverse-intersection-2} remains true if we replace $R$ by $R_V$, as claimed.
This concludes the proof of that condition~\eqref{item:moving-main:4} of Theorem~\ref{thm:moving-body} is satisfied.

\item\label{step4}%\textbf{Step 4.}
  Proof that condition~\eqref{item:moving-main:5} is satisfied.

We first note that $Z '\subset W$.
Since $ \dim W\leq \dim Z +1 $, we thus get $\dim Z '\leq   \dim Z +1$.
By condition~\eqref{item:moving-main:1}, $f^{-1}(Z ')$ has, locally at each point, codimension at least $\codim_UZ$.
But if $Z ''\subset Z '$ is a component of dimension $\dim Z +1$, then $f^{-1}(Z '')$, if non-empty,  locally has codimension at most $\codim_U Z -1$.
Hence, $f^{-1}(Z '')=\emptyset$, as claimed.

Finally,  we only ask that conditions~\eqref{item:moving-main:0} and  \eqref{item:moving-main:3} of Theorem~\ref{thm:moving-body} are satisfied (but not \eqref{item:moving-main:4}), and we aim to show that we can arrange to have $\dim Z \geq \dim Z '$. 
To this end, we use the moving lemma for cycles in the form of Lemma~\ref{lem:transversal-auxiliary-lemma} to ask that in addition to the properties in Lemma \ref{lem:transversal}, $W_{X\times X}$ is in good position with respect to the natural map
$$
R^\nu\times X\longrightarrow X\times X.
$$ 
(We can do this because the diagonal $\Delta_X$ is clearly in good position with respect to the above morphism, and so the above morphism can simply be added to the disjoint unions of morphisms used in the proof of Lemma \ref{lem:transversal}.)
By the definition of $Z'$ in \eqref{eq:Z1'}, we have
$$
\bar Z  '=q(W_{X\times X}\cap (R\times X))\cup q(\supp(\Delta_X')\cap (\bar Z \times X)), \quad
Z ':=\bar Z '\cap U .
$$

Since $W_{X\times X}$ is in good position with respect to the above map $R^\nu\times X\to X\times X$, we find that
\begin{align*}
 \dim( q(W_{X\times X}\cap (R\times X)))&\leq  \dim( W_{X\times X}\cap (R\times X)) \\
 &\leq \dim R+1\\
& \leq \dim \bar Z=\dim Z,
\end{align*}
where the last inequality uses that $R=\bar Z\setminus Z$ is a closed nowhere dense subset of $\bar Z$ and hence has strictly smaller dimension than $\bar Z$.

By Lemma \ref{lem:transversal}\eqref{item:transverse-intersection-1}, the cycle $\Delta'_X$ is in good position with respect to the natural map $e\colon \bar Z^\nu\times X\to X\times X$, which implies
\begin{align*}
\dim (q(\supp(\Delta_X')\cap (\bar Z \times X))) &\leq  \dim (\supp(\Delta_X')\cap (\bar Z \times X))  \\
&\leq \dim (\bar Z ) =\dim (Z) .
\end{align*}

Altogether, we get $\dim Z \geq \dim \bar Z' \geq \dim Z '$, as we want.
(Here we emphasize that we cannot ensure any longer that condition~\eqref{item:moving-main:4} of Theorem~\ref{thm:moving-body} holds true:
 unless every component of $W_{X\times X}\cap (\bar Z \times X)$ dominates a component of $\bar Z $, we cannot ensure for varying $R$ that the intersection of $W_{X\times X}$ with $R\times X$ is dimensionally transversely,  hence we may lose the localization property from  Step~\ref{step3}  above.)

 This concludes the proof of Step~\ref{step4}, and hence the proof of the theorem.\hfill\qedhere
\end{enumerate}
\end{proof}

\section{Applications: Proofs of Theorem \ref{thm:moving-lemma-cohomology} and Corollaries \ref{cor:global-effacement}--\ref{cor:H_j-nr-basic}}
\label{sec:applications-of-moving}

\subsection{Proof of Theorem \ref{thm:moving-lemma-cohomology}}

\begin{proof}  [Proof of Theorem \ref{thm:moving-lemma-cohomology}]
 Let $X$ be a smooth equi-dimensional $k$-scheme that admits a smooth projective compactification, and let $S,Z\subset X$ be  closed subsets.
Let $S^\nu$ be the disjoint union of the irreducible components of $S$, and let $f\colon S^\nu\to X$ be the natural map.
We apply Theorem \ref{thm:moving-body} to $U:=X$.
We only ask that conditions~\eqref{item:moving-main:0} and \eqref{item:moving-main:3} hold true, but not \eqref{item:moving-main:4}.
By condition~\eqref{item:moving-main:5}, we may thus  assume that $\dim Z' \leq \dim Z $.
By condition  \eqref{item:moving-main:3},  
\begin{align}\label{eq:moving-statement-intro}
\im (H^\ast_{Z}(X,n)\lra H^\ast_{W}(X,n) ) \subset \im (H^\ast_{Z'}(X,n)\lra H^\ast_{W}(X,n) ).
\end{align}
In other words, for any $\alpha\in H^\ast_Z(X,n)$, there is a class $\alpha'\in H_{Z'}^\ast(X,n)$ such that $\alpha$ and $\alpha'$ have the same image in 
  $H^\ast_W(X,n)$. 
We further know that $\dim W\leq \dim Z+1$, and by condition~\eqref{item:moving-main:0}, we have $\codim_X (S\cap Z')\geq \codim_X S + \codim_X Z$ and $\codim_X (S\cap (W\setminus Z))\geq \codim_X S + \codim_X Z-1$.
So if we know that the  inequalities $\dim Z' \leq \dim Z $ and $\dim W\leq \dim Z+1$ are in fact equalities, we will be done. 
Note however that passing from inequalities to equalities is trivial: if some of the inequalities are strict, we add to
$Z'$ (resp.\ $W$)  closed subsets of the desired dimensions, in general position with respect to $S$,  in such a way that $Z'\subset W$ remains true.
(This uses the assumption that $X$ admits a projective closure, and so we can construct the subsets we add as suitable complete intersections.)
Equation \eqref{eq:moving-statement-intro} then remains true by the functoriality of pushforwards (see condition~\ref{item:f_*}); the conditions $\codim_X (S\cap Z')\geq \codim_X S + \codim_X Z$ and $\codim_X (S\cap (W\setminus Z))\geq \codim_X S + \codim_X Z-1$ remain true by the fact that we added subsets  in general position with respect to $S$. 
Since $\dim W=\dim Z+1$ and $\dim Z'=\dim Z$,  it follows that
 $Z'$ and $W\setminus Z$ meet $S$ dimensionally transversely.
This concludes the proof of the theorem.
\end{proof}

\subsection{Global and local effacement theorems}
Corollary \ref{cor:global-effacement}, stated in the introduction, results from the following more general statement.

\begin{corollary} 
\label{cor:global-effacement-body}
 Let $X$ be a smooth equi-dimensional $k$-scheme that admits a smooth projective compactification.
Let $Z\subset X$ be a closed subset, and let  $f\colon S\to X$ be a morphism from an algebraic $k$-scheme $S$ that is locally of pure dimension. 
Then there is  a closed subset $W\subset X$ with $Z\subset W$, and $\dim W=\dim Z+1$, and there is an open subset $U\subset X$  with $\codim_X(X\setminus U)\geq \codim_X(Z)$ such that $f^{-1}(X\setminus U)\subset S$ locally has codimension at least $\codim_X(Z)$  and  
 such that  the following composition is zero:
$$
H^\ast_{Z }(X,n)\longrightarrow H^\ast_{W}(X,n)\longrightarrow H^\ast_{W}(U,n) .
$$ 
Moreover,  if $\dim f(S)+\dim Z<\dim X$, then 
 $U$ is an open neighbourhood of $f(S)$ in $X$.
 \end{corollary} 

\begin{proof} 
We apply  Theorem \ref{thm:moving-body} to $X$ and ask that conditions~\eqref{item:moving-main:0}, \eqref{item:moving-main:3}, and \eqref{item:moving-main:5} are satisfied, but we do not require the localization property from condition~\eqref{item:moving-main:4}.
We thus get closed subsets $Z'\subset W\subset X$ with $Z\subset W$, $\dim Z'=\dim Z$ and $\dim W=\dim Z+1$ such that $f^{-1}(Z')\subset S$  locally has codimension at least $\codim_X(Z)$ (see condition~\eqref{item:moving-main:1}).
We then define $U:=X\setminus Z'$. 
By the long exact sequence of triples (see condition~\ref{item:les-triple}), Equation~\eqref{eq:moving-statement-intro} implies that the composition
$ H^\ast_{Z}(X,n)\to H^\ast_{W}(X,n) \to H^\ast_{W}(U,n)
$ 
is zero, as claimed. 
Moreover, if $\dim f(S)+\dim Z<\dim X$, then Lemma \ref{lem:fibre-dimension} together with the fact that $f^{-1}(Z')\subset S$  locally has codimension at least $\codim_X(Z)$ implies $f^{-1}(X\setminus U)=f^{-1}(Z')=\emptyset$, and so $U$ is an open neighbourhood of $f(S)$, as we want.
This concludes the proof.
\end{proof}

The following result proves Corollary \ref{cor:tubular-neighbourhood}, stated in the introduction.
Before we state it, we recall that for a smooth equi-dimensional $k$-scheme $X$ and a closed subset $S\subset X$, we denote by $X_S$ the pro-scheme that consists of all open subsets $U\subset X$ with $S\subset U$.
The cohomology of $X_S$ with support in a closed subset $Z\subset X$ is then formally defined by the direct limit
$$
H^\ast_Z(X_S,n):=\lim_{\substack{\longrightarrow\\ S\subset U\subset X}} H^\ast_Z(U,n),
$$
where $U$ runs through all open subsets of $X$ that contain $S$.

\begin{corollary} \label{cor:ses-neighbourhoods}
 Let $X$ be a smooth equi-dimensional $k$-scheme that admits a smooth projective compactification.
 Let $Z,S\subset X$ be closed subsets with $\dim S+\dim Z<\dim X$.
Then  the natural map $H^\ast_Z(X_S,n)\to H^\ast(X_S,n)$ is zero, and there is a natural short exact sequence 
$$
0\longrightarrow H^i(X_S,n)\longrightarrow H^i(X_S\setminus Z,n)\stackrel{\del}\longrightarrow H^{i+1}_Z(X_S,n)\longrightarrow 0 .
$$
\end{corollary}

\begin{proof}  
Recall that the proper pushforwards from condition~\ref{item:f_*} are functorial and compatible with respect to pullbacks along open immersions (see condition~\ref{item:f_*:open-immersion}).
Using this, we see that
Corollary \ref{cor:global-effacement}  implies that the natural map $H^\ast_Z(X_S,n)\to H^\ast(X_S,n)$ is zero.
The given short exact sequence is then a direct consequence of  the long exact sequence of triples in condition~\ref{item:les-triple} which is functorial with respect to pullbacks and hence with respect to restrictions to the open neighbourhoods of $S$ in $X$ that define the pro-scheme $X_S$. 
This concludes the proof.
\end{proof}

The following corollary implies Corollary \ref{cor:local-effacement}, stated in the introduction.

\begin{corollary} 
\label{cor:local-effacement-body}
Let $X$ be a smooth equi-dimensional $k$-scheme which admits a smooth projective compactification.
Let $Z\subset X$ be a nowhere dense closed subset.
Let $S\subset X$ be either closed with $\dim S <\codim Z-1$ or a finite set of points.
Then there is a closed subset $W\subset X$ with $Z\subset W$ and $\dim W=\dim Z+1$ such that  the natural map  
$
H^\ast_Z(X_S,n)\to H^\ast_W(X_S,n) 
$
is zero.
Moreover, if $S$ is a finite set of points and $T\subset X_S$ is closed, then we may assume that the restriction of\, $W\setminus Z$ to $X_S$ meets  $T$ dimensionally transversally.
\end{corollary}

\begin{proof} %[Proof of Corollary \ref{cor:local-effacement}]
Let $Z\subset X$ be closed and nowhere dense.  
We first deal with the case where $S\subset X$ is closed with $\dim S+\dim Z<\dim X-1$.
Let $S^\nu$ be the disjoint union of the irreducible components of $S$, and let $f\colon S^\nu\to X$ be the natural map.
We apply Theorem \ref{thm:moving-body} to $U:=X$ and $f$ to get closed subsets $Z',W\subset X$ with $Z,Z'\subset W$ and $\dim W\leq \dim Z+1$.
In contrast to the proof of Theorem \ref{thm:moving-lemma-cohomology}, we require  that in addition to conditions~\eqref{item:moving-main:0} and \eqref{item:moving-main:3},  condition~\eqref{item:moving-main:4} also holds  true,  and so our assumptions on $S$ ensure that $W$ is well behaved with respect to localization. 
The disadvantage is that $Z'$ may have components of dimension $\dim Z+1$, but condition~\eqref{item:moving-main:5} ensures that any such component is disjoint from $S$.
Hence, $U=X\setminus Z'$ is a neighbourhood of $S$ such that the composition
$ H^\ast_{Z}(X,n)\to H^\ast_{W}(X,n) \to H^\ast_{W}(U,n)
$ is zero.
The fact that $W$ is well behaved under localization then implies that
$H^\ast_Z(X_S,n)\to H^\ast_W(X_S,n)
$
is zero, as claimed.

It remains to deal with the case where
 $S\subset X$ is a finite set of possibly non-closed points.
Let $\bar S^\nu$ be the disjoint union of the closures in $X$ of the points in $S$.
(In particular, there  $S$ corresponds to the generic points of the irreducible components of $\bar S^\nu$.) 
We apply Theorem \ref{thm:moving-body} to $U=X$ and the natural map $f\colon\bar S ^{\nu} \to X$.
By  Theorem \ref{thm:moving-body}\eqref{item:moving-main:4}, there is a dense open subset $\bar S^\circ \subset  \bar S^\nu$ such that the closed subset $W$ in Theorem \ref{thm:moving-body} with respect to $f|_{\bar S^{\circ}}\colon\bar S^\circ \to X$ is well behaved under localization. 
The result follows because the open subset $\bar S^\circ \subset  \bar S^\nu$  is dense and so $S$ corresponds by the construction of $\bar S^\nu$ to the generic points of the components of~$\bar S^\circ$. 

If moreover $T\subset X_S$ is closed, we denote its closure in $X$ by $\bar T\subset X$.
Let further $\bar T^\nu$ denote the disjoint union of the irreducible components of $\bar T$.
Then in the above argument we  replace the morphism $f$ by the natural morphism $\tilde f\colon \bar S \sqcup \bar T^\nu \to X$.
By  Theorem \ref{thm:moving-body}\eqref{item:moving-main:4}, we get dense open subsets $\bar S^\circ\subset \bar S$ and $\bar T^\circ \subset\bar  T^\nu$ such that the closed subset $W\subset X$ produced this way is well behaved under localization.
We claim that locally at any point of $S$,  $W\setminus Z$ meets each component of $\bar T$ dimensionally transversally.
This will follow from  Theorem \ref{thm:moving-body}\eqref{item:moving-main:2} if we can show that 
\begin{align} \label{eq:ScapT=ScapTcirc}
S\cap \bar T=S\cap \bar T^{\circ}
\end{align}
as this implies that $\bar T$ and $\bar T^\circ$ have the same restriction to $X_S$ and so $\bar T^\circ|_{X_S}=T$.
To prove  \eqref{eq:ScapT=ScapTcirc}, we note that by Theorem \ref{thm:moving-body}\eqref{item:moving-main:4}, $\bar S^\circ \cup \bar T^\circ$ is the preimage of some open subset via the natural map $\bar S\cup \bar T\to X$. 
Since $S\subset \bar S^\circ$, we get that $S\cap \bar T\subset \bar T^\circ$ and so \eqref{eq:ScapT=ScapTcirc} holds, as we want.
This concludes the proof of the corollary.
\end{proof}

\subsection{Finite-level version of the Gersten conjecture}

The following  result  implies Corollary \ref{cor:gersten}, stated in the introduction. 

\begin{corollary} \label{cor:gersten-body}
Assume that the given cohomology theory with support satisfies conditions \ref{item:excision}--\ref{item:semi-purity}.
Let $X$ be a smooth equi-dimensional algebraic $k$-scheme which admits a smooth projective compactification $($\textit{e.g.} $\operatorname{char} k=0)$.
Let $S\subset X^{(c)}$ be a finite set of codimension $c$ points, and let $X_S=\Spec(\mathcal O_{X,S})$ be the localization of\, $X$ at $S$.
Let $T\subset X_S$ be closed, and let $Z_c=S\subset Z_{c-1}\subset \dots \subset Z_1\subset Z_0=X_S$ be a chain of closed reduced subschemes of\, $X_S$ of increasing dimensions such that each $Z_j$ meets $T\setminus S$ properly. 
Up to replacing the given chain $\{Z_j\}_j $ by one that is finer, i.e.\ by a chain $\{Z_j'\}$ as above with $Z_j\subset Z_j'$ for all $j$ $($in particular, each $Z'_j$ meets $T\setminus S$ properly$)$, 
the following complex is exact:
\begin{align*}
0\lra H^i(X_S,n)\lra  H_{\BM}^{i}(X_S\setminus Z_{1}) \stackrel{\del}\lra H_{\BM}^{i-1} (Z_1\setminus Z_{2}) \stackrel{\del}\lra \cdots \stackrel{\del}\lra H_{\BM}^{1}(Z_{i-1}\setminus Z_{i}) \stackrel{\del}\lra H_{\BM}^{0}(Z_{i}\setminus Z_{i+1}) \lra 0 ,
\end{align*}
where
$$
H^{i-j}_{\BM}(Z_j\setminus Z_{j+1}):=\lim_{\substack{\longrightarrow \\ S\subset  U\subset X}}H^{i+j}_{\bar Z_{j}\setminus \bar Z_{j+1}}(U\setminus  \bar Z_{j+1},n+j ),
$$
$\bar Z_j\subset X$ denotes the closure of $Z_j$ in $X$, and $\del$ is induced by the respective residue maps from the long exact sequence of triples in property~\ref{item:les-triple}.  
\end{corollary}

\begin{proof}
This follows from Corollary \ref{cor:local-effacement-body} by standard arguments (\textit{cf.} \cite[Section~1]{CTHK}). 
We give some details for the reader's convenience.

Up to replacing $X$ by an affine neighbourhood of $S$, we can assume that $X$ is affine.
Taking the closure of the given chain $\{Z_j\}_j$ of closed subsets of the localization $X_S$, we get  a chain 
$$
\bar Z_c=\bar S\subset \bar Z_{c-1}\subset \dots \subset \bar Z_1\subset \bar Z_0=X
$$
of closed subsets of $X$ with $\dim \bar Z_j=\dim \bar Z_{j+1}+1$.
We replace this chain inductively by a finer chain $\{\bar Z'_j\}$ of closed subsets of $X$, as follows.
We let $\bar Z'_c:=\bar Z_c=S$, and if $\bar Z'_{j+1}$ is defined, then we define $\bar Z'_j$ to be the closed subset from Corollary \ref{cor:local-effacement-body} which satisfies 
\begin{itemize}
\item $\bar Z'_{j+1}\subset \bar Z'_j$ and $\dim \bar Z'_j=\dim \bar Z'_{j+1}+1$;
\item locally at any point of $S$, the closure  $\bar T\subset X$ of $T$ in $X$ meets  the subset $\bar Z'_j\setminus \bar Z'_{j+1}$ dimensionally transversely;
\item the map
\begin{align} \label{eq:gersten-body-effacement}
H^i_{\bar Z'_{j+1}}(X_S,n)\longrightarrow H^i_{\bar Z'_{j}}(X_S,n)
\end{align}
is zero.
\end{itemize}  
Up to enlarging $\bar Z_j'$, we may in addition assume that $\bar Z_j\subset \bar Z'_j$, so that the inductively defined chain $\{\bar Z'_j\}$ is a refinement of $\{ \bar Z_j\}$.
Since $\bar Z_j$ meets $\bar T\setminus \bar S$ properly, $\bar T$ meets $ \bar  Z'_j\setminus \bar Z'_{j+1}$ properly for all $j\geq c-1$.
Since $\bar Z'_c=\bar Z_c=\bar S$, it follows inductively that $\bar T$ meets $\bar Z'_j\setminus \bar S$ properly.
This way we constructed a refinement $\{\bar Z_j'\} $ of $\{\bar Z_j\} $.
We then let $Z'_j\subset X_S$ be the closed subset given by restriction of $\bar Z_j'\subset X$.
This way we get a chain $\{Z_j'\} $ of closed subsets of $X_S$ such that each $Z'_j$ meets $T\setminus S$ properly and \eqref{eq:gersten-body-effacement} is zero.
Hence, up to replacing the chain $\{Z_j\}$ of closed subsets of $X_S$ by the refinement $\{  Z'_j\}$  induced by $\{\bar Z'_j\}$, we may assume that   each $Z_j$ meets $T\setminus S$ properly and the natural map
\begin{align} \label{eq:proof-gersten-body:map-is-zero}
H^i_{\bar Z_{j+1}}(X_S,n)\longrightarrow H^i_{\bar Z_{j}}(X_S,n)
\end{align}
is zero for all $i,j$.

For each open neighbourhood $U\subset X$ of $S$, the long exact sequence of  triples from condition~\ref{item:les-triple} applied to $\bar Z_{j+1}\subset \bar Z_j\subset U$ yields
$$
\cdots \lra  H^{\ast-1}_{\bar Z_{j} }(U\setminus \bar Z_{j+1},n) \stackrel{\del} \lra H^\ast_{\bar Z_{j+1}}(U,n) \stackrel{\iota_\ast} \lra H^\ast_{\bar Z_{j}}(U,n) \stackrel{r} \lra  H^\ast_{\bar Z_{j} }(U\setminus \bar Z_{j+1},n) \stackrel{\del}\lra H^{\ast+1}_{\bar Z_{j+1}}(U,n) \lra \cdots .
$$
These sequences are compatible with respect to restrictions to finer open subsets $U'\subset U$.
This gives rise to compatible exact couples $D_1(U)\stackrel{\iota_\ast}\to D_1(U)\stackrel{r}\to E_1(U) \stackrel{\del}\to$ with
$$
D_1^{p,q}(U):=H^{p+q}_{\bar Z_p}(U,n)\quad \text{and}\quad E_1^{p,q}(U):=H^{p+q}_{\bar Z_{p}}(U\setminus \bar Z_{p+1},n),
$$
where $\iota_\ast$, $r$, and $\del$ are of bidegrees $(-1,1)$, $(0,0)$, and $(0,1)$, respectively; see \cite[Section~1.1]{CTHK}.
Passing to the direct limit over all open neighbourhoods $U\subset X$ of $S$ in $X$, we arrive at an exact couple $D_1\stackrel{\iota_\ast}\to D_1\stackrel{r}\to E_1 \stackrel{\del}\to$, where
$$
D_1^{p,q}:=\lim_{\substack{\longrightarrow \\ S\subset U\subset X}} H^{p+q}_{\bar Z_p}(U,n)\quad \text{and}\quad
 E_1^{p,q} :=\lim_{\substack{\longrightarrow \\ S\subset U\subset X}} H^{p+q}_{\bar Z_{p}}(U\setminus \bar Z_{p+1},n) .
$$
The map $\iota_\ast\colon D_1^{p,q}\to D_1^{p-1,q+1}$ is zero by
 \eqref{eq:proof-gersten-body:map-is-zero} for all $p,q$.
 It follows that the spectral sequence associated to the couple $D_1\stackrel{\iota_\ast}\to D_1\stackrel{r}\to E_1 \stackrel{\del}\to$ degenerates at $E_1$.
That is, we get a long exact exact sequence
$$
\cdots \lra  E_1^{p-1,q} \lra  E_1^{p,q} \lra  E_1^{p+1,q} \lra \cdots, 
$$ 
where the differential is given by $r\circ \del$.
This is precisely the exact complex stated in the corollary.
The corollary follows because each $Z_j$ meets $T\setminus S$ properly and semi-purity (see condition~\ref{item:semi-purity}) implies that
$$
H^{i-j}_{\BM}(Z_j\setminus Z_{j+1})=\lim_{\substack{\longrightarrow \\ S\subset  U\subset X}}H^{i+j}_{\bar Z_{j}\setminus \bar Z_{j+1}}(U\setminus  \bar Z_{j+1},n+j )
$$
vanishes for $j>i$; \textit{i.e.} Borel--Moore cohomology vanishes in negative degrees.
\end{proof}

\subsection{Injectivity and codimension j+1 purity theorems}

 Corollary \ref{cor:injectivity+purity-thm} follows from the following more general statement.

 \begin{corollary} 
\label{cor:inj+purity-thm-body}
Let $X$ be a smooth variety over $k$ that admits a smooth projective compactification.
Let $f\colon S\to X$ be a morphism from a locally equi-dimensional algebraic $k$-scheme $S$.
Then the following holds for any twisted cohomology theory as in \eqref{eq:cohomology-functor} that satisfies conditions \ref{item:excision}--\ref{item:action-of-cycles}: 
\begin{enumerate}
\item  Any class in $H^\ast(X,n)$ that vanishes on the complement of a closed subset $Z\subset X$ already vanishes on the complement of another closed subset $Z'\subset X$ with $\dim Z'=\dim Z$ such that $f^{-1}(Z')$  locally has codimension at least $\codim_X Z$ on $S$.
In particular, if $\dim f(S) <\codim_XZ$, then $f^{-1}(Z')=\emptyset$.
\label{item:thm:inj-body}
\item \label{item:thm:purity-body}
 Let $U\subset X$ be open with $\codim(X\setminus U)= j+2$.
There are open subsets $U'\subset X$ and $V\subset U\cap U'$  with $\codim(X\setminus V)= j+1$ and $\codim(X\setminus U')=j+2$  such that $f^{-1}(X\setminus U')$  locally has codimension at least $j+2$ on $S$ and
$$
\im(H^\ast(U,n)\lra H^\ast(V,n))\subset \im(H^\ast(U',n)\lra H^\ast(V,n)).
$$
In particular,  if $\dim f(S)<j+2$, then $f^{-1}(X\setminus U')=\emptyset$ and so $f(S)\subset U'$.
\end{enumerate}
\end{corollary}
\begin{proof}
To prove item \eqref{item:thm:inj-body}, we apply Corollary \ref{cor:global-effacement-body} to the closed subset $Z\subset X$ and the morphism $f\colon S\to X$.
We then get a closed subset $W\subset X$ with $\codim_X(W)\geq \codim_X(Z)-1$ and $Z\subset W$,  and an open subset $U\subset X$ with complement $Z':=X\setminus U$ such that $\codim_X (Z')\geq \codim_X(Z)$ and $f^{-1}(Z')\subset S$ locally has codimension at least $\codim_X(Z)$ and the composition 
$$
H^\ast_{Z}(X,n)\longrightarrow H^\ast_{W}(X,n)\longrightarrow H^\ast_{W}(U,n) 
$$
is zero. 
If $\dim f(S)<\codim_X(Z)$, then Lemma \ref{lem:fibre-dimension}, together with the fact that $f^{-1}(Z')\subset S$ has codimension at least $\codim_X(Z)$,  implies $f^{-1}(Z')=\emptyset$, as claimed.

Now let $\alpha\in H^\ast(X,n)$ be a class that vanishes on $X\setminus Z$.
By the long exact sequence of triples from condition~\ref{item:les-triple}, the class $\alpha$ lifts to a class $\alpha'\in H^\ast_{Z}(X,n)$.
By the functoriality of the long exact sequence of triples in condition~\ref{item:les-triple}, the fact that the image of $\alpha'$ in $H^\ast_{W}(U,n)$ vanishes implies that the image of $\alpha$ in  $H^\ast (U,n) $ vanishes.
The claim in item  \eqref{item:thm:inj-body} of the corollary thus follows from the fact that $f^{-1}(Z')\subset S$ locally has codimension at least $\codim_X(Z)$.

To prove item \eqref{item:thm:purity-body},  we apply Corollary \ref{cor:global-effacement-body} to the closed subset $Z:=X\setminus U$  and  to the morphism $f\colon S \to X$, as follows. 
Any class $\alpha\in H^\ast(U,n)$ gives by the long exact sequence of triples rise to a residue
$$
\del \alpha\in H^{\ast+1}_{Z}(X,n) .
$$ 
By Corollary \ref{cor:global-effacement-body},  there is a closed subset $W\subset X$ with $Z\subset W$ and   $\codim_X W\geq j+1$  such that  the following holds: 
there is an open subset $U'\subset X$ with $\codim(X\setminus U')\geq j+2$ such that $f^{-1}(X\setminus U')$  locally has codimension at least $j+2$ and the image of $\del \alpha$ in $H^\ast_{W}(U',n) $ vanishes.
We define $V:=(X\setminus W)\cap U'$.
Since $Z\subset W$, we have $V\subset U=X\setminus Z$.
By the functoriality of the  long exact sequence of triples in condition~\ref{item:les-triple},  we have a commutative diagram with exact rows 
$$
\xymatrix{
H^\ast(X,n)\ar[d]\ar[r]&H^\ast(U,n )\ar[d] \ar[r]^-{\del}&\ar[d] H^{\ast+1}_Z(X,n) \\
H^\ast(U',n)\ar[r]&H^\ast(V,n ) \ar[r]^-{\del}& H^{\ast+1}_W(U',n)\rlap{.}
}
$$
Since the image of  $\del \alpha$ in $H^{\ast+1}_{W}(U',n) $ vanishes, we conclude that the restriction of $\alpha$ to $V$ admits a lift to $H^\ast(U',n)$.
Finally, we note that $U'$ and $W$ in the above argument do not depend on the class $\alpha$.
This proves item \eqref{item:thm:purity-body}.
\end{proof}

\subsection{Refined unramified cohomology is motivic}

\begin{definition}
Let $X$ be a smooth equi-dimensional algebraic $k$-scheme.
For $j\in \Z$, let $F_jX$ be the pro-scheme given by all open subsets $U\subset X$ with $X^{(j)}\subset U$.
Let further  
$$
H^\ast(F_jX,n):=\lim_{\substack{\longrightarrow \\ U\subset X}}H^\ast(U,n) ,
$$  
where $U $ runs through all open subsets of $X$ with $X^{(j)}\subset U$, \textit{i.e.} with $\codim_X(X\setminus U)>j$.
\end{definition}

Note that $F_jX=\emptyset$ for $j<0$, and so $H^\ast(F_jX,n)=0$ for $j<0$ by Remark \ref{rem:H_emptyset=0}.

\begin{definition}
Let $X$ be a smooth equi-dimensional algebraic $k$-scheme. 
The $\supth{j}$ refined unramified cohomology of $X$ with respect to the cohomology theory \eqref{eq:cohomology-functor} is defined by
$$
H^\ast_{j,\nr}(X,n):=\im(H^\ast(F_{j+1}X,n)\lra H^\ast(F_{j}X,n)).
$$
\end{definition}

Note that $H^\ast_{j,\nr}(X,n)=H^\ast (X,n)$ for $j\geq \dim X$ (because $F_jX=X$ in this case); see Lemma \ref{lem:H_nr=H} below for a more general statement.
Moreover,  
  $H^\ast_{0,\nr}(X,n)=H_{\nr}^\ast(X,n)$ agrees with traditional unramified cohomology (\textit{cf.} \cite{CTO,CT,Sch-survey}); see \cite[Theorem 4.1.1(a)]{CT} and  \cite[Lemma 5.8]{Sch-refined}.

For $j'\leq j$, a class on $F_{j}X$ may be restricted to $F_{j'}X$.
This implies in particular that for $j'\leq j$,  there are canonical restriction maps
$$
H^\ast_{j,\nr}(X,n)\longrightarrow H^\ast_{j',\nr}(X,n) .
$$

The main result of this section is the following 
application
 of the moving lemma (Theorem \ref{thm:moving-body}), which shows that refined unramified cohomology is a motivic invariant naturally attached to any smooth projective variety and to any cohomology functor satisfying conditions~\ref{item:excision}--\ref{item:action-of-cycles}.

\begin{corollary} \label{cor:motivic-body}
Assume that the functor from \eqref{eq:cohomology-functor} satisfies conditions~\ref{item:excision}--\ref{item:action-of-cycles}, and let $X$ and $Y$ be smooth projective equi-dimensional $k$-schemes.
Then for any $c,i,j\geq 0$, there is a bi-additive pairing
$$
\CH^c(X\times Y)\times H^i_{j,\nr}(X,n)\longrightarrow H^{i+2c-2d_X}_{j+c-d_X,\nr}(Y,n+c-d_X),\quad ([\Gamma],[\alpha])\longmapsto [\Gamma]_\ast[\alpha] 
$$
with the following properties:
\begin{enumerate}
\item \label{item:thm:motivic-body:Gamma(W)}
 Let $\Gamma\in Z^c(X\times Y)$ be a representative of\, $[\Gamma]\in \CH^c(X\times Y)$, and let $\alpha\in H^i(U,n)$ be a representative of $[\alpha]\in H^i_{j,\nr}(X,n)$ for some open $U\subset X$ whose complement $R=X\setminus U$ has codimension at least $j+2$.
Let $W:=\supp \Gamma$, $R':=q(W\cap (R\times Y) )$, and $U':=Y\setminus R'$.

If\, $W\cap(R\times Y)$ has codimension at least $j+c+2$ in $X\times Y$, 
then 
$$
[\Gamma]_\ast[\alpha]=[\Gamma(W)_\ast(\alpha)]\in H^{i+2c-2d_X}_{j+c-d_X,\nr}(Y,n+c-d_X)
$$
is represented by $\Gamma(W)_\ast(\alpha)\in H^{i+2c-2d_X}(U',n+c-d_X)$ constructed in Lemma \ref{lem:Gamma-ast}.
\item \label{item:thm:motivic-body:restriction}
If $j'\leq j$, then the following diagram commutes: 
$$
\xymatrix{
\CH^c(X\times Y)\times H^i_{j,\nr}(X,n)\ar[rr]\ar[d]&& H^{i+2c-2d_X}_{j+c-d_X,\nr}(Y,n+c-d_X)\ar[d]\\
\CH^c(X\times Y)\times H^i_{j',\nr}(X,n)\ar[rr] && H^{i+2c-2d_X}_{j'+c-d_X,\nr}(Y,n+c-d_X) ,
}
$$
where the horizontal maps are the given pairings and the vertical maps are induced by the canonical restriction maps.
\item \label{item:thm:motivic-body:correspondence-functorial}
Let $Z$ be a smooth projective equi-dimensional $k$-scheme, and let $[\Gamma']\in \CH^{c'}(Y\times Z)$.
Then for all $[\alpha]\in H^i_{j,\nr}(X,n)$,
$$
[\Gamma']_\ast([\Gamma]_\ast [\alpha]) = ([\Gamma'] \circ [\Gamma])_\ast [\alpha]  \in H^{i+2c+2c'-2d_X-2d_Y}_{j+c+c'-d_X-d_Y,\nr}(Z,n+c+c'-d_X-d_Y) ,
$$ 
where $[\Gamma'] \circ [\Gamma]$ is the composition of correspondences and $d_Y=\dim Y$.
\item \label{item:thm:motivic-body:wrong-support}
If\, $\Gamma\in Z^c(X\times Y)$ is such that $q(\supp \Gamma)\subset Y$ has codimension at least $j+c-d_X+1$,  where $q\colon X\times Y\to Y$ denotes the second projection, then $[\Gamma]_\ast [\alpha]=0$ for all $[\alpha]\in H^i_{j,\nr}(X,n)$.
\end{enumerate} 
\end{corollary}

\begin{proof}
We first construct an action as required.
For this let $[\alpha]\in H^i_{j,\nr}(X,n)$, and let $\Gamma\in Z^c(X\times Y)$.
Let $W\subset X\times Y$ be a closed subset with $\supp \Gamma \subset W$, and let $S:=W^\nu$ be the disjoint union of the irreducible components of $W$ together with the natural map $f\colon S\to X$ given by the composition $S\to W\to X\times Y\stackrel{p}\to X$.
Applying Corollary \ref{cor:inj+purity-thm-body}\eqref{item:thm:purity-body}, we find that  $[\alpha]$ can be represented by a class $\alpha\in H^i(U,n)$ such that $R=X\setminus U$ satisfies $\codim_X (R)=j+2$ and $f^{-1}(R)\subset S$ locally has codimension at least $j+2$.
Since $S:=W^\nu$, we have $\dim(f^{-1}(R))=\dim (W\cap (R\times Y))$.
Hence,  
$$
\dim W-\dim (W\cap (R\times Y))\geq j+2.
$$
If we set $c':=\codim_{X\times Y}(W)=\dim(X\times Y)-\dim W$, then we get
$$
\codim_{X\times Y}( W\cap (R\times Y))\geq j+2+c' .
$$ 
 Hence,  $R':=q(W\cap (R\times Y))\subset Y$ satisfies
 $$
 \codim_Y(R')\geq j+c'-d_X+2.
 $$  
 
We get a class
$$
\Gamma(W)_\ast(\alpha)\in H^{i+2c-2d_X}(U',n+c-d_X) ,
$$ 
constructed in Lemma \ref{lem:Gamma-ast}, where $U':=Y\setminus R'$.

From now on we assume that $c'\in \{c,c-1\}$.
Then $F_{j+c-d_X}X\subset U'$, and we get a class
\begin{equation}\label{*}
\Gamma(W)_\ast(\alpha)\in H^{i+2c-2d_X}(F_{j+c-d_X}X,n+c-d_X).
\end{equation}
It follows from the commutativity of the first square  in \eqref{eq:diag:lem:Gamma-ast-compatible-W&U} in  Lemma \ref{lem:Gamma-ast-compatible-W&U} that \eqref{*} does not depend on the chosen representative $\alpha$ of $[\alpha]\in H^i_{j,\nr}(X,n)$; that is, \eqref{*} does not change if we replace $\alpha$ by another representative $\tilde \alpha\in H^i(\tilde U,n)$ such that $\tilde R=X\setminus \tilde R$ satisfies $\codim_X(\tilde R)=j+2$ and  $f^{-1}(\tilde R)\subset S$  locally has codimension at least $j+2$.
Using this, we can apply the above construction to the union of two given closed subsets to deduce 
 from the commutativity of the second square in  \eqref{eq:diag:lem:Gamma-ast-compatible-W&U} that \eqref{*} does not depend on the chosen closed subset $W\subset X\times Y$ with $\supp \Gamma \subset W$ and $\codim_{X\times X}(W) \in \{c,c-1\}$.
Applying this to the particular case where $W=\supp \Gamma$ has codimension $c$, we see that the closed subset $R'\subset Y$ has codimension at least $j+c-d_X+2$ and so  we find that \eqref{*} is unramified:
$$
\Gamma(W)_\ast(\alpha)\in H^{i+2c-2d_X}_{j+c-d_X,\nr}(X,n+c-d_X).
$$
Moreover, the additivity in Lemma \ref{lem:Gamma-ast} implies that the above class depends additively on $\Gamma$.
Finally, if $\Gamma$ is rationally equivalent to zero, then by the definition of rational equivalence there is a closed subset $W$ of codimension $c-1$ such that $\Gamma$ is rationally equivalent to zero on $W$. 
Applying the above construction to any such $W$,  Lemma \ref{lem:Gamma-ast:Gamma-sim-0} implies $\Gamma(W)_\ast(\alpha)=0$.
Since we already know the additivity of the above construction in $\Gamma$, we conclude that \eqref{*} depends not on $\Gamma$ but only on the rational equivalence class of $\Gamma$.
Altogether, we conclude that 
$$
[\Gamma]_\ast[\alpha]:=[\Gamma(W)_\ast(\alpha)]\in H^{i+2c-2d_X}_{j+c-d_X,\nr}(Y,n+c-d_X)
$$
for some $W\subset X\times Y$ of codimension $c$ or $c-1$ and with $\supp \Gamma \subset W$
yields a well-defined pairing as we want.
 
Item \eqref{item:thm:motivic-body:Gamma(W)} in the corollary is now a direct consequence of the construction.

Item \eqref{item:thm:motivic-body:restriction} follows from the well-definedness of the construction together with commutativity of the first square of \eqref{eq:diag:lem:Gamma-ast-compatible-W&U}.

To prove item \eqref{item:thm:motivic-body:correspondence-functorial}, let $[\Gamma'] \in \CH^{c'}(Y\times Z)$.
Let further $\alpha \in H^i(U,n)$ be a representative of $[\alpha]\in H^i_{j,\nr}(X,n)$ for some open subset $U\subset X$ with complement $R_1\subset X$ of codimension at least $j+2$.
By the moving lemma for algebraic cycles (see Theorem \ref{thm:moving}), we may choose a representative  $\Gamma$ of $[\Gamma]\in \CH^c(X\times Y) $ such that $W_1=\supp \Gamma$ meets $R_1\times Y$ dimensionally transversely.
Hence, $U_2:=Y\setminus p^{12}_2(W_1\cap (R_1\times Y))$ contains $F_{c-d_X+j+1}Y$.
Applying Theorem \ref{thm:moving} once again, we may choose a   representative $\Gamma'$ of $[\Gamma']$ with support $W_2=\supp  \Gamma'$ such that 
$$
R_3=p_3( (W_1\times Z) \cap (X\times W_2) \cap (R_1\times Y\times Z))\subset Z
$$ 
has codimension at least $c+c'-d_X-d_Y+j+2$. 
Hence, the open subset $U_3:=Z\setminus R_3$ contains $F_{j+c+c'-d_X-d_Y+1}Z$, and by Proposition \ref{prop:action:composition-of-correspondences}, we have 
$$
( \Gamma\circ \Gamma')(W_{12})(\alpha)= \Gamma'(W_2)_\ast( \Gamma(W_1)_\ast (\alpha))
 $$ 
The compatibility stated in item~\eqref{item:thm:motivic-body:correspondence-functorial} therefore follows from the construction (\textit{cf.} item \eqref{item:thm:motivic-body:Gamma(W)} proven above), together with Proposition \ref{prop:action:composition-of-correspondences}.

Finally, item \eqref{item:thm:motivic-body:wrong-support} follows from item \eqref{item:thm:motivic-body:Gamma(W)} and Lemmas \ref{lem:Gamma-ast-compatible-W&U} and  \ref{lem:Gamma-ast-zero-if-wrong-support}.

This concludes the proof of the corollary.
\end{proof}

Corollary \ref{cor:motivic-body} implies the existence of pullback maps along morphisms between smooth projective equi-dimensional $k$-schemes.
This can be generalized to open varieties as follows.

\begin{corollary} \label{cor:pullbacks-for-H_nr}
Let $f\to X\to Y$ be a morphism between smooth quasi-projective equi-dimensional $k$-schemes that admit a smooth projective compactification.
Then there are functorial pullback maps $f^\ast\to H^i_{j,\nr}(Y,n)\to H^i_{j,\nr}(X,n)$ with the following property: if $[\alpha]\in H^i_{j,\nr}(Y,n)$ is represented by a class $\alpha\in H^i(U,n)$ with $F_{j+1}Y\subset U$ such that $F_{j+1}X\subset f^{-1}(U)$, then $f^\ast[\alpha]$ is represented by $f^\ast(\alpha)\in H^i(f^{-1}(U),n)$.
\end{corollary}
\begin{proof} 
By Corollary \ref{cor:inj+purity-thm-body}\eqref{item:thm:purity-body} applied to the morphism $f\colon X\to Y$, we find that any class $[\alpha]\in H^i_{j,\nr}(Y,n)$ is represented by a class $\alpha\in H^i(U,n)$ with $F_{j+1}Y\subset U$ such that $F_{j+1}X\subset f^{-1}(U)$.
We may use such a representative to define $f^\ast[\alpha]:=[f^\ast(\alpha)]$, where $f^\ast(\alpha)\in H^i(f^{-1}(U),n)$.
We claim that this construction is well defined; \textit{i.e.} it does not depend on the choice of $U$ or on the choice of the representative $\alpha$ of $[\alpha]$.

\textit{Independence of the choice of $U$}: If $V\subset Y$ is another open subsets with $F_{j+1}Y\subset V$ and $F_{j+1}X\subset f^{-1}(V)$, then the resulting pullbacks coincide on $f^{-1}(U\cap V)$ by functoriality, and hence the respective class in $H^i(F_jY,n)$ is well defined.

\textit{Independence of the choice of $\alpha$}: Let $[\alpha]=[\beta]$.
By the independence of $U$, proven above, we can assume that $\alpha,\beta \in H^i(U,n)$ for some open subset $U\subset Y$  with $F_{j+1}Y\subset U$ and $F_{j+1}X\subset f^{-1}(U)$.
We replace $\alpha$ by $\alpha-\beta$ and are reduced to showing the following:
assume that $\alpha\in H^i(U,n)$ vanishes when restricted to $F_jY$; then $f^\ast\alpha$ vanishes when restricted to $F_jX$.
By Corollary \ref{cor:inj+purity-thm-body}\eqref{item:thm:inj-body}, there is an open subset $V\subset U$ with $F_jY\subset V$ and $F_jX\subset f^{-1}(V)$ such that $\alpha|_V=0$.
The functoriality of pullbacks thus shows that 
$$
f^\ast \alpha=0\in H^i(F_jX,n),
$$
as we want.
This concludes the proof of the well-definedness of the above pullback map.

The functoriality of the pullback on refined unramified cohomology follows from well-definedness together with the functoriality of pullbacks for the functor in Section \ref{sec:axioms}.
This concludes the proof.
\end{proof}

\begin{remark}
If $f\colon X\to Y$ is any morphism between smooth equi-dimensional algebraic $k$-schemes with $c=\dim Y-\dim X$, then the image $f(x)$ of a codimension $j$ point $x\in X^{(j)}$ has codimension at least $j+c$, and so the
 pushforward maps from condition~\ref{item:f_*} induce pushforward maps
$$
f_\ast\colon H^i(F_jX,n)\longrightarrow H^{i+2c}(F_{j+c}X,n+c)\quad \text{and}\quad  f_\ast\colon H^i_{j,\nr}(X,n)\longrightarrow H^{i+2c}_{j+c,\nr}(X,n+c)
$$ 
that are compatible with the pullbacks from Corollary \ref{cor:pullbacks-for-H_nr}.
We note that the existence of pushforwards does not rely on the moving lemma; see also \cite[Lemma 2.5]{Sch-griffiths}.
\end{remark}

\subsection{Behaviour under products with projective space  and birational properties of refined unramified cohomology}

The following lemma is an analogue of \cite[Corollary 5.10]{Sch-refined}.

\begin{lemma} \label{lem:H_nr=H}
Assume that the functor from \eqref{eq:cohomology-functor} satisfies conditions~\ref{item:les-triple} and \ref{item:semi-purity}.
Then for any smooth equi-dimensional  quasi-projective $k$-scheme,  the natural map  $H^i(X,m)\stackrel{\lowcong}\to H^i_{j,\nr}(X,m)$ is an isomorphism for $j\geq \lceil i/2\rceil$.
\end{lemma}

\begin{proof}
It suffices to prove that the natural map $H^i(X,m)\to H^i(F_jX,m)$ is an isomorphism for $j\geq \lceil i/2\rceil$.
To this end, let $U\subset X$ be open with complement $Z=X\setminus U$ of codimension at least $j+1$.
Then condition~\ref{item:les-triple} yields a long exact sequence
$$
\cdots \lra H^{i }_Z(X,m)\lra H^i(X,m)\lra H^i(U,m)\lra H^{i+1}_Z(X,m)\lra \cdots .
$$
By condition~\ref{item:semi-purity}, we have $H^{i+1}_Z(X,m)=0$ for $i+1<2\codim_X(Z)$.
Since $\codim_X(Z)\geq j+1$, we get $H^{i+1}_Z(X,m)=0$ for $j+1>(i+1)/2$.
The latter condition is equivalent to any of the following conditions: $j>(i-1)/2$,  $j\geq i/2$, or $j \geq \lceil i/2\rceil$.
The same argument shows that $H^i_Z(X,m)=0$ for $j+1>i/2$,  which holds if   $j \geq \lceil i/2\rceil$.
The above sequence thus  proves the lemma. 
\end{proof}

\begin{corollary}[Projective bundle formula] \label{cor:YxP^n} 
For any twisted cohomology theory as in Section \ref{sec:axioms} which satisfies conditions~\ref{item:excision}--\ref{item:semi-purity} and any smooth projective equi-dimensional $k$-scheme $Y$, there is a canonical isomorphism 
$$
\sum f_l:
 \bigoplus_{l=0}^{\min(j,n)}H^{i-2l}_{j-l,\nr}(Y,m-l)\stackrel{\lowcong}\longrightarrow 
H^i_{j,\nr}(Y \times \CP^n_k,m) , 
$$
where $f_l\colon H^{i-2l}_{j-l,\nr}(Y,m-l)\to  H^i_{j,\nr}(Y \times \CP^n_k,m)$ is given by the natural pullback map along the projection $Y\times \CP^{n-l}_k\to Y$ followed by the pushforward along the inclusion  $Y\times \CP^{n-l}_k\to Y\times \CP^{n}_k$ induced by some linear embedding $ \CP^{n-l}_k\subset  \CP^{n}_k $.
\end{corollary}
\begin{proof}
We first note  that
$$
 \bigoplus_{l=0}^{\min(j,n)}H^{i-2l}_{j-l,\nr}(Y,m-l)= \bigoplus_{l=0}^n H^{i-2l}_{j-l,\nr}(Y,m-l)
$$ 
 because $H^{i-2l}_{j-l,\nr}(Y,m-l)$ vanishes for $l>j$ by Remark \ref{rem:H_emptyset=0} since $F_jX=\emptyset$ for $j<0$. 

We have the map 
$$
f_l:H^{i-2l}_{j-l,\nr}(Y,m-l) \longrightarrow 
H^i_{j,\nr}(Y \times \CP^n_k,m)
$$ 
from the statement of the corollary, which is given by the pullback along the projection $Y\times \CP^{n-l}_k\to Y$ followed by the pushforward along the inclusion  $Y\times \CP^{n-l}_k\to Y\times \CP^{n}_k$.
In addition, we consider the map 
$$
g_l:H^i_{j,\nr}(Y \times \CP^n_k,m)\longrightarrow H^{i-2l}_{j-l,\nr}(Y,m-l)
$$
given by the pullback along a linear embedding $Y\times \CP^l_k\to Y\times \CP_k^n$ followed by the pushforward along the projection $Y\times \CP^l_k\to Y$.

Taking sums, we get maps
$$
\sum_lf_l: \bigoplus_lH^{i-2l}_{j-l,\nr}(Y,m-l) \longrightarrow 
H^i_{j,\nr}(Y \times \CP^n_k,m)
$$
and
$$
\sum_l g_l: H^i_{j,\nr}(Y \times \CP^n_k,m)\longrightarrow \bigoplus_l  H^{i-2l}_{j-l,\nr}(Y,m-l) ;
$$
here and in what follows, the indices will always run from $0$ to $n$ if not mentioned otherwise.
We claim that these maps are inverses to each other.
To see this, it suffices to show that the compositions
\begin{align} \label{eq:fl-gl-composition1}
\sum_{l'} g_{l'}\circ \sum_lf_l=\sum_{l,l'} g_{l'}\circ f_l:
\bigoplus_l H^{i-2l}_{j-l,\nr}(Y,m-l)  \longrightarrow 
H^i_{j,\nr}(Y \times \CP^n_k,m) \longrightarrow \bigoplus_l H^{i-2l}_{j-l,\nr}(Y,m-l)
\end{align}
and
\begin{align}\label{eq:fl-gl-composition2}
\sum_lf_l\circ \sum_{l'} g_{l'}=\sum_l f_l\circ g_l:
H^i_{j,\nr}(Y \times \CP^n_k,m)\longrightarrow \bigoplus_l H^{i-2l}_{j-l,\nr}(Y,m-l)\longrightarrow H^i_{j,\nr}(Y \times \CP^n_k,m)
\end{align}
both agree with the respective identities.
(Note that there is a certain asymmetry in the computation of the two compositions above, due to the fact that $\sum f_l$ maps from a direct sum while $\sum g_l$ maps to a direct sum.)

To show this,  note that
$f_l$ agrees with the action $[\Gamma_l]_\ast$ of the correspondence 
$$
\Gamma_l:=\Delta_Y\times \CP^{n-l}\in Z^{d_Y+l}(Y\times Y\times  \CP^n),
$$ 
while $g_l$ agrees with the action $[\Omega_l]_\ast$ of the correspondence 
$$
\Omega_l:=\Delta_Y\times \CP^{l}\in Z^{d_Y+n-l}(Y\times \CP^n\times Y),
$$
where $d_Y:=\dim Y$.
 
A simple computation shows that
$$
[\Omega_{l'}]\circ [\Gamma_l]=\delta_{l,l'}[\Delta_Y]\in \CH^{d_Y}(Y\times Y),
$$
where $\delta_{l,l'}$ denotes the Kronecker delta, which is $1$ for $l=l'$ and zero otherwise.
It thus follows from the additivity and functoriality of the action of correspondences in Corollary \ref{cor:motivic-body} 
 that \eqref{eq:fl-gl-composition1}
agrees with the identity map,
as we want.

To prove the same for the other composition, we will use that
$$
\Delta_{\CP^n}=\sum_{i=0}^n h^i\times h^{n-i}\in \CH^n(\CP^n\times \CP^n), 
$$
where $h\in \CH^1(\CP^n)$ denotes the class of a linear section.
With this at hand, one easily gets
$$
\sum_{l=0}^n [\Gamma_l] \circ  [\Omega_{l}] =[\Delta_{Y\times \CP^n}] \in \CH^{d_Y+n}(Y\times \CP^n\times Y\times \CP^n) .
$$
Hence,  \eqref{eq:fl-gl-composition2} agrees with the respective identity map, 
as we want. 
This concludes the proof.
\end{proof}

\begin{remark}
Kok and Zhou give an alternative proof of a version of Corollary \ref{cor:YxP^n} in \cite{kok-zhou}.
\end{remark}

The groups $H^i_{j,\nr}(X,n)$ carry a natural descending filtration $F^\ast$ with $F^{j+1}H^i_{j,\nr}(X,n)=H^i_{j,\nr}(X,n)$,  which for $m\geq j+1$ is given by (see \cite[Definition 5.3]{Sch-refined})
$$
F^mH^i_{j,\nr}(X,n):=\im(H^i(F_mX,n)\lra H^i(F_jX,n)) .
$$

\begin{corollary} \label{cor:iso-in-codim-c-H_nr}
Let $X$ and $Y$ be smooth projective equi-dimensional $k$-schemes, and let $f:X\dashrightarrow Y$ be a birational map which is an isomorphism in codimension $c$; \textit{i.e.} there are closed subsets $Z_X\subset X$ and $Z_Y\subset Y$ of codimensions greater than~$c$ such that $f\colon X\setminus Z_X\to Y\setminus Z_Y$ is an isomorphism.
Then for all $j\leq c$, there is a natural isomorphism
$$
f^\ast \colon H^i_{j,\nr}(Y,n) \stackrel{\lowsim}\longrightarrow H^i_{j,\nr}(X,n),
$$
given by the action $[\Gamma_{f^{-1}}]_\ast$ of the closure of the graph of $f^{-1}$.
Moreover, $f^\ast$
 respects the filtration $F^\ast$.
\end{corollary}

\begin{proof}
Since $f$ is birational, $d:=\dim X=\dim Y$.
Let $\Gamma:=\Gamma_f\in Z^d(X\times Y)$ and $\Gamma':=\Gamma_{f^{-1}}\in Z^d(Y\times X)$ be the closures of the graphs of $f$ and $f^{-1}$, respectively.
Since $f$ is an isomorphism in codimension $c$, we get
$$
[\Gamma]\circ [\Gamma']=[\Delta_Y]+[\Omega_Y]\in \CH^d(Y\times Y)
\quad 
\text{and}
\quad 
[\Gamma']\circ [\Gamma] =[\Delta_X]+[\Omega_X]\in \CH^d(X\times X)
$$
for some cycles $\Omega_Y,\Omega_X$ with the property that the respective images via the second projection $\pr_2(\supp(\Omega_Y))\subset Y$ and $\pr_2(\supp(\Omega_X))\subset X$ have codimension at least $c+1$.
Hence, Corollary \ref{cor:motivic-body}\eqref{item:thm:motivic-body:wrong-support} implies that 
$$
[\Omega_X]_\ast \colon H^i_{j,\nr}(X,n)\longrightarrow H^i_{j,\nr}(X,n) \quad \text{and}\quad [\Omega_Y]_\ast \colon H^i_{j,\nr}(Y,n)\longrightarrow H^i_{j,\nr}(Y,n) 
$$
are zero for $j\leq c$.
The above decompositions together with Corollary \ref{cor:motivic-body}\eqref{item:thm:motivic-body:correspondence-functorial} thus implies that 
$$
[\Gamma]_\ast\colon H^i_{j,\nr}(X,n)\longrightarrow H^i_{j,\nr}(Y,n) \quad \text{and}\quad  [\Gamma']_\ast\colon H^i_{j,\nr}(Y,n)\longrightarrow H^i_{j,\nr}(X,n)
$$
are inverse to each other for $j\leq c$.
Moreover,  Corollary \ref{cor:motivic-body}\eqref{item:thm:motivic-body:restriction}  implies that the above isomorphisms respect the filtrations $F^\ast$ on both sides.
This concludes the proof of the corollary.
\end{proof}

\begin{proof}[Proof of Corollary \ref{cor:H_j-nr-basic}]
Corollary \ref{cor:H_j-nr-basic} follows from Corollaries \ref{cor:YxP^n} and \ref{cor:iso-in-codim-c-H_nr}. 
\end{proof}

\appendix

\section{Twisted pro-\'etale cohomology satisfies the properties from Section \ref{sec:axioms}} \label{app:A}

The purpose of this appendix is to prove Proposition \ref{prop:pro-etale-coho} from Section \ref{subsec:example}. 
We will see that item~\eqref{item:prop:pro-etale-coho} of  Proposition \ref{prop:pro-etale-coho} implies items \eqref{item:prop:etale-coho} and \eqref{item:prop:continuous-etale-coho}, and so it suffices to treat  pro-\'etale cohomology. 
Most results and arguments are certainly well known to experts; we include them here for the reader's convenience.

 \subsection{Pro-\'etale and continuous \'etale cohomology}
The pro-\'etale site $X_{\proet}$ of a scheme $X$ is formed by weakly \'etale maps of schemes $U\to X$; see \cite[Definition 4.1.1 and Remark 4.1.2]{BS}.  
Since any \'etale map is also weakly \'etale, there is a natural map of topoi $\nu\colon \Sh(X_{\proet})\to \Sh(X_{\et})$; see \cite[Section~5]{BS}.  
In particular, we get functors $\nu^\ast\colon\Ab(X_\et)\to \Ab(X_\proet)$ and $\nu_\ast \colon\Ab(X_\proet)\to \Ab(X_\et)$ on the corresponding categories of sheaves of abelian groups.
Here,  $\nu^\ast$ is exact and right adjoint to the left  exact functor $\nu_\ast$; see  \cite[\href{https://stacks.math.columbia.edu/tag/00X9}{Tag 00X9}]{stacks-project}.
Since $\nu^\ast$ is (left) exact, $\nu_\ast$ maps injective objects to injective objects. 

For a closed subset $Z\subset X$, we have the functor
$
\Gamma_Z\colon \Ab(X_\proet)\to \Ab ,
$
given by taking global sections with support on $Z$.
For a bounded complex $K$ of sheaves of abelian groups on $X_\proet$,  we let 
\begin{align} \label{def:H_proet(X,K)}
H^i_Z(X_\proet,K):=\RR^i\Gamma_Z(X_\proet,K).
\end{align}
In the special case where $K$ is a sheaf $F$ concentrated in degree zero, we also write $H^i_Z(X_\proet,F):= \RR^i\Gamma_Z(X_\proet,F)$.
(Note that by slight abuse of notation, we do not distinguish between the hypercohomology of a complex as in \eqref{def:H_proet(X,K)} and the cohomology of a sheaf.)

\begin{lemma}[\textit{cf.} \cite{BS}] \label{lem:BS-H_cont=H_proet}
Let $X$ be a scheme, and let $(F_r )$ be an inductive system of sheaves on the small \'etale site $X_{\et}$ and with surjective transition maps. 
Let $Z\subset X$ be a closed subset.
Then there is a canonical identification
$$
H^i_{Z,\cont}(X_{\et},(F_r))\cong H^i_Z(X_{\proet},\lim_{\substack{\longleftarrow\\ r}} \nu^\ast F_r) ,
$$
where the left-hand side denotes Jannsen's continuous \'etale cohomology groups with support; see \cite{jannsen}.
\end{lemma}
\begin{proof}
This follows by the same argument as in \cite[Proposition 5.6.2]{BS}; we give the details for the reader's convenience.
By definition, $H^i_{Z,\cont}(X_{\et},(F_r))$, resp.\ $H^i_Z(X_{\proet},\lim \nu^\ast F_r) $, is the $\supth{i}$ cohomology of the complex
$$
\RR\Gamma_Z (X_{ \et},\RR \lim(F_r)) \in D(\Ab), \quad \text{resp. }\  \RR\Gamma_Z (X_{ \proet},\lim \nu^\ast F_r )\in D(\Ab) .
$$

We claim that $\Gamma_Z(X_{\proet},-)\colon\Ab(X_\proet)\to \Ab$ takes injectives to injectives. 
To prove this, note that $\Gamma_Z(X_{\proet},-)$ is right adjoint to the exact functor $A\mapsto \iota_\ast \underline{A}_Z$, where $\iota\colon Z\hookrightarrow X$ denotes the inclusion and $ \underline{A}_Z$ denotes the locally constant sheaf with values $A$ on $Z$. 
The adjointness property thus implies that  for any injective sheaf $\mathcal I\in \Ab(X_{\proet})$, the functor $\Hom(-,\Gamma_Z(X_\proet,\mathcal I))\colon\Ab\to \Ab$ is exact.
Hence, $\Gamma_Z(X_\proet,\mathcal I)\in \Ab$ is an injective object, as we want.

Similarly, for any abelian category $\mathcal B$, the functor $\lim\colon\mathcal B^{\N}\to \mathcal B$ maps injectives to injectives, see \cite[Remark 1.17(a)]{jannsen}, where $\mathcal B^{\N}$ denotes the category of inverse systems of objects in $\mathcal B$.

Since $\lim$ and $\Gamma_Z$ commute, we conclude from Grothendieck's theorem on composed functors that
$$
\RR \lim \circ \RR \Gamma_Z=\RR (\lim\circ \Gamma_Z)= \RR \Gamma_Z \circ \RR \lim
$$
as functors $\Ab(X_{\proet})^{\N}\to \Ab$.
The same argument shows that the above identity holds when the objects are viewed as functors $\Ab(X_{\et})^{\N}\to \Ab$.
 Hence,
\begin{align} \label{eq:RGammaRlim}
\RR\Gamma_Z (X_{ \et},\RR \lim(F_r))\cong \RR \lim \circ \RR\Gamma_Z (X_{ \et},F_r)&\cong \RR \lim\circ  \RR\Gamma_Z (X_{ \proet},\nu^\ast F_r )\\
&\cong  \RR\Gamma_Z (X_{ \proet},\RR \lim (\nu^\ast F_r) ), \nonumber
\end{align}
where the second isomorphism uses the fact that $\nu_\ast\colon\Ab(X_{\proet})\to \Ab(X_\et)$ maps injectives to injectives and $\id\to \RR \nu_\ast \nu^\ast$ is an equivalence on bounded below complexes (see \cite[Proposition 5.2.6(2)]{BS}), while the last equality uses $\RR \Gamma_Z\circ \RR \lim\cong\RR \lim\circ \RR \Gamma_Z$. 
By \cite[Propositions 3.2.3 and 4.2.8]{BS}, the topos $\Sh(X_\proet)$ is replete, 
and so \cite[Proposition 3.1.10]{BS} implies 
\begin{align} \label{eq:Rlim=lim}
\RR \lim (\nu^\ast F_r)\cong \lim (\nu^\ast F_r),
\end{align}
where we use that the transition maps in $(F_r)$ are surjective.
Hence, the above isomorphism simplifies to
$
\RR\Gamma_Z (X_{ \et},\RR \lim(F_r))\cong  \RR\Gamma_Z (X_{ \proet},  \lim (\nu^\ast F_r) ) 
$, which proves the proposition. 
\end{proof}

\subsection{Constructible complexes and six-functor formalism}
For simplicity, we recall the six-functor formalism in the pro-\'etale  topology  that we need from \cite{BS} only in the special case of algebraic schemes, which suffices for the purpose of this paper. 

Let $k$ be a field, and let $\ell$ be a prime that is invertible in $k$.
Let $X$ be an algebraic scheme over $k$ (\textit{i.e.} a separated scheme of finite type over $k$). 
For any integer $n$, we consider the sheaf 
$$
\widehat{\Z}_\ell(n):=\lim_{\substack{\longleftarrow\\ r}} \mu_{\ell^r}^{\otimes n} \in \Ab(X_{\proet}) 
$$
and write $\widehat {\Z}_\ell:=\widehat {\Z}_\ell(0)$, which is a sheaf of rings.
Let $D(X_{\proet},\widehat {\Z}_\ell)$ denote the derived category of the abelian category  
of $\widehat{\Z}_\ell$-modules on $ X_{\proet}$.

A complex $K\in D(X_{\proet},\widehat {\Z}_\ell)$ is constructible if it is $\ell$-adically complete, \textit{i.e.}  the canonical map
$$
K\longrightarrow \RR \lim  (K\otimes^{\mathbb L}_{\widehat \Z_\ell} \Z/\ell^{r})
$$
is a quasi-isomorphism (see \cite[Definition 3.5.2]{BS}) and if for each $r$, we have $K\otimes^{\mathbb L}_{\widehat \Z_\ell} \Z/\ell^{r}\cong \nu^\ast K_r$ for some constructible complex $K_r\in D(X_{\et},\Z/\ell^r)$ of sheaves of $\Z/\ell^r$-modules on $X_{\et}$; see \cite[Definition 6.5.1]{BS}.
The full triangulated subcategory spanned by constructible complexes is denoted by 
$$
D_{\cons}(X_{\proet},\widehat {\Z}_\ell)\subset D(X_{\proet},\widehat {\Z}_\ell).
$$
By \cite[Lemma 6.5.3]{BS}, each $K\in D_{\cons}(X_{\proet},\widehat {\Z}_\ell)$ is bounded.
Since $\mu_{\ell^{r+1}}^{\otimes n}\to \mu_{\ell^{r}}^{\otimes n}$ is surjective in the \'etale topology, \eqref{eq:Rlim=lim} implies that $\widehat{\Z}_\ell(n)$ is $\ell$-adically complete, and so for any integer $m$, the complex $\widehat{\Z}_\ell(n)[m]\in D(X_{\proet},\widehat {\Z}_\ell)$ is constructible.

For a morphism $f\colon X\to Y$ between algebraic schemes over $k$, there are functors
 {\small
$$
\RR f_\ast, \RR f_!\colon D_{\cons}(X_\proet,\widehat \Z_\ell)\longrightarrow D_{\cons}(Y_\proet,\widehat \Z_\ell),\quad f^\ast_{\comp}, f^!\colon D_{\cons}(Y_\proet,\widehat \Z_\ell)\longrightarrow D_{\cons}(X_\proet,\widehat \Z_\ell), 
$$
}where
$f_{\comp}^\ast$ is defined as the usual pullback followed by derived completion; see \cite[Section~6.7]{BS}.

We aim to recall the construction of $f^!$.
To this end,   first note that 
\begin{align}\label{eq:nu*}
\nu^\ast\colon D_{\cons}(X_\et,  \Z/\ell^r) \stackrel{\lowcong} \longrightarrow D_{\cons}(X_\proet,  \Z/\ell^r)
\end{align} 
is an equivalence for all $r\geq 1$, see \cite[paragraph after Definition 6.5.1]{BS};  we will freely use this in the following to identify complexes in the two categories with each other.
For $K\in D_{\cons}(Y_\proet,\widehat \Z_\ell)$,  we set $K_r:= K\otimes^{\mathbb L}_{\widehat \Z_\ell} \Z/\ell^{r}$.
Note that any constructible complex of sheaves of $\Z/\ell^r$-modules on $X_\proet$ is also constructible when viewed as a complex of $\widehat \Z_\ell$-modules.
Using the equivalence in \eqref{eq:nu*}, we may therefore view $f_r^!K_r$  as an object in $D_{\cons}(X_\proet,\widehat  \Z_\ell) $, where $f_r^!\colon D_{\cons}(Y_\et, \Z/\ell^r)\to D_{\cons}(X_\et, \Z/\ell^r) $ denotes the exceptional pullback along $f$ in the \'etale site; \textit{cf.} \cite[Expos\'e XVIII]{SGA4.3}.
By \cite[Lemma 6.7.18]{BS}, the $f_r^!K_r\in D_{\cons}(X_\proet,\widehat  \Z_\ell) $ form a projective system.
Bhatt--Scholze then define
\begin{align} \label{eq:f!K}
f^!K:=\RR \lim f_r^! K_r \in D_{\cons}(X_\proet,\widehat \Z_\ell); 
\end{align}
see \cite[Lemma 6.7.19]{BS}. 

Using the above construction of $f^!$, it follows from  \cite[Proposition XVIII.3.1.8]{SGA4.3} that for any closed immersion $i\colon Z\to X$, 
\begin{align} \label{eq:RGamma_Z}
\RR \Gamma_Z(X_\proet,-)\cong \RR \Gamma(Z_\proet, i^!(-)) .
\end{align}

By adjunction (see \cite[Lemmas 6.7.2 and 6.7.19]{BS}),  there are morphisms of functors $\Tr_f\colon\RR f_!f^!\to \id$ and $\theta_f\colon\id\to \RR f_\ast f^\ast_{\comp}$. 
In particular,  for any $K\in D(Y_\proet,\widehat \Z_\ell)$, the morphism $\theta_f$ yields functorial pullback maps
\begin{align} \label{eq:f^ast-K}
f^\ast\colon\RR \Gamma^i(Y_\proet,K)\longrightarrow \RR \Gamma^i(X_\proet, f_{\comp}^\ast K).
\end{align}
If $f$ is proper, then $\RR f_\ast=\RR f_!$ (see \cite[Definition 6.7.6]{BS}), and so $\Tr_f$ yields a functorial pushforward map
\begin{align} \label{eq:f_ast-K}
f_\ast \colon\RR \Gamma^i(X_\proet,f^!K)\longrightarrow \RR \Gamma^i(Y_\proet, K).
\end{align} 

\subsection{Notation}
We aim to prove Proposition \ref{prop:pro-etale-coho}.
To this end, we fix a constructible complex  
$$
K\in D_{\cons}((\Spec(k))_{\proet},\widehat {\Z}_\ell)
$$ 
of $\widehat {\Z}_\ell$-modules on $(\Spec(k))_\proet$.
We then use the following notation: 
$$
H^\ast_Z(X,n):= H^\ast_{Z}(X_\proet, \pi_{\comp}^\ast K \otimes _{\widehat \Z_\ell} \widehat \Z_{\ell}(n)).
$$ 
We emphasize that the above group of course depends on the complex $K$,  even though we decided to drop it from the notation in what follows.  
The existence of functorial pullback maps follows from \eqref{eq:f^ast-K} together with the fact that the tensor product of two constructible  complexes is constructible; see \cite[Lemma 6.5.5]{BS}.
 
\subsection{Poincar\'e duality}

\begin{lemma}\label{lem:f!g!=(gf)!}
Let $f\colon X\to Y$ and $g\colon Y\to Z$ be morphisms of algebraic schemes.
Then there is a canonical isomorphism of functors $f^!g^!\stackrel{\lowcong}\to (g\circ f)^!$ on $D_{\cons}(Z_\proet,\widehat \Z_\ell)$.
\end{lemma}
\begin{proof}
This follows from \eqref{eq:f!K} together with the analogous result on the \'etale site; 
see \textit{e.g.}  \cite[Lemma 6.1(3)]{Sch-refined} for more details.
\end{proof}

\begin{lemma}[Poincar\'e duality] \label{lem:PD}
Let $f\colon X\to Y$ be a smooth morphism of algebraic $k$-schemes of pure dimension $d$.
Then there is a canonical isomorphism $f^\ast_{\comp}(d)[2d]\stackrel{\lowcong}\to f^!$ of functors on $D_{\cons}(Y_{\proet},\widehat \Z_\ell)$, where $f^\ast_{\comp}(d):=f_{\comp}^\ast(-\otimes_{\widehat \Z_\ell}  \widehat \Z_\ell(d))$.
\end{lemma}

\begin{proof} 
This follows from Poincar\'e duality on the \'etale site; see \cite[Th\'eor\`eme XVIII.3.2.5]{SGA4.3}, by applying $\RR \lim$; the details are given in \cite[Lemma 6.1(4)]{Sch-refined}.
\end{proof}

For open immersions $j\colon U\to X$ we have a canonical isomorphism $j_{\comp}^\ast\cong j^!$; see \textit{e.g.} \cite[Lemma 6.1(2)]{Sch-refined}.

\begin{lemma}\label{lem:PD-compatible}
In the notation of Lemma \ref{lem:PD}, the following holds.
For any open immersion $j:U\hookrightarrow X$ and $f|_U:=f\circ j$, the diagram
$$
\xymatrix{
 (f|_U)_{\comp}^\ast(d)[2d] \ar[r]^-{\cong}&
f|_U^!\\ 
 j_{\comp}^\ast f_{\comp}^\ast(d)[2d] \ar[r]^-{\cong} \ar[u]  &
  j_{\comp}^\ast f^! \cong j^!f^! \ar[u]\\ 
}
$$
commutes, where the horizontal isomorphisms are given by Lemma \ref{lem:PD} and the vertical maps are the canonical maps, given by the functoriality of $f_{\comp}^\ast$ and $f^!$, respectively.
\end{lemma}
\begin{proof}
This is \cite[Lemma 6.1(5)]{Sch-refined}.
\end{proof}

 \begin{lemma}[Purity] \label{lem:purity}
Let $i:Z\hookrightarrow X$ be a closed immersion between smooth equi-dimensional algebraic schemes over $k$.
Assume that $Z$ has codimension $c$ in $X$.
Then  there is a canonical isomorphism of functors
$$
i^\ast_{\comp}(-c)[-2c]\stackrel{\lowcong}\longrightarrow i^! .
$$ 
\end{lemma}
\begin{proof}  
Let $\pi_X$ and $\pi_Z$ denote the structure morphisms of $X$ and $Z$, respectively. 
There is a canonical isomorphism  $\pi_X^!i^!\stackrel{\lowcong}\to \pi_Z^!$ (see Lemma \ref{lem:f!g!=(gf)!}).
Since $X$ and $Z$ are smooth and equi-dimensional, $\pi_X^!$ and $\pi_Z^!$ are expressed by Lemma \ref{lem:PD}, and we get a canonical isomorphism $i^\ast_{\comp}(-c)[-2c]\stackrel{\lowcong}\to i^!$.
This proves the lemma. 
\end{proof}

\subsection{Pullbacks and proper pushforwards}

\begin{lemma}\label{lem:f_*-f*}
Let $f\colon X\to Y$ be a morphism between algebraic schemes, and let $Z_Y\subset Y$ and $Z_X\subset X$ be closed subsets.
\begin{enumerate}
\item If $f^{-1}(Z_Y)\subset Z_X$, then there are functorial pullback maps\label{item:pullback}
\begin{align}  \label{eq:f^*A(n)}
f^\ast\colon H^i_{Z_Y}(Y,n)\longrightarrow H^i_{Z_X}(X,n).
\end{align}
\item Assume that $X$ and $Y$ are smooth and equi-dimensional.\label{item:pushforward}
If $f$ is proper of relative codimension $c$  and $f(Z_X)\subset Z_Y$, then there are functorial 
pushforward maps
\begin{align}  \label{eq:f_*A(n)}
f_\ast \colon H^{i-2c}_{Z_X}(X,n-c)\longrightarrow H^i_{Z_Y}(Y,n) .
\end{align}
\end{enumerate}
\end{lemma}
\begin{proof}
The pullback maps are induced by the natural map of functors $ \Gamma_{Z_Y}(Y_{\proet},-) \to \Gamma_{Z_X}(X_{\proet},f^\ast_{\comp}(-)) $ on $\Mod(Y_\proet,\widehat \Z_\ell)$.
The functoriality is clear from this description.
This proves item~\eqref{item:pullback}.

Let $\pi_X\colon X\to \Spec(k)$ and $\pi_Y\colon Y\to \Spec(k)$ denote the structure morphisms, and let $i\colon Z_Y\to Y$ and  $i'\colon Z_X\to X$ denote the respective closed immersions.
Let further $d_X:=\dim X$ and $d_Y:=\dim Y$, so that $c=d_Y-d_X$.
Applying Lemmas \ref{lem:f!g!=(gf)!} and \ref{lem:PD}, we get canonical isomorphisms  of functors
\begin{align} \label{eq:i!-pushforward}
(i')^! (\pi_X)_{\comp}^\ast (n)\cong (i')^! \pi_X^!(n-d_X)[-2d_X]\cong  (f|_{Z_X})^!i^!\pi_Y^!(n-d_X)[-2d_X]\cong (f|_{Z_X})^!i^! (\pi_Y)^\ast_{\comp} (n+c)[2c],
\end{align}
where $f|_{Z_X}\colon Z_X\to Z_Y $ is the restriction of $f$.
The trace map in \eqref{eq:f_ast-K} thus induces by \eqref{eq:RGamma_Z} pushforward maps as claimed in item \eqref{item:pushforward}.
The functoriality in $f$ follows from the functoriality of the trace map. 
This concludes the proof of the lemma. 
\end{proof}

\begin{remark}\label{rem:pushforward-id}
Let $X$ be a smooth equi-dimensional algebraic scheme, and let  $\iota_{Z'}:Z'\hookrightarrow X$ and $\iota_Z\colon Z\to Z'$ be closed subschemes.
By \eqref{eq:RGamma_Z}, we have $H^i_{Z'}(X_\proet,-)=H^i(Z'_\proet, \iota_{Z'}^! -)$ and 
$$
H^i_{Z}(X_\proet,-)=H^i(Z_\proet, \iota_{Z}^!\iota_{Z'}^!  -)=H^i(Z'_\proet, \RR {\iota_{Z}}_\ast \iota_{Z}^!\iota_{Z'}^! -).
$$
Hence, $\Tr_{\iota_Z}$ induces a pushforward map
$$
(\iota_Z)_\ast\colon H^i_Z(X,n)\longrightarrow H_{Z'}^i(X,n),
$$
which agrees with the respective pushforward map along the identity in \eqref{eq:f_*A(n)}.
By the construction of the adjunction map $\Tr_{\iota_Z}\colon \RR {\iota_Z}_\ast {\iota_Z}^!\to \id$, we find that $(\iota_Z)_\ast$ agrees with the map induced by the natural map $\RR \Gamma_Z(X_\proet,-)\to \RR \Gamma_{Z'} (X_\proet,-)$ of functors, induced by the inclusion $ \Gamma_Z(X_\proet,-)\hookrightarrow  \Gamma_{Z'} (X_\proet,-)$.
\end{remark}

\begin{lemma}[Topological invariance]\label{lem:top-invariance}
Let $f\colon X'\to X$ be a morphism of algebraic $k$-schemes that is a universal homeomorphism $($e.g.\ $X'=X^{\red}$ or $X'=X\times_k k'$ for a purely inseparable base extension $k'/k)$.
Let $Z\hookrightarrow X $ be a closed immersion with base change $Z'\hookrightarrow X'$.
Then
$
f^\ast\colon H^i_Z(X,n)\to H^i_{Z'}(X',n)
$ is an isomorphism.
\end{lemma}
\begin{proof}
The claim follows from \cite[Lemma 5.4.2]{BS}. 
\end{proof}

\subsection{Comparison to \'etale cohomology and excision} 

Since $K\in D_{\cons}((\Spec(k))_\proet ,\widehat \Z_\ell)$ is constructible,  $K\otimes_{\widehat \Z_\ell}\Z/\ell^r=\nu^\ast K_r$ for a constructible complex $K_r\in D_{\cons}((\Spec(k))_\et , \Z/\ell^r)$; see \cite[Definition 6.5.1]{BS}.
Moreover, $K\cong \RR \lim\nu^\ast K_r$ by completeness.
We then have the following; \textit{cf.} \cite[Proposition 1.6]{jannsen}.

\begin{lemma} \label{lem:extension-compatible}
Let $\pi_X\colon X\to \Spec(k)$ be an algebraic $k$-scheme with a closed subset $Z\subset X$.
Then there is a canonical short exact sequence
\begin{align} \label{eq:ses-cont-etale}
0\longrightarrow \RR ^1 \lim H^{i-1}_Z\left(X_\et,  \pi_X ^\ast K_r\otimes_{\Z/\ell^r} \mu_{\ell^r}^{\otimes n}\right)\longrightarrow H^i_Z(X,n)\longrightarrow  \lim H^{i}_Z\left(X_\et, \pi_X^\ast K_r\otimes_{\Z/\ell^r}\mu_{\ell^r}^{\otimes n} \right)\longrightarrow 0,
\end{align} 
where $\lim$ denotes the inverse limit functor along $r$.
Moreover, this sequence is compatible with respect to pullbacks \eqref{eq:f^*A(n)} and pushforwards \eqref{eq:f_*A(n)} on the individual pieces.
\end{lemma}
\begin{proof} 
Recall that $\RR \lim$ and $\RR \Gamma_Z$ commute with each other; see the proof of Lemma \ref{lem:BS-H_cont=H_proet}.
Since $K\cong \RR \lim\nu^\ast K_r$, $H^i_Z(X, n) $ is the $\supth{i}$ cohomology of the complex
$$
\RR \lim \RR \Gamma_Z\left(X_\proet, \pi_X^\ast \nu^\ast K_r\otimes \mu_{\ell^r}^{\otimes n} \right) \cong
\RR \lim \RR \Gamma_Z\left(X_\proet, \nu^\ast \pi_X^\ast K_r\otimes \mu_{\ell^r}^{\otimes n} \right), 
$$
where we use that $\nu^\ast$ and $\pi_X^\ast$ commute. 
Since $\nu_\ast\colon \Ab(X_{\proet})\to \Ab(X_\et)$ maps injectives to injectives and $\id\to \RR \nu_\ast \nu^\ast$ is an equivalence on bounded below complexes (see \cite[Proposition 5.2.6(2)]{BS}), the above complex identifies with
$$
\RR \lim \RR \Gamma_Z\left(X_\et, \pi_X^\ast K_r\otimes \mu_{\ell^r}^{\otimes n}\right). 
$$
The lemma then follows from the composed functor spectral sequence, where we note that $\RR \lim$ has cohomological dimension at most $1$ on $\Ab$.
The compatibility with respect to pullbacks is obvious; the compatibility with respect to pushforwards follows from \cite[Lemma 6.2]{Sch-refined}.
\end{proof}

\begin{lemma}[Excision] \label{lem:excision}
Let $f\colon X'\to X$ be an \'etale morphism between algebraic schemes.
Let $Z\subset X$ and $Z'\subset X'$ be closed subschemes with $f(Z')=Z$ such that $f|_{Z'}\colon Z'\to Z$ is an isomorphism and  $f(X'\setminus Z')\subset X\setminus Z$.
Then
$$
f^\ast\colon H^i_Z(X,n)\stackrel{\lowcong}\longrightarrow H^i_{Z'}(X',n) .
$$ 
\end{lemma}
\begin{proof}
Since $K$ is bounded (see \cite[Lemma 6.5.3]{BS}), induction on its length together with the local-to-global spectral sequence reduces us to the case where $K$ is a single sheaf.
In this case, the statement is proven for \'etale cohomology in \cite[Proposition~III.1.27]{milne}; the case of pro-\'etale cohomology follows from this by Lemma \ref{lem:extension-compatible}. 
\end{proof}

\begin{lemma}\label{lem:f_*circf*=degf}
Let $X$ and $Y$ be  smooth algebraic $k$-schemes.
Let $f\colon X\to Y$ be a finite \'etale morphism of constant degree $d$. 
Then the composition
$$
H^i(Y,n)\xlongrightarrow{f^\ast} H^i(X,n)\xlongrightarrow{f_\ast} H^{i}(Y,n)
$$
is given by multiplication with $d$.
\end{lemma}
\begin{proof}
By Lemma \ref{lem:extension-compatible}, it suffices to prove the statement for  \'etale cohomology with coefficients in constructible complexes of $\Z/\ell^r$-modules. By the canonical truncation functor and the exact triangle $\tau_{\leq j}K\to K\to \tau_{\geq j+1}K$, 
induction on the length of the complex $K$ reduces us  to the case where $K$ is a sheaf concentrated in a single degree.
(Here we may lose the condition that $K$ is constructible; the statement still makes sense because on the \'etale site $f^!$ and hence $\Tr_f$ exist on the derived category of bounded $\Z/\ell^r$-modules; see \cite[Th\'eor\`eme XVIII.3.1.4]{SGA4.3}.)
In this case, the result follows from 
\cite[Th\'eor\`eme XVIII.2.9, (Var 4)]{SGA4.3}. 
\end{proof}

\subsection{Compatibility of pushforwards and pullbacks}

\begin{lemma}\label{lem:pull-push-compatible}
Consider the  diagrams
$$
\xymatrix{
X'\ar[r]^{g'}\ar[d]_-{f'} & X \ar[d]^-{f} \\
Y'\ar[r]_{g}& Y ,
}
\qquad  
\xymatrix{
H_{Z_{X'}}^{i-2c} (X',n-c)\ar[d]^-{f'_\ast} &  H^{i-2c}_{Z_X} (X,n-c) \ar[d]^-{f_\ast} \ar[l]_{(g')^\ast} \\
 H^{i}_{Z_{Y'}}(Y',n)&  H^{i}_{Z_{Y}} (Y,n) \ar[l]_{g^\ast}  ,
}
$$ 
where the diagram on the left is a Cartesian diagram of smooth equi-dimensional algebraic $k$-schemes, $f$ is proper of pure relative codimension $c=\dim Y-\dim X$, $Z_X\subset X$ and $Z_Y\subset Y$ are closed with $f(Z_X)\subset Z_Y$, $Z_{Y'}=g^{-1}(Z_Y)$, and $Z_{X'}=Z_{Y'}\times_{Y}Z_{X}\subset X'$. 
\begin{enumerate}
\item If $g$ is an open immersion, then the diagram on the right commutes.\label{item:f_*:open-immersion-body}
\item If $f$ is smooth of pure relative dimension and if $Z_X=X$ and  $Z_Y= Y$,  then the diagram on the right commutes. \label{item:f_*:f,g-smooth-body}
\end{enumerate}  
\end{lemma}
\begin{proof} 
By Lemma \ref{lem:extension-compatible}, it suffices to prove the compatibilities in question for \'etale cohomology with values in constructible complexes of $\Z/\ell^r$-modules.

The case where $g$ is an open immersion holds because the isomorphism in \eqref{eq:i!-pushforward} which is used to define the pushforward map is compatible with open immersions (\textit{cf.} Lemma \ref{lem:PD-compatible}) and the trace map is compatible with base change (see  \cite[Th\'eor\`eme XVIII.2.9, (Var 2)]{SGA4.3} and \cite[Lemma 6.7.19]{BS}) and hence in particular with respect to Zariski localization.
This proves item \eqref{item:f_*:open-immersion-body}.
(In fact, by the arguments in \cite[Axiom~(1.2.2), Property~(1.3.5),  and Section~1, p.\ 186]{BO}, the result holds  true more generally in the case where $g$ is \'etale, but we will not need this.)

Next we deal with item \eqref{item:f_*:f,g-smooth-body}, so that $f$ is smooth of  pure relative dimension $d=\dim X-\dim Y$.
By Lemma \ref{lem:PD}, there are natural isomorphisms of functors 
\begin{align} \label{eq:PD-compatibility-lemma-BO}
f^!\cong f^\ast(d)[2d],\quad (f')^!\cong (f')^\ast(d)[2d] .
\end{align}
By \cite[Proposition 4.5(2)]{verdier}, the natural map of functors $(g')^\ast f^!\stackrel{\lowcong}\to  (f')^! g^\ast $ is an isomorphism.
Combining this with $\RR f_\ast \RR g'_\ast=\RR g_\ast \RR f'_\ast$, we get an isomorphism $\alpha\colon \RR f_\ast \RR g'_\ast (g')^\ast f^! \stackrel{\lowcong}\to   \RR g_\ast \RR f'_\ast (f')^! g^\ast$ which fits into a diagram
\begin{align}\label{eq:diag:BO}
\xymatrix{
\RR f_\ast f^!\ar[r]^-{\theta_{g'}}\ar[d]^{\Tr_f}& \RR f_\ast \RR g'_\ast (g')^\ast f^! \ar[r]_-{\cong}^-\alpha& \RR g_\ast \RR f'_\ast (f')^! g^\ast \ar[d] ^{\Tr_{f'}}\\
\id\ar[rr]^-{\theta_g}&& \RR g_\ast g^\ast 
}
\end{align}
of functors on $D_{\cons}(Y_\et,\Z/\ell^r)$. 
We claim that \eqref{eq:diag:BO} is commutative.
Recall that any complex $K\in D_{\cons}(Y_\et,\Z/\ell^r) $ is bounded.
If $K$ has length $n>1$, then  the canonical truncation functors yield an exact triangle $\tau_{\leq j}K\to K\to \tau_{\geq j+1}K$ such that for suitable $j$, $ \tau_{\leq j}K$ and $ \tau_{\geq j+1}K $ have length less than $n$.
As in the proof of Lemma \ref{lem:f_*circf*=degf}, by induction on the length, it thus suffices to prove that \eqref{eq:diag:BO} is commutative when applied to sheaves concentrated in degree zero: $K=F[0]$ for some $F\in \Mod(Y_\et,\Z/\ell^r)$.
To simplify notation, we will identify $K=F[0]$ with $F$.
We further use the identifications in \eqref{eq:PD-compatibility-lemma-BO}. 
It follows that $\RR f_\ast f^\ast F(d)[2d]$ is a complex concentrated in non-positive degrees, while $\RR g_\ast g^\ast F$ is concentrated in non-negative degrees.
This implies that any map $\RR f_\ast f^\ast F(d)[2d]\to  \RR g_\ast g^\ast F$ in the derived category is uniquely determined by the induced map
$$
\RR^{2d}f_\ast  f^\ast F(d)\longrightarrow g_\ast g^\ast F
$$
on the $\supth{0}$ cohomology sheaf, 
where $d=\dim X-\dim Y=-c$.
(To see this, use the canonical truncation triangle $\tau_{\leq j}\to \id \to \tau_{\geq j+1}$ for $j=0$ and for $j=-1$.)
Hence, in order to check that the diagram in question is commutative, it suffices to note that
the trace map $\Tr_f\colon \RR^{2d}f_\ast f^\ast F(d) \to F$ is compatible with base change; see \cite[Th\'eor\`eme XVIII.2.9, (Var 2)]{SGA4.3}.
This concludes the proof of the claim that 
 \eqref{eq:diag:BO} is commutative.

The commutative diagram of cohomology groups claimed in the lemma now follows by precomposing the   functors in \eqref{eq:diag:BO} with $\RR \Gamma(Y_\et,-)$,  using \eqref{eq:PD-compatibility-lemma-BO}, and applying this to the  complex $\pi_Y^! K_r\otimes_{\Z/\ell^r} \mu_{\ell^r}^{\otimes n}$. 
This concludes the proof of the lemma. 
\end{proof}

\subsection{Long exact sequence of triples and semi-purity}
\begin{lemma}[Long exact sequence of triples] 
\label{lem:les-pairs} 
Let $X$ be an algebraic scheme, and let $Z\subset W\subset X$ be closed subsets. 
Then there is a canonical long exact sequence
$$
\cdots \longrightarrow H^i_Z(X,n)  \stackrel{\iota_\ast} \longrightarrow 
H^i_W(X,n) \xrightarrow{\restr}  H^i_{W\setminus Z}(X\setminus Z,n)\stackrel{\del}\longrightarrow H^{i+1}_Z(X,n) \longrightarrow \cdots 
$$
such that for any morphism  $f\colon X'\to X$ of algebraic  schemes, with closed subsets $Z'\subset W'\subset X'$, the following hold:
\begin{enumerate}
\item If $f^{-1}(Z)\subset Z'$ and $f^{-1}(W)\subset W'$, then pullback along $f $ $($as in \eqref{eq:f^*A(n)}$)$ induces a commutative ladder between the long exact sequence of the triple $(X',W',Z')$ and that of  $(X,W,Z)$.\label{item:les:pullback}
\item If\, $X'$ and $X$ are smooth,  $f$ is proper of pure relative codimension $c$,  and $f(Z')= Z$ and $f(W')= W$,  then pushforward along $f$ $($as in \eqref{eq:f_*A(n)}$)$ induces a commutative ladder between the long exact sequence of the triple $(X',W',Z')$ and that of  $(X,W,Z)$, suitably shifted.\label{item:les:pushforward}
\end{enumerate}
\end{lemma}
\begin{proof}
If $j$ is an open immersion, then $j_{\comp}^\ast= j^\ast $  (see \cite[Remark 6.5.10]{BS}), and so we will suppress the subscript ``comp'' for pullbacks along open immersions in what follows.

Let $i^W_Z\colon Z\to W$, $i_W\colon W\to X$, $j_{W\setminus Z}\colon W\setminus Z\to W$, and $j_{X\setminus Z}\colon X\setminus Z\to X$ denote the natural inclusions.
By \cite[Lemma 6.1.16]{BS}, there is an exact triangle
$$
 \RR (i^W_Z)_\ast (i^W_Z)^! \longrightarrow \id \longrightarrow  \RR (j_{W\setminus Z})_\ast (j_{W\setminus Z})^\ast  .
$$ 
Precomposing with $\RR (i_W)_\ast$ and postcomposing with $(i_W)^!$, we get the exact triangle
$$
\RR (i_W)_\ast \RR (i^W_Z)_\ast (i^W_Z)^! (i_W)^! \longrightarrow \RR (i_W)_\ast  (i_W)^! \longrightarrow \RR (i_W)_\ast \RR (j_{W\setminus Z})_\ast (j_{W\setminus Z})^\ast (i_W)^!    .
$$ 
We apply $\RR \Gamma(X_\proet,-)$ and use $(i^W_Z)^! (i_W)^!\cong i_Z^!$ and \eqref{eq:RGamma_Z} to arrive at the triangle
$$
\RR \Gamma_Z(X_\proet,-)\longrightarrow \RR \Gamma_W(X_\proet,-)\longrightarrow \RR \Gamma ((W\setminus Z)_\proet,(j_{W\setminus Z})^\ast (i_W)^! (-)).
$$
If  $j$ is an open immersion,  then $j^\ast=j^!$ (see \textit{e.g.} \cite[Lemma 6.1(2)]{Sch-refined}), and so
$$
(j_{W\setminus Z})^\ast (i_W)^! \cong (j_{W\setminus Z})^! (i_W)^! \cong  (j_{X\setminus Z}\circ i_{W\setminus Z})^!   \cong   (i_{W\setminus Z})^!    (j_{X\setminus Z}) ^! \cong  (i_{W\setminus Z})^!    (j_{X\setminus Z}) ^\ast ,
$$
where $i_{W\setminus Z}\colon W\setminus Z\to X\setminus Z$ denotes the restriction of $i_W$.
This implies that the last term in the above triangle identifies with $\RR \Gamma_{W\setminus Z} ((X\setminus Z)_\proet, j_{X\setminus Z}^\ast (-))$. 
The long exact sequence for the triple $(X,W,Z)$ follows immediately from this.
By Lemma \ref{lem:extension-compatible}, the compatibility properties with respect to pullbacks and pushforwards reduce to the analogous results for \'etale cohomology, which are well known; see \textit{e.g.} \cite[Definition 5.1.1 and Section~7.3(1)]{CTHK}.
\end{proof}

\begin{lemma}[Semi-purity]  \label{lem:semi-purity}
Let $X$ be a smooth equi-dimensional algebraic $k$-scheme,  and let $i:Z\hookrightarrow X$ be a closed subset.   
Assume that the complex $K\in D_{\cons}((\Spec(k))_\proet,\widehat \Z_\ell)$ is concentrated in non-negative degrees.
Then
$$
H^i_Z(X,n)=0\quad \text{for all $i<2\codim_X Z$.}
$$ 
\end{lemma}
\begin{proof}
By \eqref{eq:RGamma_Z}, we have
$$
H^i_Z(X,n)\cong
H^i(Z_\proet, i^! ( \pi_{\comp}^\ast K \otimes _{\widehat \Z_\ell} \widehat \Z_{\ell}(n))) ,
$$
where $\pi\colon X\to \Spec(k)$ is the structure morphism.

First assume that $Z$ is smooth.
By \'etale excision (see Lemma \ref{lem:excision}), we may assume that $Z$ is equi-dimensional of pure codimension $c$ in $X$.
By Lemma \ref{lem:purity}, we have $i^\ast_{\comp}(-c)[-2c]\stackrel{\lowcong}\to i^!$, and so
$$
H^i_Z(X,n)\cong
H^{i-2c}(Z_\proet, (\pi\circ i)_{\comp}^\ast K \otimes _{\widehat \Z_\ell} \widehat \Z_{\ell}(n-c)) .
$$
Since $K$ is concentrated in non-negative degrees, the above hyper-cohomology group vanishes for $i-2c<0$.
This concludes the case where $Z$ is smooth.

For the general case, note that  by topological invariance (see Lemma \ref{lem:top-invariance}), we may assume that $k$ is perfect.
Hence, $Z$ contains an open subset that is smooth over $k$.
Using this and the fact that we know the result for $Z$ smooth, we can prove the claim with the help of Lemma \ref{lem:les-pairs} by induction on $\dim Z$.
This concludes the proof.
\end{proof}

\begin{lemma}  \label{lem:purity-les}
Let $X$ be a smooth algebraic scheme,  and let $i:Z\hookrightarrow X$ be a closed subset. 
Let further $W\subset X$ be a closed subset with complement $U=X\setminus W$ such that $Z\cap W$ has codimension at least $j$ in $X$.
If the complex $K$ is concentrated in non-negative degrees, 
then the restriction map
$
H^i_Z(X,n)\to H^i_{Z\cap U}(U,n)
$
is an isomorphism for all $i<2j-1$.
\end{lemma}
\begin{proof}  
Excision (see Lemma \ref{lem:excision}) reduces us to the case where $W\subset Z$. 
By semi-purity (see Lemma \ref{lem:semi-purity}), $H^i_W(X,n)=0$ for all $i<2j$, and so the result follows from the long exact sequence in Lemma \ref{lem:les-pairs}.
\end{proof}

\subsection{Cup products} \label{subsubsec:cup}

The pro-\'etale topos $\Sh(X_\proet)$ of a scheme $X$ is locally weakly contractible (see \cite[Definition 3.2.1 and Proposition 4.2.8]{BS}), which means in particular that it is locally coherent.
From now on we assume that $X$ is an algebraic scheme, and so local coherence implies that $X$ is coherent.
By \cite[Proposition VI.9.0]{SGA4.2}, it follows that $\Sh(X_\proet)$ has enough points.
Hence, any $F\in \Mod(X_\proet,\widehat \Z_\ell)$ admits a Godement resolution.
More generally,  for any bounded below complex $K\in D^+(X_\proet,\widehat \Z_\ell)$, there is a natural quasi-isomorphism $K\to \mathcal G^\bullet (K)$, given by the   simple complex of the Godement double complex  associated to $K$.
The sheaves in the complex $\mathcal G^\bullet (K)$ are $\Gamma_Z$-acyclic, and so
  $\RR \Gamma_Z(X_{\proet}, K)= \Gamma_Z(X_{\proet},\mathcal G^\bullet(K))$.

For any $F,G\in \Mod(X_\proet,\widehat \Z_\ell)$ and closed subschemes $Z,W\subset X$,  we have a canonical map
\begin{align} \label{eq:cup-Gamma}
\Gamma_{Z}(X,F)\otimes_{\Z_\ell}  \Gamma_W(X,G)\longrightarrow \Gamma_{Z\cap W}(X,F\otimes_{\widehat \Z_\ell}  G),
\end{align}
given by the tensor product of sections.  
This  induces  a map
\begin{align} \label{eq:cup-RGamma}
\RR \Gamma_Z(X_\proet,M)\otimes_{\Z_\ell}^{\mathbb L}\RR \Gamma_W(X_\proet,L)\longrightarrow \RR \Gamma_{Z\cap W}(X_\proet,M\otimes^{\mathbb L}_{\widehat \Z_\ell} L)
\end{align}
in the derived category for any constructible complexes $M,L\in D_{\cons}(X_\proet,\widehat \Z_\ell)$.
Using the above-mentioned Godement resolutions,  \eqref{eq:cup-RGamma} may be described as follows.
Recall that constructible complexes are bounded (see \cite[Lemma 6.5.3]{BS}) and that there are enough $\widehat \Z_\ell$-flat objects in $\Mod(X_\proet, \widehat \Z_\ell)$ (the latter is a general fact that holds for any ringed site; see \cite[\href{https://stacks.math.columbia.edu/tag/05NI}{Tag 05NI}]{stacks-project}). 
We may therefore choose a quasi-isomorphism $\mathcal F^\bullet(M)\to M$, where $\mathcal F^\bullet (M)$ is a complex of $\widehat \Z_\ell$-flat sheaves; see  \cite[\href{https://stacks.math.columbia.edu/tag/05T7}{Tag 05T7}]{stacks-project}. 
Then the Godement resolution $\mathcal F^\bullet(M)\to \mathcal G^\bullet(\mathcal F^\bullet(M))$ is still a complex of flat $\widehat \Z_\ell$-modules.
Hence,   $\mathcal G^\bullet(\mathcal F^\bullet(M))\otimes_{\widehat \Z_\ell} \mathcal G^\bullet(L)$ represents the derived tensor product $M\otimes^{\mathbb L}_{\widehat \Z_\ell} L$.
Applying the Godement resolution once again, we get a quasi-isomorphism 
$$
\mathcal G^\bullet(\mathcal F^\bullet (M))\otimes_{\widehat \Z_\ell} \mathcal G^\bullet(L)\longrightarrow \mathcal G^\bullet(\mathcal G^\bullet(\mathcal F^\bullet (M))\otimes_{\widehat \Z_\ell} \mathcal G^\bullet(L)).
$$
The cup product \eqref{eq:cup-RGamma} may then explicitly be described by the natural map 
$$
\Gamma_Z(X_\proet,\mathcal G^\bullet (\mathcal F^\bullet (M)))\otimes_{\Z_\ell}   \Gamma_W(X_\proet,\mathcal G^\bullet (L))\longrightarrow   \Gamma_{Z\cap W}(X_\proet,\mathcal G^\bullet (\mathcal G^\bullet (\mathcal F^\bullet (M))\otimes_{\widehat \Z_\ell} \mathcal G^\bullet (L))) .
$$ 
In the following arguments, the above explicit description of \eqref{eq:cup-RGamma} will be used throughout.

We denote from now on by 
\begin{align} \label{eq:H_cont(X,Z_ell(n))}
  H^i_Z(X,\Z_\ell(n)):=\RR^i \Gamma_Z(X_\proet,\pi_X^\ast \widehat \Z_\ell(n))\cong H^i_{Z,\cont}(X_\et,(\mu_{\ell^r}^{\otimes n})_r)
\end{align}
twisted $\ell$-adic pro-\'etale cohomology with support, which coincides with continuous $\ell$-adic \'etale cohomology with support by Lemma \ref{lem:BS-H_cont=H_proet}.
As before we write $H^i_Z(X,n):=H^i_{Z}(X_\proet, \pi_{\comp}^\ast K \otimes _{\widehat \Z_\ell} \widehat \Z_{\ell}(n))$.

\begin{lemma}[Cup products] \label{lem:cup}
Let $X$ be an algebraic $k$-scheme, and let $Z,W\subset X$ be closed subsets. 
Then there are cup product maps
$$
H^i_Z(X,\Z_\ell(n))\otimes_{\Z_\ell}  H^j_W(X,m)\longrightarrow H^{i+j}_{Z\cap W}(X, n+m) ,\quad \alpha\otimes \beta\longmapsto \alpha\cup \beta
$$
which are functorial with respect to pullbacks.
\end{lemma}
\begin{proof}
We apply  \eqref{eq:cup-RGamma} to $M=\widehat \Z_\ell(m)$ and $L=K\otimes_{\widehat \Z_\ell} \widehat \Z_\ell(n)$.
The lemma thus follows by taking cohomology because
$$
\widehat \Z_{\ell}(n)
\otimes^{\mathbb L}_{\widehat \Z_\ell}
(\pi_{\comp}^\ast K \otimes _{\widehat \Z_\ell} \widehat \Z_{\ell}(m)) \cong \pi_{\comp}^\ast K \otimes _{\widehat \Z_\ell} \widehat \Z_{\ell}(n+m) 
$$
since $\widehat \Z_{\ell}(n)$ is a locally free $\Z_\ell$-module (of rank 1). 
\end{proof}

\begin{lemma}  \label{lem:cup-enlarge-support}
Let $X$ be an algebraic $k$-scheme, and let $Z,Z',W,W'\subset X$ be closed with $Z\subset Z'$ and $W\subset W'$.
Then the cup product maps from Lemma \ref{lem:cup} fit into a commutative diagram
$$
\xymatrix{
H^i_Z(X,\Z_\ell(n))\otimes_{\Z_\ell}  H^j_W(X,m) \ar[r]^-{\cup}\ar[d] & H^{i+j}_{Z\cap W}(X, n+m) \ar[d] \\
H^i_{Z'}(X,\Z_\ell(n))\otimes_{\Z_\ell}  H^j_{W'}(X ,m) \ar[r]^-{\cup} & H^{i+j}_{Z'\cap W'}(X, n+m) ,
}
$$
where the horizontal maps are the natural maps given by the pushforward along the identity on $X$. 
\end{lemma}
\begin{proof}
Recall that the pushforward with respect to the identity corresponds to the map induced by the natural inclusion $\Gamma_Z(X,-)\hookrightarrow \Gamma_{Z'}(X,-)$; see Remark \ref{rem:pushforward-id}.
Clearly, the product map \eqref{eq:cup-Gamma} fits into a similar commutative diagram.
Passing to the derived category,  for bounded below complexes of sheaves $M,L\in D(X_\proet,\widehat \Z_\ell)$, we get a commutative diagram
$$
\xymatrix{
\RR \Gamma_Z(X_\proet, M) \otimes_{\Z_\ell}^{\mathbb L} \RR \Gamma_W(X_\proet, L)\ar[r]\ar[d]&\RR \Gamma_{Z\cap W}(X_\proet, M\otimes_{\widehat \Z_\ell}^{\mathbb L} L)\ar[d] \\
\RR \Gamma_{Z'}(X_\proet, M) \otimes_{\Z_\ell}^{\mathbb L} \RR \Gamma_{W'}(X_\proet, L)\ar[r] & \RR \Gamma_{Z'\cap W'}(X_\proet, M\otimes_{\widehat \Z_\ell}^{\mathbb L} L).
}
$$
Applying this to $M=\widehat \Z_\ell(m)$ and $L=K\otimes_{\widehat \Z_\ell} \widehat \Z_\ell(n)$, we obtain the lemma by taking cohomology.
\end{proof}

\begin{lemma} \label{lem:trace-tensor}
Let $f\colon X\to Y$ be a proper morphism between smooth equi-dimensional algebraic $k$-schemes.  
Assume that $f$ is either a closed immersion or a smooth map.
Let $M,L\in D_{\cons}(Y_\proet,\widehat \Z_\ell)$.
Then the following diagram is commutative: 
\begin{align} \label{diag:psi}
\xymatrix{
(\RR f_\ast f^!   M)\otimes_{\widehat \Z_\ell}^{\mathbb L}  L  \ar[r]^-{\psi}_-{\cong} \ar[d]^{\Tr_f\otimes \id} & \RR f_\ast f^!   (M\otimes_{\widehat \Z_\ell}^{\mathbb L}   L) \ar[d]^{\Tr_f}  \\
  M\otimes_{\widehat \Z_\ell}^{\mathbb L}   L \ar[r]^-{=} &  M\otimes_{\widehat \Z_\ell}^{\mathbb L}  L ,
}
\end{align}
where $\psi$ is a  natural isomorphism $($described in the proof\,$)$. 
\end{lemma}

\begin{proof}
Recall that the derived tensor product of $\ell$-adically complete complexes is again $\ell$-adically complete; see \cite[Lemma 6.5.5]{BS}.
Since $f$ is proper,  the projection formula in \cite[Lemma 6.7.14]{BS} implies that the canonical map
\begin{align} \label{eq:diag:projection-formula-2}
\RR f_\ast f^!  M\otimes_{\widehat \Z_\ell}^{\mathbb L}  L \stackrel{\lowcong}\longrightarrow \RR f_\ast ( f^!  M\otimes_{\widehat \Z_\ell}^{\mathbb L} f^\ast_{\comp}   L ) 
\end{align}
is an isomorphism;  see also  \cite[Proposition~XVII.5.2.9]{SGA4.3}.
By Lemmas \ref{lem:PD} and \ref{lem:purity},  our assumptions imply that there is a natural isomorphism  $f^!\cong f^\ast(d)[2d]$, where $d:=\dim X-\dim Y$.
Using this, we find that the above isomorphism induces a natural isomorphism
$$
\psi\colon \RR f_\ast f^!  M\otimes_{\widehat \Z_\ell}^{\mathbb L}  L \stackrel{\lowcong}\longrightarrow \RR f_\ast  f^!  ( M\otimes_{\widehat \Z_\ell}^{\mathbb L}    L)  ,
$$
where we used that $f^\ast$ commutes with tensor products.
It remains to see that  \eqref{diag:psi}  is commutative.

First assume  that $f$ is a closed immersion.
In this case, $\RR f_\ast f^!$ is the derived functor associated to the functor $\mathcal F\mapsto \mathcal H^0_X(Y,\mathcal F)$ that maps a sheaf $\mathcal F\in \Mod(X_\proet, \widehat \Z_\ell)$ to the subsheaf of  local sections with support contained in $X$.
Moreover, on the non-derived level, the trace map corresponds  to the natural inclusion $\mathcal H^0_X(Y,\mathcal F)\to \mathcal F$.
Using this, we find that \eqref{diag:psi} arises by deriving the  diagram
$$
\xymatrix{
 \mathcal H^0_X(Y,\mathcal F)\otimes_{\widehat \Z_\ell} \mathcal G \ar[r]\ar[d] &  \mathcal H^0_X(Y,\mathcal F\otimes_{\widehat \Z_\ell} \mathcal G )\ar[d]\\
 \mathcal F\otimes_{\widehat \Z_\ell} \mathcal G\ar[r]^{=}&  \mathcal F\otimes_{\widehat \Z_\ell} \mathcal G ,
}
$$
where the upper horizontal arrow is given by noting that  if $s$ is a local section of $\mathcal F$ whose support is contained in $X$ and $t$ is a local section of $\mathcal G$, then the local section $s\otimes t$ of $\mathcal F\otimes_{\widehat \Z_\ell} \mathcal G$ has support contained in~$X$.
The above diagram is clearly commutative, and so is the derived version in \eqref{diag:psi}.

Now assume  that $f$ is smooth.
Here it is convenient to apply further reduction steps first.
In fact, one can easily formally reduce the problem  to  the case where $M=\widehat \Z_\ell$ (by applying the result to various choices for $L$ and comparing the results).
Applying the canonical truncation triangle $\tau_{\leq j}L \to L \to \tau_{\geq j+1} L$, by induction on the length of $L$, we further  reduce  to the case where $L=\mathcal G[0]$ is a sheaf $\mathcal G\in \Mod(X_\proet,\widehat \Z_\ell)$ placed in degree zero.
The complexes
$$
(\RR f_\ast f^!   M)\otimes_{\widehat \Z_\ell}^{\mathbb L}  L =(\RR f_\ast f^!   \widehat \Z_\ell)\otimes_{\widehat \Z_\ell}^{\mathbb L}  \mathcal G 
$$
and
$$
 \RR f_\ast f^!   (M\otimes_{\widehat \Z_\ell}^{\mathbb L}   L)= \RR f_\ast f^!  \mathcal G
$$
are then concentrated in non-positive degrees, while $ M\otimes_{\widehat \Z_\ell}^{\mathbb L}   L=\mathcal G$ is concentrated in degree zero.
It follows that the vertical arrows in 
\eqref{diag:psi}  factor through the respective $\supth{0}$ cohomology sheaves, and so it suffices to prove the commutativity of the following diagram:
\begin{align} \label{diag:psi-2}
\xymatrix{
 \RR^{2d} f_\ast \widehat \Z_\ell(d) \otimes_{\widehat \Z_\ell}  \mathcal G  \ar[r] \ar[d]^{\Tr_f\otimes \id} & \RR^{2d} f_\ast \mathcal G(d)  \ar[d]^{\Tr_f}  \\
 \widehat  \Z_\ell  \otimes_{\widehat \Z_\ell} \mathcal G\ar[r]^-{=} & \mathcal G .
}
\end{align}
This follows directly from the construction of the trace map; see \textit{e.g.} \cite[Equation~(6.10)]{Sch-refined} and \cite[Step~(2) in the proof of Proposition~3.1]{verdier}.
This concludes the proof of the lemma. 
\end{proof}

\begin{lemma}[Projection formula] \label{lem:projection-formula}
Let $f\colon X\to Y$ be a proper morphism between smooth equi-dimensional $k$-schemes of relative codimension $c=\dim Y-\dim X$. 
Let $Z_Y,W_Y\subset Y$ be closed, and let $Z_X=f^{-1}(Z_Y)$ and $W_X=f^{-1}(W_Y)$.
Assume that $f$ is either a closed immersion or a smooth map.
Then the following hold:
\begin{enumerate}
\item \label{item:lem:proj-formula-1} For any $\alpha\in H^j_{W_X}(X,\Z_\ell(m))$ and $\beta\in H^i_{Z_Y}(Y,n)$,  we have
$$
f_\ast(\alpha \cup f^\ast \beta )=f_\ast \alpha\cup \beta \in H^{i+j+2c}_{Z_Y\cap W_Y}(Y,n+m+c) ,
$$
where $f_\ast \alpha \in H^{j+2c}_{W_Y}(Y,\Z_\ell(m+c))$ and $f^\ast \beta\in H^i_{Z_X}(X,n)$.
\item \label{item:lem:proj-formula-2} For any $\alpha\in H^i_{Z_Y}(Y,\Z_\ell(n))$ and $\beta\in H^j_{W_X}(X, m )$, we have
$$
f_\ast(f^\ast \alpha \cup  \beta )=\alpha\cup f_\ast \beta \in H^{i+j+2c}_{Z_Y\cap W_Y}(Y,n+m+c) ,
$$
where $f^\ast \alpha \in H^{i}_{Z_X}(X,\Z_\ell(n))$ and $f_\ast \beta\in H^{j+2c}_{W_Y}(Y,m+c)$.
\end{enumerate} 
\end{lemma}
\begin{proof}
Let $M,L\in D_{\cons}((\Spec(k))_\proet,\widehat \Z_\ell)$ be constructible complexes of  $\widehat \Z_\ell$-modules on the pro-\'etale site $(\Spec(k))_\proet$, and let $\pi_X\colon X\to \Spec(k)$ and $\pi_Y\colon Y\to \Spec(k)$ be the structure morphisms. 
Since $Z_X=f^{-1}(Z_Y)$,
$$
\RR \Gamma_{Z_X}(X, (\pi_X)_{\comp}^\ast M(-c)[-2c])=\RR \Gamma_{Z_Y}(Y, \RR f_\ast (\pi_X)_{\comp}^\ast M(-c)[-2c]) .
$$
Since $X$ and $Y$ are smooth, Lemmas \ref{lem:f!g!=(gf)!} and \ref{lem:PD} imply $ (\pi_X)_{\comp}^\ast M(-c)[-2c]\cong f^! (\pi_Y)^\ast_{\comp} M $ and so 
\begin{align} \label{eq:diag:projection-formula-0}
\RR \Gamma_{Z_X}(X, (\pi_X)_{\comp}^\ast M(-c)[-2c])=\RR \Gamma_{Z_Y}(Y,\RR f_\ast f^! (\pi_Y)^\ast_{\comp} M ) .
\end{align}
We then consider the  diagram
\begin{align} \label{eq:diag:projection-formula-1}
\xymatrix{
\RR \Gamma_{Z_Y}(Y,\RR f_\ast f^! (\pi_Y)^\ast_{\comp} M)\otimes^{\mathbb L}_{\Z_\ell}  \RR \Gamma_{W_Y}(Y, (\pi_Y)^\ast_{\comp} L) \ar[r]^-{\cup} \ar[d]^{\Tr_f\otimes \id} & \RR \Gamma_{Z_Y\cap W_Y}(Y,\RR f_\ast f^!  (\pi_Y)^\ast_{\comp} M\otimes_{\widehat \Z_\ell}^{\mathbb L}  (\pi_Y)^\ast_{\comp}  L ) \ar[d]^{\Tr_f\otimes \id}   \\
\RR \Gamma_{Z_Y}(Y, (\pi_Y)^\ast_{\comp} M)\otimes ^{\mathbb L}_{\Z_\ell} \RR \Gamma_{W_Y}(Y,(\pi_Y)^\ast_{\comp} L)\ar[r]^-{\cup} & \RR \Gamma_{Z_Y\cap W_Y}(Y, (\pi_Y)^\ast_{\comp}  M\otimes_{\widehat \Z_\ell}^{\mathbb L}  (\pi_Y)^\ast_{\comp}  L  ) ,
}
\end{align}
which is clearly commutative.

By Lemma \ref{lem:trace-tensor}, there is a commutative diagram
\begin{align*} % \label{eq:diag:projection-formula-3}
\xymatrix{
\RR f_\ast f^!  (\pi_Y)^\ast_{\comp} M\otimes_{\widehat \Z_\ell}^{\mathbb L}  (\pi_Y)^\ast_{\comp}  L  \ar[r]_-{\psi}^-{\cong} \ar[d]^{\Tr_f\otimes \id} & \RR f_\ast f^!   (\pi_Y)^\ast_{\comp}  (M\otimes_{\widehat \Z_\ell}^{\mathbb L}   L) \ar[d]^{\Tr_f}  \\
(\pi_Y)^\ast_{\comp}  M\otimes_{\widehat \Z_\ell}^{\mathbb L}  (\pi_Y)^\ast_{\comp}  L \ar[r]^-{\cong} &  (\pi_Y)^\ast_{\comp}  (M\otimes_{\widehat \Z_\ell}^{\mathbb L}  L ) .
}
\end{align*}
Applying $\RR \Gamma_{Z_{Y}\cap W_Y}(Y,-)$ to this and combining the resulting commutative square with \eqref{eq:diag:projection-formula-1}, we get a commutative diagram of the form
$$
\xymatrix{
\RR \Gamma_{Z_Y}(Y,\RR f_\ast f^! (\pi_Y)^\ast_{\comp} M)\otimes^{\mathbb L}_{\Z_\ell}  \RR \Gamma_{W_Y}(Y, (\pi_Y)^\ast_{\comp} L) \ar[r]^-{\cup} \ar[d]^{\Tr_f\otimes \id} & \RR \Gamma_{Z_Y\cap W_Y}(Y,\RR f_\ast f^!   (\pi_Y)^\ast_{\comp}  (M\otimes_{\widehat \Z_\ell}^{\mathbb L}   L) ) \ar[d]^{\Tr_f }   \\
\RR \Gamma_{Z_Y}(Y, (\pi_Y)^\ast_{\comp} M)\otimes ^{\mathbb L}_{\Z_\ell} \RR \Gamma_{W_Y}(Y,(\pi_Y)^\ast_{\comp} L)\ar[r]^-{\cup} & \RR \Gamma_{Z_Y\cap W_Y}(Y, (\pi_Y)^\ast_{\comp}  (M\otimes_{\widehat \Z_\ell}^{\mathbb L}  L ) ) .
}
$$
By \eqref{eq:diag:projection-formula-0} (applied to the complex $M\otimes_{\widehat \Z_\ell}^{\mathbb L}   L$ and the closed subset $Z_X\cap W_X\subset X$), the term in the right upper corner of the above diagram is given by
 $$
 \RR \Gamma_{Z_Y\cap W_Y}(Y,\RR f_\ast f^!   (\pi_Y)^\ast_{\comp}  (M\otimes_{\widehat \Z_\ell}^{\mathbb L}   L) )=\RR \Gamma_{Z_X\cap W_X}(X, (\pi_X)_{\comp}^\ast (M\otimes_{\widehat \Z_\ell}^{\mathbb L}   L)(-c)[-2c]). 
 $$
Item \eqref{item:lem:proj-formula-1} in the lemma therefore follows by applying the above argument to $M=  \widehat \Z_\ell(m)$ and $L=K\otimes_{\widehat \Z_\ell} \widehat \Z_\ell(n)$ and taking cohomology of the above commutative square. 
Item  \eqref{item:lem:proj-formula-2} follows similarly by applying the above argument to  $L= \widehat \Z_\ell(m)$ and $M=K\otimes_{\widehat \Z_\ell} \widehat \Z_\ell(n)$ and using the anti-commutativity of cup products.
This concludes the proof of the lemma.
\end{proof}

\subsection{Cycle class}

\begin{lemma}[Cycle class] \label{lem:cycle-class}
Let $X$ be an equi-dimensional smooth algebraic $k$-scheme, and let
$\Gamma\in Z^c(X)$ be a cycle of codimension $c$.  Let $Z\subset X$ be
any closed subscheme that contains the support of\, $\Gamma$.  Then
there is a cycle class
$$
\cl^X_Z(\Gamma)\in H^{2c}_Z(X,\Z_\ell(c)) 
$$
with the following properties:
\begin{enumerate}
\item If\, $W\subset X$ is closed with $ Z\subset W$, then $\cl_W^X(\Gamma)$ is the image of\, $\cl_Z^X(\Gamma)$ via the natural map $H^{2c}_Z(X,\Z_\ell(c))\to H^{2c}_W(X,\Z_\ell(c)) $.\label{item:lem:cycle-class:1}
\item If\, $\Gamma=\Gamma_1+\Gamma_2$ for cycles $\Gamma_i\in Z^c(X)$ with $\supp(\Gamma_i)\subset Z$, then $\cl_Z^X(\Gamma)=\cl_Z^X(\Gamma_1)+\cl_Z^X(\Gamma_2)$.\label{item:lem:cycle-class:2} 
\end{enumerate}
\end{lemma}
\begin{proof}
Recall from Lemma \ref{lem:BS-H_cont=H_proet} that $H^{2c}_Z(X,\Z_\ell(c)) =H^{2c}_{Z,\cont}(X_{\et},\Z_\ell(c)) $ agrees with continuous \'etale cohomology with values in the inverse system $(\mu_{\ell^r}^{\otimes c})_r$.
By the additivity of continuous \'etale cohomology, it suffices to deal with the case where $X$ is integral, hence a variety.
If $\Gamma=\sum a_i\Gamma_i$ with $\supp \Gamma_i \subset Z$,  we define
$$
\cl_Z^X(\Gamma):=\sum_i a_i\cdot \cl_Z^X(\Gamma_i)
$$
 and hence reduce the problem to the case where $\Gamma$ is a prime cycle, \textit{i.e.} a subvariety of codimension $1$. 
Let us first assume that $Z=\supp \Gamma$.
By \eqref{eq:H_cont(X,Z_ell(n))} and \cite[Theorem 3.23]{jannsen}, there is a canonical isomorphism 
\begin{align} \label{eq:H_Z=limH_Z(X)}
 H^{2c}_Z(X,\Z_\ell(c))\cong \lim_{\substack{\longleftarrow\\ r}}H^{2c}_Z(X,\mu_{\ell^r}^{\otimes c}), 
 \end{align}
 and the class $\cl^X_Z(\Gamma)$ is defined as limit of the cycle classes of $\Gamma$ in
 $
 H^{2c}_Z(X,\mu_{\ell^r}^{\otimes c})
 $
 from \cite[Section~2.2.10, p.\ 143]{SGA4.5},  where one uses the compatibility for different values of $r$; \textit{cf.}  \cite[Proof of Theorem~3.23, p.\ 221]{jannsen}. 
In general, if $Z$ is any closed subset with $|\Gamma|:=\supp \Gamma \subset Z$, then we define $\cl_Z^X(\Gamma)$ as the image of $\cl_{|\Gamma|}^X(\Gamma)$ via the natural pushforward map $H^{2c}_{|\Gamma|}(X,\Z_\ell(c))\to H^{2c}_Z(X,\Z_\ell(c)) $.
Item \eqref{item:lem:cycle-class:1} then follows from the functoriality of  pushforwards, while item \eqref{item:lem:cycle-class:2} follows directly from the construction together with the linearity of pushforwards.
 This proves the lemma.
\end{proof}

\begin{lemma} \label{lem:pullback-cycle-class}
Let $f\colon X\to Y$ be a morphism between smooth  equi-dimensional algebraic $k$-schemes.  
\begin{enumerate}
\item \label{item:pullback-cycle-class} Let $\Gamma\in Z^{i}(Y)$ be a cycle on $Y$ with support $Z_Y:=\supp(\Gamma)$.
Assume that $Z_X:=f^{-1}(Z_Y)\subset X$ has pure codimension $i$.
Then $f^\ast \cl^Y_{Z_Y}(\Gamma)= \cl^X_{Z_X}(f^\ast \Gamma)$. 
\item  \label{item:pushforward-cycle-class} Let $\Gamma\in Z^{i}(X)$ with support $Z_X:=\supp(\Gamma)$.
Assume that $f$ is proper of pure relative codimension $c$.
Set $Z_Y:=f(Z_X)$ with the reduced scheme structure. 
Then $f_\ast \cl^X_{Z_X}(\Gamma)=\cl^Y_{Z_Y}(f_\ast \Gamma)$. 
\end{enumerate}

\end{lemma}
\begin{proof}
To prove item \eqref{item:pullback-cycle-class}, we note that by the construction of the cycle class via the isomorphism \eqref{eq:H_Z=limH_Z(X)},  it suffices to prove the compatibility result for \'etale cohomology with coefficients in  $\mu_{\ell^r}^{\otimes i}$ with $r\geq 1$.
Hence,  item \eqref{item:pullback-cycle-class} follows from \cite[Th\'eor\`eme 2.3.8.(ii)]{SGA4.5}.

We prove  item \eqref{item:pushforward-cycle-class} next.
Lemma \ref{lem:top-invariance} reduces us to the case where the ground field $k$ is perfect.
By  Lemma \ref{lem:purity-les}, we may then assume that $Z_Y$ and $Z_X$ are smooth and pure-dimensional of pure codimensions $i$ and $i+c$ in $Y$ and $X$, respectively (this uses that any reduced scheme is generically regular, hence smooth because $k$ is perfect).
Lemma \ref{lem:purity} yields canonical isomorphisms
$$
H^{2i}_{Z_X}(X,\Z_\ell(i))\cong H^0(Z_X,\Z_\ell(0))\quad \text{and}\quad H^{2i+2c}_{Z_Y}(Y,\Z_\ell(i+c))\cong H^0(Z_Y,\Z_\ell(0)) .
$$
It follows from the construction of the pushforward $f_\ast$ that via the above isomorphism, it is compatible  with the pushforward 
$$
\left(f|_{Z_X}\right)_\ast\colon H^0(Z_X,\Z_\ell(0))\longrightarrow  H^0(Z_Y,\Z_\ell(0)) .
$$
Since $Z_X$ and $Z_Y$ are smooth, the above groups decompose according to the irreducible components of $Z_X$ and $Z_Y$, respectively.
Using this, we reduce to the case where  $Z_X$ and $Z_Y$ are integral, and we need to show that $(f|_{Z_X})_\ast$ is given by multiplication with $\deg(f|_{Z_X})$.
Up to shrinking $Z_Y$ further, we may assume that $f|_{Z_X}\colon Z_X\to Z_Y$ is \'etale (this uses that $k$ is perfect), and so the claim follows from Lemma \ref{lem:f_*circf*=degf}.
This proves item \eqref{item:pushforward-cycle-class} and concludes the proof of the lemma.
\end{proof}

\begin{lemma} \label{lem:cycle-class-cup}
Let $X$ be a smooth equi-dimensional algebraic scheme, and  for $i=1,2$, let  $\Gamma_i\in Z^{c_i}(X)$ be a cycle of codimension $c_i$ on $X$.
We put $Z_i:=\supp \Gamma_i$, which is closed of pure codimension $c_i$ in $X$. 
Assume that $Z_1$ and $Z_2$ meet dimensionally transversely, so that $Z:=Z_1\cap Z_2$ either is empty or has codimension $c=c_1+c_2$ in $X$.
Then  
$$
\cl^X_Z(\Gamma_1\cdot \Gamma_2)=\cl^X_{Z_1}(\Gamma_1)\cup \cl^X_{Z_2}(\Gamma_2)\in  H^{2c}_Z(X,\Z_\ell(c)) .
$$ 
\end{lemma}
\begin{proof} 
Since the cycle class in Jannsen's continuous \text{\'etale} cohomology from Lemma \ref{lem:cycle-class} is constructed  via the isomorphism \eqref{eq:H_Z=limH_Z(X)},  and because the cup product constructed in Lemma \ref{lem:cup} is compatible with reduction modulo $\ell^r$, it suffices to prove the lemma modulo $\ell^r$, \textit{i.e.} for \'etale cohomology with coefficients in $\mu_{\ell^r}^{\otimes c}$.
In this case, the   result in question follows from \cite[Th\'eor\`eme 2.3.8(iii) and Remarque 2.3.9]{SGA4.5} together with the fact that the multiplicities of the intersection product $\Gamma_1\cdot \Gamma_2$ may be computed by Serre's Tor formula; see \cite[Section~20.4 and Example 7.1.2]{fulton}.
This concludes the proof of the lemma. 
\end{proof}

\begin{lemma} \label{lem:del-vaprhi=cl}
Let $X$ be a smooth equi-dimensional algebraic scheme over a perfect field $k$.
Let $\Gamma\in Z^{c}(X )$ be a cycle with support $Z:=\supp(\Gamma)$.
Assume that there is a closed subset $W\subset X$ of codimension $c-1$  with $Z\subset W$ such that $\Gamma$, viewed as a cycle on $W$, is rationally equivalent to zero on $W$.
Then there is a class
$$
\varphi \in H^{2c-1}_{W\setminus Z}(X\setminus Z,\Z_\ell (c))\quad \text{with }\ \del \varphi=\cl_Z^{X}(\Gamma)\in H^c_Z(X,\Z_\ell (c)).
$$
\end{lemma}
\begin{proof}
We endow $W$ with the canonical reduced scheme structure and consider the normalization $\tau\colon W'\to W$.
The assumptions imply that there is a rational function $\varphi'$ on $W'$ such that $\Gamma=\tau_\ast \Div(\varphi')$.  

By Lemma \ref{lem:purity-les}, the groups in question do not change when we remove from $X$ a closed subset of codimension at least $c+1$.
Since $W'$ is regular in codimension $1$,  we may thus without loss of generality assume that $W'$ is regular, hence smooth.
For the same reason, we may assume that $Z$ as well as the preimage $Z'=\tau^{-1}(Z)$ are smooth.
We thus have a pushforward map 
$$
\tau_\ast\colon H^1(W'\setminus Z',\Z_\ell (1))\longrightarrow H^{2c-1}_W(X\setminus Z,\Z_\ell (c)),
$$
 given by \eqref{eq:f_*A(n)}.
Via Jannsen's Kummer sequence \cite[Equation~(3.27)]{jannsen}, the rational function $\varphi'$ yields a class $(\varphi') \in H^1(W'\setminus Z',\Z_\ell (1))$,  and we let 
$$
\varphi:=\tau_\ast (1/\varphi')\in H^{2c-1}_{W\setminus Z}(X\setminus Z,\Z_\ell (c)).
$$
The image of $(\varphi')$ via the residue map $ H^1(W'\setminus Z',\Z_\ell (1))\to H^2_{Z'}(W' ,\Z_\ell (1))\cong H^0(Z',\Z_\ell (0))$ coincides with the cycle given by $-\Div(\varphi')$ (indeed, it suffices to prove this modulo $\ell^r$, hence for \'etale cohomology with coefficients in $\mu_{\ell^r}$, where it follows from \cite[Lemme 2.3.6]{SGA4.5} and the anti-commutativity of the diagram in \cite[Section~2.1.3]{SGA4.5}).
Since by Lemma \ref{lem:les-pairs}, the pushforwards from \eqref{eq:f_*A(n)} induce  commutative ladders between the respective long exact sequences of triples, we find that 
$$
\del(\varphi)=
\del(\tau_\ast (1/\varphi'))=\tau_\ast(\del (1/\varphi'))=-\tau_\ast(\del (\varphi'))=\cl_Z^X(\tau_\ast \Div(\varphi'))=\cl_Z^X(\Gamma) .
$$ 
This concludes the proof. 
\end{proof}

\begin{lemma} \label{lem:cl=0}
In the notation of Lemma \ref{lem:del-vaprhi=cl}, we have 
 $\cl_W^X(\Gamma)=0\in  H^{2c}_W(X,\Z_\ell (c))$.
\end{lemma}
\begin{proof}
By  Lemma \ref{lem:pullback-cycle-class}, $\cl_W^X(\Gamma) $ is the image of $\cl_Z^X(\Gamma)$ via the natural map $H^{2c}_Z(X,\Z_\ell (c))\to H^{2c}_W(X,\Z_\ell (c))$.
The assertion in the lemma thus follows from the long exact sequence of triples (see Lemma \ref{lem:les-pairs}) together with
the fact that $ \del \varphi=\cl_Z^{X}(\Gamma)\in H^{2c}_Z(X,\Z_\ell (c))$.
\end{proof}

\subsection{Proof of Proposition  \ref{prop:pro-etale-coho}} \label{subsec:proof-prop-appendix}

\begin{proof}[Proof of Proposition \ref{prop:pro-etale-coho}] 
By Lemma \ref{lem:BS-H_cont=H_proet},  item  \eqref{item:prop:continuous-etale-coho} follows from item \eqref{item:prop:pro-etale-coho} applied to the complex $K$ given by the pro-\'etale sheaf $\lim \nu^\ast F_r$ placed in degree zero.
We further claim that item  \eqref{item:prop:etale-coho} follows from item \eqref{item:prop:continuous-etale-coho}.
To see this, note that \'etale cohomology commutes with direct limits in the coefficients; see \cite[\href{https://stacks.math.columbia.edu/tag/09YQ}{Tag 09YQ}]{stacks-project}.
It thus suffices to prove  item  \eqref{item:prop:etale-coho}  in the case where there is some integer $r$ so that $\mathcal F=\pi_X^\ast F$ is an $\ell^r$-torsion \'etale sheaf on   $X_\et$.
By \cite[Equation~(3.1)]{jannsen}, the result then follows  from  item \eqref{item:prop:continuous-etale-coho} applied to the constant inverse system $(\mathcal F)_r$. 

Altogether we have thus seen that it suffices to prove
 item \eqref{item:prop:pro-etale-coho}.
 To this end, we use the notation $H^i_Z(X,n):=H^\ast _{Z}(X_\proet, (\pi_{X})_{\comp}^\ast K \otimes _{\widehat \Z_\ell} \widehat \Z_{\ell}(n))$ from above.
The pullback maps from Lemma \ref{lem:f_*-f*} make $(X,Z)\mapsto H^i_Z(X,n)$ into a functor as we want.
If the complex $K$ is concentrated in non-negative degrees, then semi-purity (\textit{i.e.} condition~\ref{item:semi-purity}) follows from Lemma \ref{lem:semi-purity}.

Condition~\ref{item:excision} follows from Lemma \ref{lem:excision}, condition~\ref{item:f_*} follows from Lemmas \ref{lem:f_*-f*} and \ref{lem:pull-push-compatible}, and condition~\ref{item:les-triple} follows from Lemma \ref{lem:les-pairs}.
It thus remains to prove condition
\ref{item:action-of-cycles}.
The action of cycles is given by cup product and the cycle class map; see Lemmas \ref{lem:cup} and \ref{lem:cycle-class}.
Condition \ref{item:action-of-cycles-rationally-trivial} then follows from Lemma~\ref{lem:cl=0}.\looseness=-1

To prove condition~\ref{item:action-of-cycles:projection-formula}, let $\iota:W\hookrightarrow X$ be a closed embedding with $W$ smooth and irreducible of codimension $c$, and let $\Gamma=[W]\in Z^c(X)$.
Let  $\alpha\in H^i_Z(X,n)$, and note that $\cl_X^X(X)\in H^0(X,\Z_\ell(0))$ satisfies $\cl_X^X(X)\cup \alpha=\alpha$. 
The projection formula in Lemma \ref{lem:projection-formula} thus yields 
$$
\iota_\ast \iota^\ast \alpha=\iota_\ast (\iota^\ast \cl_X^X(X) \cup \iota^\ast \alpha)=(\iota_\ast \iota^\ast \cl_X^X(X))  \cup \alpha .
$$
By Lemma \ref{lem:pullback-cycle-class}, we have $\iota_\ast \iota^\ast \cl_X^X(X) =\iota_\ast \cl_W^W(W)=\cl_W^X(\Gamma)$, and so the above equation reads 
$
\iota_\ast \iota^\ast \alpha=\cl_W^X(\Gamma)\cup \alpha
$,
as we want. 

Condition~\ref{item:action-of-cycles:enlargening-the-support} follows  from Lemma \ref{lem:cup-enlarge-support}. 
Condition~\ref{item:action-of-cycles:compatible-with-restriction} follows from the functoriality of the cup product in Lemma \ref{lem:cup} with respect to pullbacks (hence with respect to restrictions), together with the compatibility of the cycle class map with flat pullback from Lemma \ref{lem:pullback-cycle-class}.

The formula $\cl_W^X(\Gamma)\cup  (\cl_{W'}^X(\Gamma')\cup  \alpha)=\cl^X_{W\cap W'}(\Gamma\cdot \Gamma')\cup \alpha $  in condition~\ref{item:action-of-cycles:compatible-with-intersections} follows from the associativity of cup products in \eqref{eq:cup-RGamma} together with the formula $\cl_W^X(\Gamma)\cup  \cl_{W'}^X(\Gamma')= \cl^X_{W\cap W'}(\Gamma\cdot \Gamma')$ from Lemma \ref{lem:cycle-class-cup}, where we use that $W=\supp \Gamma$ meets $W'=\supp \Gamma'$ properly by assumption.

It remains to prove condition~\ref{item:action-of-cycles:projection-formula-everything}.
If $f\colon X'\to X$ is flat, then the formula $f^\ast(\cl_W^X(\Gamma)\cup  \alpha)=\cl_{f^{-1}(W)}^X(f^\ast\Gamma)\cup f^\ast \alpha$ follows from the compatibility of cup products with pullbacks (see Lemma \ref{lem:cup}) together with the fact that $f^\ast \cl_W^X(\Gamma)=\cl_{f^{-1}(W)}^X(f^\ast\Gamma)$ by Lemma \ref{lem:pullback-cycle-class}.
Finally, if  $f\colon X'\to X$ is smooth and proper, the formula $f_\ast(\cl_{W'}^{X'}(\Gamma)\cup f^\ast \alpha)=\cl_{f(W')}^X(f_\ast\Gamma)\cup \alpha $ follows from Lemma  \ref{lem:projection-formula}\eqref{item:lem:proj-formula-1} together with Lemma \ref{lem:pullback-cycle-class}, while $ f_\ast(\cl_{W}^{X}(f^\ast \Gamma)\cup \alpha)=\cl_{W}^X(\Gamma)\cup f_\ast \alpha $ follows from Lemma  \ref{lem:projection-formula}\eqref{item:lem:proj-formula-2}  because $\cl_{W}^{X}(f^\ast \Gamma)=f^\ast\cl_{W}^{X}(\Gamma) $ by Lemma \ref{lem:pullback-cycle-class}.

This concludes the proof of the proposition.
\end{proof}

\section{Comparison with action on algebraic cycles} \label{ap:B}
 
\subsection{Comparison to the action on Chow groups} \label{subsec:comparison-act-on-Chow}  
Let $X$ be a  smooth projective equi-dimensional scheme over a perfect field $k$, and let $\ell$ be a prime that is invertible in $k$.
For $Z\subset X$ closed, we let as before
$$
H^i_Z(X,\Z_\ell(n)):=\RR^i\Gamma_Z(X_\proet,\widehat \Z_\ell(n))  
\quad \text{and}\quad   H^i(X,\Z_\ell(n)):=H^i_X(X,\Z_\ell(n)).
$$

The theory of refined unramified cohomology developed in \cite{Sch-refined} relies on an $\ell$-adic Borel--Moore cohomology theory $H^i_{\BM}(Z,\Z_\ell(n))$ for each algebraic $k$-scheme $Z$; see \cite[Section 4 and Proposition~6.6]{Sch-refined}. 
If $X$ is a equi-dimensional smooth algebraic $k$-scheme and  $Z\subset X$ is closed of codimension $c$ (not necessarily of pure dimension), then by \eqref{eq:RGamma_Z} and Lemmas \ref{lem:f!g!=(gf)!} and \ref{lem:PD}, the Borel--Moore cohomology of $Z$ is given by 
$$
H^i_{\BM}(Z,\Z_\ell(n))=H^{i+2c}_Z(X,\Z_\ell(n+c)).
$$
In particular, $H^i_{\BM}(X,\Z_\ell(n)=H^i(X,\Z_\ell(n))$ for $X$ smooth and equi-dimensional.

We let $\CH^i(X)_{\Z_\ell}:=\CH^i(X)\otimes_\Z \Z_\ell$.
The coniveau filtration $N^\ast$ on $\CH^i(X)_{\Z_\ell}$ is the decreasing filtration given by the condition that a class $[z]\in \CH^i(X)_{\Z_\ell}$ lies in $N^j\CH^i(X)_{\Z_\ell}$ if and only if $[z]$ can be represented by a cycle $z$ that is homologically trivial on a closed subset $W\subset X$ of codimension $j$; \textit{i.e.} $\supp z\subset W$ with $\cl_W^X(z)=0\in H^{2i}_W(X,\Z_\ell(i))$; see \cite{bloch-DMJ,jannsen-3} or \cite[Definition 7.3]{Sch-refined}. 
For instance, 
$ 
N^{0} \CH^i(X)_{\Z_\ell}
$ 
is the space of $\ell$-adic cycles with trivial cycle class in $H^{2i}(X,\Z_\ell(i))$.
Moreover, 
$ 
N^{i-1} \CH^i(X)_{\Z_\ell}
$ 
is zero if $k$ is finitely generated (or an inseparable extension thereof), and  it 
is the space of algebraically trivial $\ell$-adic cycles of codimension $i$ if $k$ is algebraically closed; see  \cite[Proposition 6.6 and Lemma 7.5]{Sch-refined} or \cite[Lemmas 5.7 and 5.8]{jannsen-3} for the corresponding rational statements.

We define
\begin{align*} % \label{eq:AiX}
A ^i(X)_{\Z_\ell}:=\CH^i(X)_{\Z_\ell}/N^{i-1}\CH^i(X)_{\Z_\ell}.
\end{align*}  
By what we have said above, this is the $\ell$-adic Chow group of algebraic cycles modulo rational equivalence if $k$ is finitely generated (or the perfect closure of such a field), while it is the group of $\ell$-adic cycles modulo algebraic equivalence if $k$ is algebraically closed.
We further let $A_0^i(X)_{\Z_\ell}\subset A ^i(X)_{\Z_\ell}$ be the subspace of cycles with trivial cycle class on $X$, \textit{i.e.}
\begin{align} \label{eq:A0iX}
A_0^i(X)_{\Z_\ell} =N^0\CH^i(X)_{\Z_\ell}/N^{i-1}\CH^i(X)_{\Z_\ell}.
\end{align}   
By  \cite[Lemma 7.4 and Proposition 7.11]{Sch-refined}, there is  a canonical isomorphism
\begin{align} \label{eq:Grifft=H_nr}
 \frac{H^{2i-1}_{i-2,\nr}(X,\Z_\ell(i))}{H^{2i-1}(X,\Z_\ell(i))} \stackrel{\lowcong}\longrightarrow \Grifft ^i(X)_{\Z_\ell} .
\end{align}  
 
By Corollary \ref{cor:motivic-body}\eqref{item:thm:motivic-body:restriction}, the action of correspondences on refined unramified cohomology  descends to an action on the left-hand side of \eqref{eq:Grifft=H_nr}, where we use $H^{2i-1}_{m,\nr}(X,\Z_\ell(i))=H^{2i-1}(X,\Z_\ell(i))$ for $m\geq \dim X $.  
There is also a natural action on the right-hand side of \eqref{eq:Grifft=H_nr}.
The main result of this appendix shows that both actions agree with each other.

\begin{proposition} \label{prop:Gamma-ast-compat-Chow}
Let $X$ and $Y$ be  smooth projective equi-dimensional  schemes over a perfect  field $k$, let $d_X:=\dim (X)$, and let $\ell$ be a prime that is invertible in $k$.
Then the bilinear pairing 
$$
\CH^{c}(X\times Y)\times \frac{H^{2i-1}_{i-2,\nr}(X,\Z_\ell(i))}{H^{2i-1}(X,\Z_\ell(i))}\longrightarrow \frac{H^{2(i+c-d_X)-1}_{i+c-d_X-2,\nr}(Y,\Z_\ell(i+c-d_X))}{H^{2(i+c-d_X)-1}(Y,\Z_\ell(i+c-d_X))}
$$
induced by the pairing on refined unramified cohomology in Corollary \ref{cor:motivic-body} agrees via the isomorphism in \eqref{eq:Grifft=H_nr} with the natural action on $A_0^i(X)_{\Z_\ell} =N^0\CH^i(X)_{\Z_\ell}/N^{i-1}\CH^i(X)_{\Z_\ell}$. 
\end{proposition}

Proposition~\ref{prop:Gamma-ast-compat-Chow} relies on the results recalled in Appendix \ref{app:A}, together with the following result (we will only need the special case where $W=U=X$, but we state and prove the more general version below).

\begin{lemma}[Compatibility of cup products with residue maps] \label{lem:les-pairs-cup} 
Let $X$ be an algebraic $k$-scheme.
Let $Z\subset W\subset X$ and $Y\subset X$ be closed subsets.
Let $U\subset X$ be an open subset with $W\cap Y\subset U$. 
 
Then for any $\beta \in H^j_{Y}(U,\Z_\ell(m))$, the following diagram is commutative:
$$
\xymatrix{
 H^i_{W }(U\setminus Z,n)\ar[r]^-{\del} \ar[d]^{\cup \beta|_{U\setminus Z}} & H^{i+1}_{Z}(U,n)\ar[d]^{\cup \beta} \\
 H^{i+j}_{W\cap Y}(U\setminus Z, n+m ) & H^{i+j+1}_{Z\cap Y}(U, n+m ) \\
  H^{i+j}_{W\cap Y }(X\setminus (Z\cap Y), n+m )\ar[u]^{\cong}\ar[r]^-{\del} & H^{i+j+1}_{Z\cap Y}(X, n+m ) \ar[u]^{\cong} ,
}
$$
where the cohomology groups in question are those from Proposition \ref{prop:pro-etale-coho}\,\eqref{item:prop:pro-etale-coho}  $($and the notation $H^i_Z(U,n):=H^i_{Z\cap U}(U,n)$ etc.\ is used\,$)$, the horizontal maps are parts of the long exact sequence in Lemma \ref{lem:les-pairs}, the upper vertical maps are the cup product maps from Lemma \ref{lem:cup}, and the lower vertical maps are the isomorphisms given by pullback and excision $($where we use $W\cap Y\subset U$ and $Z\cap Y\subset U)$.
\end{lemma}
\begin{proof} 
We write $Z_Y:=Z\cap Y$ and $W_Y:=W\cap Y$.
For $U\subset X$ open, we will further use the notation $\RR \Gamma_Z(U,-):=\RR \Gamma_{Z\cap U}(U,-)$ and so on.
Then let $M,L\in D_{\cons}(X_{\proet},\widehat{\Z}_\ell )$, and consider the exact triangle
$$
\RR \Gamma_{Z}(U,M)\longrightarrow \RR \Gamma_{W}(U,M)\longrightarrow \RR \Gamma_{W\setminus Z}(U\setminus Z,M)
$$
in $D(\Mod_{\Z_\ell})$ from the proof of Lemma \ref{lem:les-pairs}, where by slight abuse of notation we do not distinguish between $M$ and the pullback to the respective open subsets above.

We may take the derived tensor product of the above triangle with the complex $\RR \Gamma_{Y}(U,L)$.
This gives rise to an exact triangle
\begin{align*}
\RR \Gamma_{Z}(U,M)\otimes_{\Z_\ell}^{\mathbb L} \RR \Gamma_{Y}(U,L)\longrightarrow \RR \Gamma_{W}&(U,M)\otimes_{\Z_\ell}^{\mathbb L} \RR \Gamma_{Y}(U,L)
\longrightarrow \RR \Gamma_{W\setminus Z}(U\setminus Z,M) \otimes ^{\mathbb L}_{\Z_\ell}  \RR \Gamma_{Y}(U,L).
\end{align*}
The maps in \eqref{eq:cup-RGamma} yield a map from this triangle 
to the sequence
\begin{align} \label{eq:lem:les-pairs-cup}
\RR \Gamma_{ Z_Y }(U,M\otimes_{\widehat \Z_\ell}^{\mathbb L} L)\longrightarrow \RR \Gamma_{W_Y}(U,M\otimes_{\widehat \Z_\ell}^{\mathbb L}  L)\longrightarrow \RR \Gamma_{ W_ Y  \setminus   Z_Y}(U\setminus Z,M\otimes_{\widehat \Z_\ell}^{\mathbb L}  L)
\end{align} 
such that the corresponding diagram commutes (as it commutes before taking derived functors).
Note that the sequence in \eqref{eq:lem:les-pairs-cup} is \textit{a priori} not an exact triangle.
However, the natural restriction maps yield a map from the  exact  triangle
$$
\RR \Gamma_{Z_Y }(X,M\otimes_{\widehat \Z_\ell}^{\mathbb L} L)\longrightarrow \RR \Gamma_{W_Y}(X,M\otimes_{\widehat \Z_\ell}^{\mathbb L}  L)\longrightarrow \RR \Gamma_{W_ Y\setminus Z_Y}(X\setminus Z_Y,M\otimes_{\widehat \Z_\ell}^{\mathbb L}  L)
$$
 to the sequence  in \eqref{eq:lem:les-pairs-cup} such that the corresponding diagram commutes. 
The lemma follows from this by setting $M:=\widehat \Z_\ell(m)$ and $L:=(\pi_X)^\ast_{\comp} K\otimes_{\widehat \Z_\ell} \widehat \Z_\ell(n)$ and taking cohomology, where we note that the restriction map from the above triangle to \eqref{eq:lem:les-pairs-cup} yields isomorphisms in cohomology by Lemma \ref{lem:excision} because $W_Y\subset W_U$ and $Z_Y\subset Z_U$.
(This shows in fact that \eqref{eq:lem:les-pairs-cup} is isomorphic to an exact triangle, hence is an exact triangle itself.) 
\end{proof}

\begin{proof}[Proof of Proposition \ref{prop:Gamma-ast-compat-Chow}]
Let $[\alpha]\in H^{2i-1}_{i-2,\nr}(X,\Z_\ell(i))$, and let $[\Gamma]\in \CH^c(X\times Y)$. 
Applying either Theorem \ref{thm:moving} to $\Gamma$ or Corollary \ref{cor:inj+purity-thm-body}\eqref{item:thm:purity-body} to $[\alpha]$ (\textit{cf.} the proof of Corollary \ref{cor:motivic-body}), we can assume that 
there is a representative $\alpha \in H^{2i-1}(U,\Z_\ell(i))$ for some open subset $U\subset X$ whose complement $R=X\setminus U$ is pure-dimensional of codimension $i$ and such that $R\times Y$ meets $W:=\supp \Gamma$ properly.
Let $S:=q( (R\times Y) \cap W)$ and $V=Y\setminus S$.
Then the class
$$
\Gamma(W)_\ast(\alpha)\in H^{2(i+c-d_X)-1} (V,\Z_\ell(i+c-d_X))
$$
from Lemma \ref{lem:Gamma-ast} represents $[\Gamma]_\ast [\alpha]\in H^{2(i+c-d_X)-1}_{i+c-d_X-2,\nr}(Y,\Z_\ell(i+c-d_X))$. 
This yields via \eqref{eq:Grifft=H_nr} an ($\ell$-adic)  cycle on $Y$, and we aim to show that this cycle is $[\Gamma]_\ast[z]$, 
where $[z]\in A^i_0(X)_{\Z_\ell}$ is the class represented by $[\alpha]$ via \eqref{eq:Grifft=H_nr}.

To begin with, we aim to describe the cycle $z$ on $X$ explicitly.
To this end,  let
$$
\del \alpha\in H^{2i}_{R}(X,\Z_\ell(i)) ,
$$
where $\del$ denotes the residue map from Lemma \ref{lem:les-pairs}.
Since $k$ is perfect, $R$ is generically smooth, and so its singular locus $R^{\sing}$ has codimension at least $ i+1$ in $X$.
Hence, Lemma \ref{lem:purity-les} implies that the natural map
$$
H^{2i}_{R}(X,\Z_\ell(i))\stackrel{\lowcong} \longrightarrow H^{2i}_{ R^{\sm}}(X\setminus R^{\sing},\Z_\ell(i))
$$
is an isomorphism, where $R^{\sm}=R\setminus R^{\sing}$.
By purity (see Lemma \ref{lem:purity}), there is a natural isomorphism
$$
H^{2i}_{R^{\sm}}(X\setminus R^{\sing},\Z_\ell(i))\cong H^{0}(R^{\sm},\Z_\ell(0)) .
$$
Combining the above isomorphisms, we see that the  natural map
\begin{align} \label{eq:iso:appendixB}
H^{2i}_{R}(X,\Z_\ell(i))\longrightarrow 
\bigoplus_{x\in  R ^{(0)}}H^0(x,\Z_\ell)=\bigoplus_{x\in R^{(0)}} [x]\Z_\ell
\end{align}
given by pullback to $X\setminus R^{\sing}$ and purity is an isomorphism.
The image of $\del \alpha$ via \eqref{eq:iso:appendixB} is a cycle $z \in Z^i(X)_{\Z_\ell}$ with coefficients in $\Z_\ell$ whose support is given by some components of $R$.
In fact,  we get 
\begin{align} \label{eq:Theta}
\cl_{R}^X(z)=\del \alpha \in H^{2i}_{R}(X,\Z_\ell(i)) ,
\end{align}  
which defines $z$ uniquely  because \eqref{eq:iso:appendixB} is an isomorphism.
It follows directly from the construction of the map in  \eqref{eq:Grifft=H_nr} (see \cite[Proposition 7.11]{Sch-refined}) that via  \eqref{eq:Grifft=H_nr}, the class $[z]\in A^i_0(X)_{\Z_\ell}$ is  represented by 
$$
[\alpha]\in H^{2i-1}_{i-2,\nr}(X,\Z_\ell(i))/H^{2i-1} X,\Z_\ell(i)).
$$

Similarly, the unramified class  $[\Gamma]_\ast [\alpha]\in H^{2(i+c-d_X)-1}_{i+c-d_X-2,\nr}(Y,\Z_\ell(i+c-d_X))$ corresponds to the cycle $z'$ with $\supp z'\subset S$ on $Y$,  which is uniquely determined by
\begin{align} \label{eq:Omega}
\cl_S^Y(z')=\del ( \Gamma(W)_\ast(\alpha)) \in H^{2(i+c-d_X)}_S (Y,\Z_\ell(i+c-d_X)).
\end{align}

By the construction of $\Gamma(W)_\ast(\alpha)$, we have
$$
\Gamma(W)_\ast(\alpha)=q_\ast( \exc( \cl_{W}^{X\times Y}(\Gamma)\cup p^\ast \alpha)) ,
$$
where $\exc\colon H^{2i-1+2c}_{W}(U\times Y, \Z_\ell(i+c))\stackrel{\lowcong}\to H^{2i-1+2c}_{W}((X\times Y)\setminus W_R, \Z_\ell(i+c)) $ with $W_R:=W\cap R\times Y$ is the isomorphism given by excision.
By Lemma \ref{lem:les-pairs}\eqref{item:les:pushforward}, we find
$$
\del (\Gamma(W)_\ast(\alpha))=q_\ast\left( \del \left(  \exc\left(\cl_{W}^{X\times Y}(\Gamma)\cup p^\ast \alpha\right) \right) \right).
$$
Since
$
\cl_W^{X\times Y}(\Gamma)\in H^{2c}_{W }(X\times Y ,\Z_\ell(c)) ,
$
 Lemma \ref{lem:les-pairs-cup} implies that 
$$
\del(\exc (\cl_W^{X\times Y}(\Gamma) \cup p^\ast \alpha))=\cl_W^{X\times Y}(\Gamma) \cup \del(p^\ast \alpha) \in H^{2(i+c)}_{W_{R}}(X\times Y,\Z_\ell(i+c)),
$$
where $W_{R}=W\cap (R\times Y)$.
By Lemma \ref{lem:les-pairs}\eqref{item:les:pullback}, 
$$
\del(p^\ast \alpha) =p^\ast (\del \alpha)=p^\ast \cl_{R}^X(z) \in H^{2i}_{R\times Y}(X\times Y,\Z_\ell(i)),
$$
where $\cl_{R}^X(z)=\del \alpha$ is from \eqref{eq:Theta}.
By Lemma \ref{lem:pullback-cycle-class}, $p^\ast \cl_{R}^X(z)= \cl_{R\times Y}^{X\times Y}(p^\ast z)$ and so
$$
\del(\cl_W^{X\times Y}(\Gamma) \cup p^\ast \alpha )=\cl_W^{X\times Y}(\Gamma) \cup \cl_{R\times Y}^{X\times Y}(p^\ast\ z)\in H^{2(i+c)}_{W_{R}}(X\times Y,\Z_\ell(i+c)).
$$
By Lemma \ref{lem:cycle-class-cup},
$$
\cl_W^{X\times Y}(\Gamma) \cup \cl_{R\times Y}^X(p^\ast z)=\cl_{W_{R}}^{X\times Y}(\Gamma\cdot p^\ast z)\in H^{2(i+c)}_{W_{R}}(X\times Y,\Z_\ell(i+c)).
$$
Since $W$ and $R\times Y$ meet properly by assumption,  $W_{R}$ has codimension $c+i$, and the natural restriction map together with purity (see Lemma \ref{lem:purity}) yield a map
$$
H^{2(i+c)}_{W_{R}}(X\times Y,\Z_\ell(i+c))\longrightarrow \bigoplus_{x\in W_{R}^{(0)}} H^0(x,\Z_\ell)=\bigoplus_{x\in W_{R}^{(0)}}[x]\cdot \Z_\ell
$$
which, by Lemma \ref{lem:purity-les}, is an isomorphism as before.
The image of $\cl_{W_{R}}^{X\times Y}(\Gamma\cdot p^\ast\ z)$ via this map is given by the cycle $\Gamma\cdot p^\ast z$, where multiplicities are computed via Serre's Tor formula; see \cite[Section~20.4]{fulton}.
The cycle 
$$
z' \in \bigoplus_{x\in S_{(d_Y+d_X-i-c)}}[x]\cdot \Z_\ell
$$
from \eqref{eq:Omega} is thus given by the pushforward of  $\Gamma\cdot p^\ast z$ via $q\colon z'=q_\ast(\Gamma\cdot p^\ast z) $, 
which proves the proposition.
\end{proof}

\subsection{Transcendental Abel--Jacobi maps are motivic}

Let $X$ be a smooth equi-dimensional scheme over a perfect field $k$, and let $\ell$ be a prime invertible in $k$.
We let $H^i(X,\Q_\ell(n)):=H^i(X_\proet,\widehat \Z_\ell(n)) \otimes_{\Z_\ell} \Q_\ell$ and $H^i(X,\Q_\ell/\Z_\ell(n))=\colim_r H^i(X_\proet,\nu^\ast \mu_{\ell^r}^{\otimes n})$, where we note that $H^i(X_\proet,\nu^\ast \mu_{\ell^r}^{\otimes n})\cong H^i(X_\et, \mu_{\ell^r}^{\otimes n})$; \textit{cf.}  Lemma \ref{lem:BS-H_cont=H_proet}.

Recall the $\ell$-adic cycle group $A_0^i(X)_{\Z_\ell}$ from \eqref{eq:A0iX}.
We denote by $A_0^i(X)[\ell^\infty]$ the torsion subgroup of $A_0^i(X)_{\Z_\ell}$.
By \cite[Section~7.5]{Sch-refined}, there is a transcendental Abel--Jacobi map on torsion cycles
\begin{equation}\label{**}
  \lambda_{\tr}^i\colon A_0^i(X)[\ell^\infty]\longrightarrow H^{2i-1}(X,\Q_\ell/\Z_\ell(i))/N^{i-1}H^{2i-1}(X,\Q_\ell (i)),
\end{equation}
where $N^\ast$ denotes Grothendieck's coniveau filtration on cohomology; \textit{i.e.} $\alpha\in H^{2i-1}(X,\Q_\ell (i))$ lies in $N^jH^{2i-1}(X,\Q_\ell (i))$ if $\alpha$ vanishes on the complement of a closed codimension $j$ set of $X$.
Correspondences between smooth projective equi-dimensional $k$-schemes act on both sides of \eqref{**}, and the main result in this section is that these actions are compatible with the map $\lambda_{tr}^i$.

\begin{corollary}
Let $X$ and $Y$ be  smooth projective equi-dimensional  schemes over a perfect  field $k$, let $d_X:=\dim (X)$, and let $\ell$ be a prime that is invertible in $k$.
Let $[\Gamma]\in \CH^c(X\times Y)$ be a correspondence.
Then the following diagram is commutative:
$$
\xymatrix{
A_0^i(X)[\ell^\infty]\ar[rrr]^-{\lambda_{tr}^i} \ar[d]^{[\Gamma]_\ast} & & & \frac{H^{2i-1}(X,\Q_\ell/\Z_\ell(i))}{N^{i-1}H^{2i-1}(X,\Q_\ell (i))} \ar[d]^{[\Gamma]_\ast} \\
A_0^{i+c-d_X}(Y)[\ell^\infty]\ar[rrr]^-{\lambda_{tr}^{i+c-d_X}}& & & \frac{H^{2i+2c-2d_X-1}(Y,\Q_\ell/\Z_\ell(i+c-d_X))}{N^{i+c-d_X-1}H^{2i+2c-2d_X-1}(Y,\Q_\ell (i+c-d_X))} .
}
$$
\end{corollary}
\begin{proof}
By \cite[Lemma 7.15]{Sch-refined},  the map $\lambda_{tr}^i$ can be described as follows.
Let $[z]\in A_0^i(X)[\ell^\infty]$, and let $[\alpha]\in H^{2i-1}_{i-2,\nr}(X,\Z_\ell(i))$ be a class that represents $[z]$ via the isomorphism in \eqref{eq:Grifft=H_nr}.
This means that whenever $[\alpha]$ can be represented by a class $\alpha\in H^i(U,\Z_\ell(i))$ on some open subset $U\subset X$ whose complement $R=X\setminus U$ has codimension $2i$,  the residue 
$$
\del \alpha\in H^{2i}_R(X,\Z_\ell(i))\cong \bigoplus_{x\in R_{(d_X-i)} }\Z_\ell [x]
$$
is a cycle $z$ that represents the class $[z]$.
Since $[z]$ is torsion, there is some $r\geq 1$ such that $\ell^r[z]=0\in A^i(X)_{\Z_\ell}$.
This implies that $\ell^r [\alpha]$ lifts to a class $[\beta]\in H^{2i-1}(X,\Z_\ell(i))$.
Then, $[\beta/\ell^r]$ gives rise to a class in $H^{2i-1}(X,\Q_\ell(i))$ and hence in $H^{2i-1}(X,\Q_\ell/\Z_\ell(i))$, and we have
$$
\lambda_{tr}^i([z]) = [\beta/\ell^r]\in H^{2i-1}(X,\Q_\ell/\Z_\ell(i))/N^{i-1}H^{2i-1}(X,\Q_\ell (i)) .
$$

By Proposition \ref{prop:Gamma-ast-compat-Chow}, the cycle $[\Gamma]_\ast [z]$ corresponds via \eqref{eq:Grifft=H_nr} to the class 
$$
[\Gamma]_\ast [\alpha] \in  H^{2i+2c+2d_X-1}_{i+c-d_X-2,\nr}(X,\Z_\ell(i+c+d_X)).
$$
Applying the moving lemma (Theorem \ref{thm:moving}) to $\Gamma$ or Corollary \ref{cor:inj+purity-thm-body}\eqref{item:thm:purity-body} to $[\alpha]$, we can assume that there is a representative  $\alpha\in H^i(U,\Z_\ell(i))$ of $[\alpha]$ as above such that $R\times Y$ meets $W:=\supp \Gamma$ in codimension at least $i+c$.
We then let $S=q(W\cap (R\times Y))$ and $V=Y\setminus S$ and find by Corollary \ref{cor:motivic-body}\eqref{item:thm:motivic-body:Gamma(W)}  that $[\Gamma]_\ast [\alpha] $ is represented by
$$
\Gamma(W)_\ast(\alpha)\in H^{2i+2c-2d_X-1}(V,\Z_\ell(i+c-d_X)) 
$$
from Lemma \ref{lem:Gamma-ast}.
By Corollary \ref{cor:motivic-body}\eqref{item:thm:motivic-body:restriction},  we find that $\ell^r\cdot \Gamma(W)_\ast(\alpha)=\Gamma(W)_\ast(\ell^r\cdot \alpha)$ extends to the class $\Gamma(W)_\ast(\beta)\in H^{2i+2c-2d_X-1}(Y,\Z_\ell(i+c-d_X))$.
Hence, by the description of $\lambda_{tr}^{i }$ given above, we find that 
$$
\lambda_{tr}^i([\Gamma]_\ast [z]) = [\Gamma(W)_\ast \beta/\ell^r]\in  \frac{H^{2i+2c-2d_X-1}(Y,\Q_\ell/\Z_\ell(i+c-d_X))}{N^{i+c-d_X-1}H^{2i+2c-2d_X-1}(X,\Q_\ell (i+c-d_X))} .
$$
The class $\Gamma(W)_\ast(\beta/\ell^r)\in H^{2i+2c-2d_X-1}(Y,\Q_\ell/\Z_\ell(i+c-d_X))$ from Lemma \ref{lem:Gamma-ast} agrees by construction with $[\Gamma]_\ast (\beta/\ell^r)$, \textit{i.e.}  with the image of $\beta/\ell^r$ via the action of the correspondence $\Gamma$.
This shows that the diagram in question is commutative, which concludes the proof of the corollary.
\end{proof}

 \section*{Acknowledgements}    
 I am grateful to Sergey Gorchinskiy for asking whether refined unramified cohomology is motivic  and to Kees Kok and Lin Zhou for pointing out a mistake in the previous version of Corollary \ref{cor:YxP^n}.
 Thanks to Andreas Krug for discussions and to the excellent referees for their careful reading and suggestions.

%%%%%%%%%%%%%%%%%%%%%
% References
%%%%%%%%%%%%%%%%%%%%%

\end{document}